\documentclass[11pt, a4paper]{amsart}
\usepackage{amsmath,amssymb,amsthm,amsfonts}
\usepackage{xfrac}
\usepackage{lmodern}
\usepackage{enumitem}
\usepackage{graphicx} 
\usepackage[T1]{fontenc}
\usepackage{tikz-cd}
\usepackage{pythontex}
\usepackage[maxbibnames=99,
backend=biber,
style=alphabetic,
doi=false,
url=false,giveninits=true
]{biblatex}
\usepackage{fullpage}
\usepackage{subfig}
\usepackage{pgfplots}
\usetikzlibrary{plotmarks}

\usepackage{mathrsfs}
\usepackage{bbm,color,mathtools,verbatim}
   
\usepackage[bookmarksopen=false,pdftex=true,breaklinks=true,%
      plainpages=false,%
      hyperindex=true,pdfstartview=FitH,colorlinks=true,%
      pdfpagelabels=true,colorlinks=true,linkcolor=blue,%
      citecolor=red,urlcolor=blue,hypertexnames=false%
      ]%
   {hyperref}
   
\newcommand{\abb}[5]{%
\setlength{\arraycolsep}{0.4ex}%
\begin{array}{rcccc}%
#1 &:\,& #2 & \,\,\longrightarrow\,\, & #3 \\[0.5ex]%
     & & #4 & \longmapsto & #5%
\end{array}%
}
   
\renewcommand{\vartheta}{\theta}

\newcommand{\N}{\mathbb{N}}
\newcommand{\Z}{\mathbb{Z}}
\newcommand{\R}{\mathbb{R}}

\newcommand{\x}{\mathbf{x}}
\newcommand{\Q}{\mathbb{Q}}
\renewcommand{\leq}{\leqslant}

\usepackage{faktor}

\usepackage{MnSymbol} 
\numberwithin{equation}{section}
\newtheorem{theorem}{Theorem}[section]
\newtheorem{lemma}[theorem]{Lemma}
\newtheorem{remark}[theorem]{Remark}

\newtheorem{corollary}[theorem]{Corollary}
\newtheorem{proposition}[theorem]{Proposition}

\newtheorem{example}[theorem]{Example}

    \newtheoremstyle{TheoremNum}
        {\topsep}{\topsep}              %%% space between body and thm
        {\itshape}                      %%% Thm body font
        {}                              %%% Indent amount (empty = no indent)
        {\bfseries}                     %%% Thm head font
        {.}                             %%% Punctuation after thm head
        { }                             %%% Space after thm head
        {\thmname{#1}\thmnote{ \bfseries #3}}%%% Thm head spec
    \theoremstyle{TheoremNum}
    \newtheorem{thmn}{Theorem}
  \newtheorem{corn}{Corollary}
    \newtheorem{propn}{Proposition}
\theoremstyle{definition}
\newtheorem{definition}[theorem]{Definition}

\DeclareMathOperator{\Tr}{tr}

\DeclareMathOperator{\inti}{int}
\DeclareMathOperator{\clo}{cl}
\DeclareMathOperator{\bd}{bd}
\DeclareMathOperator{\conv}{conv}

\DeclareMathOperator{\E}{\mathcal{E}}

\DeclareFontFamily{U}{mathx}{\hyphenchar\font45}
\DeclareFontShape{U}{mathx}{m}{n}{
      <5> <6> <7> <8> <9> <10>
      <10.95> <12> <14.4> <17.28> <20.74> <24.88>
      mathx10
      }{}
\DeclareSymbolFont{mathx}{U}{mathx}{m}{n}
\DeclareFontSubstitution{U}{mathx}{m}{n}
\DeclareMathAccent{\widecheck}{0}{mathx}{"71}

\definecolor{dblackcolor}{rgb}{0.0,0.0,0.0}
\definecolor{dbluecolor}{rgb}{0.01,0.02,0.7}
\definecolor{dgreencolor}{rgb}{0.2,0.4,0.0}
\definecolor{dgraycolor}{rgb}{0.30,0.3,0.30}

\begin{document}

\title{The Wonderful Geometry of the Vandermonde map}

\author{Jose Acevedo}
\address{Escuela de Matem\'aticas, Universidad Industrial de Santander, Bucaramanga, Colombia}
\email{jgacehab@correo.uis.edu.co}

\author{Grigoriy Blekherman}
\address{School of Mathematics, Georgia Institute of Technology, 686 Cherry Street Atlanta, GA 30332, USA}
\email{greg@math.gatech.edu}

\author{Sebastian Debus}
\address{Sebastian Debus, Technische Universität Chemnitz, Fakultät für Mathematik, 09107 Chemnitz, Germany}
\email{sebastian.debus@mathematik.tu-chemnitz.de}

\author{Cordian Riener}
\address{Department of Mathematics and Statistics, UiT - the Arctic University of Norway, 9037 Troms\o, Norway}
\email{cordian.riener@uit.no}

\thanks{The first and second author were partially supported by NSF grant DMS-1901950. The third and fourth author have been supported by European Union's Horizon 2020 research and innovation programme under the Marie Sk\l{}odowska-Curie grant agreement 813211 (POEMA) and the Troms\o~ Research foundation grant agreement 17matteCR. The third author was additionally supported by the Deutsche Forschungsgemeinschaft (DFG, German Research Foundation) – 314838170, GRK 2297 MathCoRe}


\begin{abstract}
We study the geometry of the image of the nonnegative orthant under the power-sum and elementary symmetric polynomials maps. After analyzing the image in a finite number of variables, we concentrate on the limit as the number of variables approaches infinity. We explain how the geometry of the limit plays a crucial role in undecidability results in nonnegativity of symmetric polynomials, deciding validity of trace inequalities in linear algebra, and extremal combinatorics (recently observed by Blekherman, Raymond, and F. Wei \cite{blekherman2022undecidability}). We verify the experimental observation that the image has the combinatorial geometry of a cyclic polytope, as noted by Melánová, Sturmfels, and  Winter \cite{melanova2022recovery}, and generalize results of Choi,  Lam, and Reznick \cite{choi1987even} on nonnegative even symmetric polynomials. We also show that undecidability does not hold for the normalized power sum map. 
\end{abstract}
\maketitle
\markboth{J.~Acevedo, G.~Blekherman, S.~Debus, and C.~Riener}{Symmetric forms}

\section{Introduction} Our main subject is the so-called Vandermonde map which appears quite naturally in various contexts, thus providing connections between different mathematical domains. Our interest in the Vandermonde map is motivated by the following  problem. Suppose that we are given a polynomial expression in traces of powers of real symmetric matrices, such as \[2[\operatorname{tr} (A^2)]^2\operatorname{tr} (B^6) - \operatorname{tr} (A^4) [\operatorname{tr} (B^2)]^3,\] is there an algorithm to decide whether this expression is nonnegative for all real symmetric matrices $A$, $B$ of all sizes? What happens if we replace trace by normalized trace $\widetilde{\operatorname{tr}} (A)=\frac{\operatorname{tr}(A)}{n}$, where $n$ is the size of the matrix?

One of our main results is that the first (unnormalized) problem is \emph{undecidable}, while the second one is \emph{decidable}. The key to the hardness of the unnormalized problem is the fascinating geometry of \emph{the image of the probability simplex under the Vandermonde map}. As we explain below, some geometric properties of this set were observed in different areas of mathematics making it an important and beautiful object to study.

For any $n\times n$ matrix $A$ recall that $\operatorname{tr} (A^d)=\lambda_1^d+\dots+\lambda_n^d$, where $\lambda_i$ are the eigenvalues of $A$. We use $p_d$ to denote the $d$-th power sum polynomial: $p_d(x)=x_1^d+\dots+x_n^d$. We see that testing whether $2[\operatorname{tr} (A^2)]^2\operatorname{tr} (B^6) - \operatorname{tr} (A^4) [\operatorname{tr} (B^2)]^3$ is nonnegative on all real symmetric matrices of all sizes is equivalent to understanding whether $2p_2^2(x)p_6(y)-p_4(x)p_2^3(y)$ is nonnegative on all real vectors $x$ and $y$ of any dimension. 
Define the \textit{$d$-th Vandermonde map $\nu_{n,d}$} by sending a point in $\mathbb{R}^n$ to its image under the first $d$ power sums:
$$\nu_{n,d}(x)=(p_1(x),\dots,p_{d}(x)). $$  
Let $\Delta_{n-1}$ be the probability simplex in $\mathbb{R}^n$: $\Delta_{n-1}$ consists of all vectors with nonnegative coordinates with the sum of coordinates equal to $1$. We call the image $\nu_{n,d}(\Delta_{n-1})$ of the probability simplex under the Vandermonde map the \textit{$(n,d)$-Vandermonde cell} and denote it by $\Pi_{n,d}$. Observe that the first coordinate of $\Pi_{n,d}$ is identically $1$, and so we may project it out, and see $\Pi_{n,d}$ as the subset of $\mathbb{R}^{d-1}$, which is the image of $\Delta_{n-1}$ under $(p_2,\dots, p_d)$.

Since $2p_2^2(x)p_6(y)-p_4(x)p_2^3(y)$ is an even polynomial homogeneous in both $x$ and $y$, using the substitution $a_k=\frac{p_{2k}(x)}{p_2^k(x)}$ and $b_k=\frac{p_{2k}(y)}{p_2^k(y)}$, we see that deciding whether it is nonnegative for all $x,y \in \mathbb{R}^n$ is equivalent to deciding whether the polynomial $2b_3-a_2$ is nonnegative on the product $\Pi_{n,2}\times \Pi_{n,3}$. 

We reach two important conclusions: 
first, we are interested in nonnegativity of polynomials on (products of) Vandermonde cells $\Pi_{n,d}$, and second, to consider matrices of all sizes we need to take the \emph{limit of the Vandermonde cell $\Pi_{n,d}$ as $n$ goes to infinity}.

The Vandermonde cell $\Pi_{n,d}$ is a compact subset of $\mathbb{R}^{d-1}$, and our first main result is that $\Pi_{n,d}$ has \emph{the combinatorial structure of a cyclic polytope}, verifying an experimental observation of \cite{melanova2022recovery}. 

For a fixed $d$ the sets $\Pi_{n,d}$ form an increasing sequence of sets in $\mathbb{R}^{d-1}$. Let $\Pi_d$ be the closure of the union of $\Pi_{n,d}$.
We show that the set $\Pi_d$ has the combinatorial structure of an \textit{infinite cyclic polytope}, and that $\Pi_d$ is not semialgebraic for all $d \geq 3$. The sets $\Pi_{n,3}$, for $n=3,4,5$, are depicted in Figure \ref{fig: N33 and N53}. 
The reduction needed to show undecidability of the unnormalized trace problem is borrowed from the one used by Hatami and Norin in \cite{hatami2011undecidability} in the context of homomorphism density inequalities in graph theory. The set used by Hatami and Norin is essentially a linear transformation of the set $\Pi_3$, and the reduction is based on the geometry of $\Pi_3$. In particular this shows that deciding validity of matrix power trace inequalities is already undecidable if we only use second, fourth and sixth matrix powers, and we need at most 9 matrix variables for the problem to become undecidable. We note that the geometry of $\Pi_3$ was also used directly by Blekherman, Raymond and Wei \cite{blekherman2022undecidability} to show undecidability of homomorphism density inequalities with arbitrary edge weights.

We also consider the image of $\Delta_{n-1}$ under elementary symmetric polynomials. Then the boundary also has the combinatorial structure of a cyclic polytope. This follows from Newton's identities. We write $E_{n,d} := (e_2,\ldots,e_d)(\Delta_{n-1})$ and denote the limit image by $E_d$. We show that the convex hull of $E_{n,d}$ is a convex polytope with combinatorial type of a cyclic polytope. This allows us to geometrically explain and slightly generalize the result of Choi, Lam and Reznick \cite{choi1987even} on test sets for nonnegativity of even symmetric sextics. We note that the convex hull result can be traced to the work of Bollobás in extremal graph theory \cite{bollobas1976relations}.

Testing nonnegativity of univariate normalized trace polynomials was considered  by Klep, Pascoe and Vol{\v{c}}i{\v{c}} \cite{klep2021positive} where the authors proved a Positivstellensatz in the univariate case. Geometrically, such normalized trace polynomials correspond to power means. Nonnegativity of polynomials in power means was investigated by Blekherman and Riener in degree $4$ \cite{blekherman2021symmetric} and more generally by Acevedo and Blekherman \cite{acevedo2024power}. We briefly illustrate the connection with the Vandermonde map. Decidability of the normalized trace problem follows quickly from the work in \cite{blekherman2021symmetric}. As before we can consider the image of the normalized Vandermonde map by fixing the degree $d$ and taking the (closure of the) limit as $n$ goes to infinity. As explained in \cite{acevedo2024power} the geometry of the normalized limit is different than in the unnormalized setting. 
For instance, the limit of the image of the normalized Vandermonde map applied to the nonnegative orthant $\mathbb{R}^n_{\geq 0}$ 
corresponds to the set of first $d$ moments of probability measures supported on $\mathbb{R}_{\geq 0}$, and it is well-known that this set can be described by linear matrix inequalities \cite{MR3729411}. In particular, the limit is semialgebraic for each $d$.

\subsection{Previous Work and Main Results in Detail}
The Vandermonde map has been studied from several different perspectives. Originating from the question of understanding univariate hyperbolic polynomials, Arnold, Givental and Kostov investigated the sets $(e_1,\ldots,e_d)(\R^n)$ \cite{arnol1986hyperbolic,givental1987moments,kostov1989geometric,kostov1999hyperbolicity}. Detailed expositions providing many important details can be found in a paper by Meguerditchian \cite{meguerditchian1992theorem} and Rosenblum's Ph.D. thesis \cite{rosenblum2023topology}. Kostov investigated the limit of the images of $\R^n$ for $d=4$. The authors observed that one can also allow positive \textit{weights} in the definition of the Vandermonde map. Their description of the boundary of the image of the Vandermonde map and of fibers generalizes to the map \[\R^n \longrightarrow \R^d, x \mapsto (a_1x_1+\ldots+a_nx_n,\ldots,a_1x_1^d+\ldots+a_nx_n^d)\]for any positive weights $a_1,\ldots,a_n > 0$. This is mainly due to the fact that Jacobians of the weighted and unweighted maps differ only by positive constant multiples.

The restriction of the Vandermonde map to the nonnegative orthant was investigated by Ursell \cite{ursell1959inequalities}. The paper contains several important results some of which we reprove. Ursell observed that the geometry of the Vandermonde map restricted to the nonnegative orthant generalizes further to arbitrary real positive exponents, i.e. to maps \[\R_{\geq 0}^n \longrightarrow \R^d, x \mapsto (x_1^{\alpha_1}+\ldots+x_n^{\alpha_1},\ldots,x_1^{\alpha_d}+\ldots+x_n^{\alpha_d})\] for which $0 < \alpha_1 < \ldots < \alpha_d$. Ursell's original motivation came from studying valid inequalities in $\ell_p$-norms.

 Recently, there has been an interest in describing fibers and the image of the Vandermonde map using computational algebraic geometry \cite{bik2021semi,melanova2022recovery}. Bik, Czapliński and Wageringel derived a semialgebraic description of $\nu_{n,3}([-1,1]^n)$ for all $n \geq 3$ which has applications in the study of $L$-functions and their zeros. Melánová, Sturmfels and Winter explored fibers and the image  of the Vandermonde map over the complex numbers and real numbers. \medskip

Our first theorem is a result initially found by Ursell \cite{ursell1959inequalities}, who proved the statement for arbitrary real exponents. We provide a different proof for integer exponents by adapting the techniques in Arnold's, Givental's and Kostov's work. We formulate here the result for the map $\nu_{n,d}$ although we later prove it for a sequence of increasing positive integer exponents whose first entry is $1$. \medskip

\begin{thmn}[\ref{thm:1}] 
For integers $n \geq d$ the set $\bd \Pi_{n,d}$ is the closure of the set of evaluations of $\nu_{n,d}$ at all points in $\Delta_{n-1}$ of the following two types: 
\begin{itemize}
 \item[\emph{(1)}] $(\underbrace{0,\ldots,0}_{m_0},\underbrace{x_1}_{m_1},\underbrace{x_2,\ldots,x_2}_{m_2},\ldots,\underbrace{x_{d-1},\ldots,x_{d-1}}_{m_{d-1}})$ with $m_0 \geq 0$, $m_{2k-1} = 1$ and $m_{2k} \geq 1$ for all $k$,
 \item[\emph{(2)}] $(\underbrace{x_1,\ldots,x_1}_{m_1},\underbrace{x_2}_{m_2},\ldots,\underbrace{x_{d-1},\ldots,x_{d-1}}_{m_{d-1}})$ with $m_{2k} = 1$ and $m_{2k-1} \geq 1$ for all $k$,
\end{itemize} 
and $0 < x_1 < x_2 < \ldots < x_{d-1}$ in both types. Moreover, the pre-image of any point in $\bd \Pi_{n,d}$ is unique, up to permutation of coordinates.
\end{thmn} \medskip

We then investigate concretely the boundaries of $\Pi_{n,3}$ and the limit set $\Pi_3$ and derive consequences for all $d\geq 3$. \medskip

\begin{corn}[\ref{cor:not semialgebraic}]
The sets $\Pi_{d}$ and $E_d$ are not semialgebraic for all $d \geq 3$.
\end{corn} \medskip

 Let $\mu_d : \R \to \R^d, t \mapsto (t,t^2,\ldots,t^d)$ denote the $d$-dimensional moment curve and let $t_1 < \ldots < t_n$. For $n > d$ the \textit{cyclic polytope} $C(n,d)$ is the convex polytope with vertices $\mu_d(t_i)$ for $1 \leq i \leq n$. The combinatorial type of the cyclic polytope is independent of the chosen $n$ points on the moment curve. Cyclic polytopes are the polytopes with maximal $f$-vector among all convex polytopes of given dimension and number of vertices \cite{mcmullen1970maximum,stanley1975upper}. The facets of $C(n,d)$ are characterized by \textit{Gale's evenness condition} \cite{gale1963neighborly}. A subset $\{\mu_d(t_{i_1}),\ldots,\mu_d(t_{i_d})\}$ with $i_j < i_{j+1}$ for all $1 \leq j <d$ spans a facet if and only if any two elements in $\{t_1,\ldots,t_n\} \setminus \{t_{i_1},\ldots,t_{i_d}\}$ are separated by an even number of elements in $\{t_{i_1},\ldots,t_{i_d}\}$. 
 
 The set $\Pi_{n,d}$ is topologically a $(d-1)$-dimensional closed ball for $n\geq d$. Answering a question in \cite{melanova2022recovery} we prove that the set $\Pi_{n,d}$ has the combinatorial structure of a cyclic polytope, in the sense that the boundary is a gluing of patches where each patch is a \emph{curved simplex} and the vertices of the patches are characterized by Gale's evenness condition. A set $S$ is a curved simplex if it is the image of $\Delta_m$ under a homeomorphism $f$, such that $f$ is a diffeomorphism when restricted to the relative interior of any face of $\Delta_m$.\medskip

\begin{thmn} [\ref{thm:combinatorially cyclic polytope}]
The set $\Pi_{n,d}$ has the combinatorial structure of the cyclic polytope $C(n,d-1)$, i.e. there is a homeomorphism $ \bd C(n,d-1) \rightarrow \bd\Pi_{n,d} $ that is a diffeomorphism when restricted to the relative interior of any face of $\bd C(n,d-1)$.
\end{thmn} \medskip

We provide an explicit map $ \bd C(n,d-1) \rightarrow \bd\Pi_{n,d} $ in Section \ref{sec:comb prop}. \\
For $n \geq d$ it follows from Newton's identities 
\begin{align} \label{eq:Newton's identities}
    p_k & =(-1)^{k-1}ke_k+\sum _{i=1}^{k-1}(-1)^{k-1+i}e_{k-i}p_i \hspace{2cm}\text{for all } 1 \leq k \leq n
\end{align}
that the image of the Vandermonde map is diffeomorphic to the image of the first $d$ elementary symmetric polynomials, i.e.
$E_{n,d} \cong \Pi_{n,d}$ under a polynomial diffeomorphism. We show that the convex set $\mathcal{E}_{n,d} := \conv E_{n,d}$ has nice properties which $ \conv \Pi_{n,d}$ does not have. \medskip

\begin{thmn}[\ref{thm:Convex hull Image of Elementary}]
For $d\ge3$ the set $\mathcal{E}_{n,d}$ is the image of the cyclic polytope $C(n,d-1)$ under an invertible affine linear map, and it is the convex hull of the points $\left({k \choose 2}\frac{1}{k^2},\dots,{k\choose d}\frac{1}{k^d}\right)$ for $k\in[n]$.
\end{thmn} \medskip

Recall that any symmetric polynomial can be written as a polynomial expression in elementary symmetric polynomials or power sums. A polynomial is called \emph{copositive} if it is nonnegative on the nonnegative orthant. The vertices of $\mathcal{E}_{n,d}$ relate to copositivity of homogeneous symmetric polynomials in which only $e_1$ occurs nonlinearly, i.e. symmetric forms which can be written as $$f=c_1e_1^d+c_2e_1^{d-2}e_2+\ldots+c_me_d$$  for some $c_1,\ldots,c_m \in \R$. We call such forms \emph{hook-shaped polynomials}. \medskip

\begin{thmn}[\ref{thm:discrete test set in elementarys}]
Let $f(e_1,\ldots,e_d)$ be a hook-shaped polynomial in $n \geq d$ variables. Then $f$ is copositive if and only if  $f\left(1,{k \choose 2}\frac{1}{k^2},\ldots,{k \choose d} \frac{1}{k^d}\right) \geq 0$ for all $k \in[n]$.
\end{thmn} \medskip

For $d=3$ this test set was found by Choi, Lam and Reznick \cite{choi1987even} but formulated for even symmetric sextics expressed in power sums. \medskip

\begin{thmn}[\ref{thm:choi-lam-reznick}] \cite{choi1987even}
Let $f(p_2,p_4,p_6)=ap_2^3+bp_4p_2+cp_6$, where $a,b,c\in\R$, be an even symmetric form of degree $6$ in $n \geq 3$ variables. Then, $f$ is nonnegative if and only if $f\left(1,\frac{1}{k},\frac{1}{k^2}\right)$ is nonnegative for all $k \in [n]$.
\end{thmn} \medskip

Using Newton's identities \eqref{eq:Newton's identities} and the fact that we can restrict to $p_2 =1$ due to homogeneity, the test sets in elementary symmetric polynomials and power sums are equivalent for $d \leq 3$ due to the linear relation of those families of polynomials on the probability simplex. Surprisingly, the test set for copositivity in power sums does not generalize to any higher degree. \medskip

\begin{propn}[\ref{prop:not power sums}]
For all $n\ge d\ge4$ the set $\Pi_{n,d}$ is not contained in $\conv \left\{ \left( \frac{1}{k},\frac{1}{k^2},\cdots,\frac{1}{k^{d-1}} \right) : k \in [n] \right\}$. Moreover, $\Pi_d \not \subseteq \conv \left\{(0,\ldots,0), \left( \frac{1}{k},\frac{1}{k^2},\cdots,\frac{1}{k^{d-1}} \right) : k \in \N_{>0} \right\}$.   
\end{propn} \medskip

Finally, we prove undecidability of testing validity of inequalities of trace polynomials in all real symmetric matrices of all sizes. \medskip

\begin{thmn} [\ref{thm:undecidable traces}]
The following decision problem is undecidable.
\begin{itemize}
    \item[{\footnotesize Instance:}] A positive integer $k$ and a trace polynomial $f(X_1,\ldots,X_k)$.
    \item[{\footnotesize Question:}] Is $f(M_1,\ldots,M_k)$ nonnegative for all real symmetric matrices $M_1,\ldots,M_k$ of all sizes?
\end{itemize}
\end{thmn}\medskip

When we replace the usual trace by the \text{normalized trace}, i.e. $\frac{\Tr (A)}{n}$ for a symmetric matrix $A$ of size $n \times n$, the problem becomes decidable. \medskip

\begin{thmn}[\ref{thm:normalized decidable}]
 The following decision problem is decidable.
\begin{itemize}
    \item[{\footnotesize Instance:}] A positive integer $k$ and a normalized trace polynomial $f(X_1,\ldots,X_k)$.
    \item[{\footnotesize Question:}] Is $f(M_1,\ldots,M_k)$ nonnegative for all real symmetric matrices $M_1,\ldots,M_k $ of all sizes?
\end{itemize}     
\end{thmn} \medskip

\paragraph{\bf Acknowledgements:} We are grateful to the two anonymous referees for their insightful comments and constructive feedback, which significantly contributed to improving the clarity and quality of this paper.

\section{The Vandermonde map} \label{sec:Vandermonde map}

In this section we study the geometry of the boundary of the Vandermonde cell. We start with some definitions.
\begin{definition} \label{def:Vandermonde}
\begin{enumerate}[leftmargin=*]
\item For $n,k\in \mathbb{N}$, $n\geq k$ we write \[e_k := \sum_{I \subset [n], |I|=k}\prod_{i \in I}x_i\] for the \emph{$k$-th elementary symmetric polynomial} in $n$ variables.
\item For $a\in\R_{>0}$ we consider  \[p_{a}:=\sum_{i=1}^n x_i^a\]
the \emph{power sum function}, which for $a \in \N$ is called a \emph{power sum polynomial}.
\item Given a sequence of strictly increasing positive rational numbers  $\alpha = (\alpha_1,\ldots,\alpha_d) \in \Q_{>0}^d$ we consider the \emph{$\alpha$-Vandermonde map} in $n$ variables to be the function 
$$ \abb{\nu_{n,\alpha}}{\R_{\geq 0}^n}{\R^d}{x}{(p_{\alpha_1}(x),p_{\alpha_2}(x),\dots,p_{\alpha_d}(x))}.$$  
For $\alpha=(1,2,\dots,d)$ we write $\nu_{n,d}$ for $\nu_{n,\alpha}$.
\end{enumerate}
\end{definition}
In the sequel we will restrict our study of Vandermonde maps to the probability simplex and power sum polynomials. 

\begin{definition}
Let $\Delta_{n-1}:=\left\{x\in\R^{n}_{\geq 0}\,:\,x_1+\ldots+x_{n}=1\right\}$ and  $\alpha = (\alpha_1,\ldots,\alpha_{d-1})\in\Z^{d-1}_{>1}$ be a  strictly increasing sequence of integers larger than $1$. The  \emph{$(n,\alpha)$-Vandermonde cell} $\Pi_{n,\alpha}$ is the set $\nu_{n,\alpha}(\Delta_{n-1})$. 
For $\alpha = (2,\ldots,d)$ we also write $\Pi_{n,\alpha} = \Pi_{n,d}$. We denote by $\mathcal{W}_n:=\{x\in\R^n:0\leq x_1\le x_2\le\dots\le x_n\}$ the nonnegative Weyl chamber, and we write $W_{n-1}:=\mathcal{W}_n \cap \Delta_{n-1}$.
\end{definition}
Note that we have $\Pi_{n,\alpha}=\nu_{n,\alpha}(W_{n-1})$.

\medskip

In Subsection \ref{subsection boundary of the vandermonde cell} we investigate the boundary of the Vandermonde cell.  In Subsection \ref{subsection power sum and elementary} we parameterize the boundaries of the sets $\Pi_{n,3}$ and $\Pi_3$ and we show that the limit set $\Pi_d$ is not semialgebraic for all $d \geq 3$. In Subsection \ref{sec:boundary of sub-probability simplex image} we analyze the boundary of the image of the sub-probability simplex $\tilde{\Delta}_n := \{ x \in \R_{\geq 0}^n : \sum_{i=1}^n x_i \leq 1\}$ under the map $\nu_{n,\alpha}$.

\subsection{The boundary of the  Vandermonde cell} \label{subsection boundary of the vandermonde cell}

We start by defining points whose image under the Vandermonde map will give the boundary of the Vandermonde cell, up to closure.

\begin{definition}\label{def:type}
Let $x=(x_1,x_2,\dots,x_n) \in \mathbb{R}_{\geq 0}^n$ with $x_1\leq x_2\leq \dots \leq x_n$  be a point with $\ell$ distinct positive coordinates $0<y_1<\dots<y_\ell$. Let $m_i$ be the multiplicity with which $y_i$ occurs as an entry of $x$ for $1 \leq i \leq \ell$ and write $m_0$ for the multiplicity of $0$ in $x$. We call the vector $m=(m_0,\ldots,m_{\ell})\in \N^{\ell+1}$ the \emph{multiplicity vector} of $x$, and $\ell$ the \emph{multiplicity length} of $x$. Note that the multiplicity length of $x$ is equal to the number of distinct positive coordinates in $x$.

If $m_0\ge0$, $m_{2i-1}=1, m_{2i} \geq 1$ for all $i\geq 1$ we say that the point $x$ is of type $(1)$. 
Points of type (1) with multiplicity length $\ell$ have the form:
\[
(\underbrace{0,\ldots,0}_{m_0},\underbrace{y_1}_{m_1},\underbrace{y_2,\ldots,y_2}_{m_2},\ldots,\underbrace{y_{\ell},\ldots,y_{\ell}}_{m_{\ell}}).
\]
 If $m_0=0, m_{2i-1}\geq 1, m_{2i}=1$ for all $i\geq 1$ we say that $x$ is of type $(2)$. Points of type (2) with multiplicity length $\ell$ have the form:
 \[ (\underbrace{y_1,\ldots,y_1}_{m_1},\underbrace{y_2}_{m_2},\ldots,\underbrace{y_{\ell},\ldots,y_{\ell}}_{m_{\ell}}).\]
Note that some points (e.g. $(1,1,\ldots,1))$ are in the closure of both type (1) and (2) points. 

For points $x=(x_1,\dots,x_n)\in \mathbb{R}^n$ we define the multiplicity vector of $x$, and the multiplicity length of $x$ as the ones for $(|x_1|,|x_2|,\dots,|x_n|)$.
\end{definition}

Our aim is to prove the following characterization of points on the boundary of the Vandermonde cell.

\begin{theorem}\label{thm:1}
For integers $n \geq d\ge2$ and positive integer exponents $1<\alpha_1<\dots<\alpha_{d-1}$, the set $\bd \Pi_{n,\alpha}$ is the closure of the set of evaluations of $\nu_{n,\alpha}$ at all points of multiplicity length equal to $d-1$ and type (1) or (2) in $\Delta_{n-1}$. Moreover, the pre-image of any point in $\bd \Pi_{n,\alpha}$ is unique, up to permutation of coordinates.
\end{theorem}

To understand the boundary of the Vandermonde cell we consider the notion of  Vandermonde varieties, which are pre-images of points under the Vandermonde map. The study of such varieties goes back to the works of Arnold, Givental and Kostov \cite{arnol1986hyperbolic,givental1987moments,kostov1989geometric} who had considered these in their study of hyperbolic polynomials. These authors consider the Vandermonde map $\nu_{n,d}$ which uses the first $d$ integer exponents. Furthermore, Rosenblum's recent Ph.D. thesis \cite{rosenblum2023topology} transfers the work of Arnold, Givental and Kostov to the first $d$ even integer exponents.

  In this paper, we restrict the Vandermonde map to the nonnegative orthant, and allow for general integer exponents, but many of the previously developed techniques will apply with some modifications. 
   Since we restrict to $x\in\R^n_{\geq 0}$ and $x_i\mapsto x_i^2$ is bijective for $x_i\geq 0$ we can ``square'' all variables, i.e., we can replace $(1,\alpha_1,\ldots,\alpha_{k-1})$, by $(2,\beta_1,\dots,\beta_{k-1})=(2,2\alpha_1,\ldots,2\alpha_{k-1})$, and replace the probability simplex $\Delta_{n-1}$ with the unit sphere $\mathbb{S}^{n-1}$. This does not change the image of our map. 
   
   As this will be a convenient way to deal with the conditions $x_i\geq 0$ in the proofs we will assume this squaring of the variables in the remainder of this section.
  We fix $2\leq d\leq n$ and a vector of positive even integer exponents $\beta = (2,\beta_1,\ldots,\beta_{d-1})$ with $2 < \beta_1 < \ldots < \beta_{d-1}$. When using even integer exponents, which will be denoted by the letter $\beta$, we will assume that the domain of the map $\nu_{n,\beta}$ is $\mathbb{R}^n$.

\begin{definition}
For $2 \leq k \leq d$ and $c\in \mathbb{R}^{k}$ define the Vandermonde variety of $c$
\[V_k^\beta(c):=\nu_{n,(2,\beta_1,\ldots,\beta_{k-1})}^{-1}(c).\]
Note that here the Vandermonde variety is not restricted to the nonnegative orthant.
\end{definition}

 Some fundamental properties of these varieties had been shown already by the mentioned authors. We show that their proofs can be almost directly adapted to the setup presented above. To this end we follow the proofs presented in \cite{meguerditchian1992theorem} (see also \cite{rainer2004perturbation} Section 3).

 We begin with an observation for the tangent space of a Vandermonde variety.

\begin{lemma}\label{lem:d-1 points} Let $n \geq k$, $c \in \{1\}\times\R^{k-1}_{\geq 0}$ and $\beta=(\beta_1,\ldots,\beta_{k-1})$ be a vector of pairwise distinct even positive integer exponents with $\beta_i\geq 4$. Then a point $x \in V_{k}^{\beta}(c)$ is smooth if and only if the multiplicity length of $x$ is at least $k$.
\end{lemma}

\begin{proof} 
Note that all polynomials $p_2,p_{\beta_1},\ldots,p_{\beta_{k-1}}$ are even symmetric, i.e. invariant under permutation and sign switching. Thus, a point $x \in V_k^\beta(c)$ is smooth if and only if any point in the orbit $\{(\epsilon_1 x_{\sigma(1)},\ldots,\epsilon_n x_{\sigma(n)}) : \epsilon_1,\ldots,\epsilon_n \in \{\pm 1\}, \sigma : [n] \to [n] \text{ a permutation}\}$ is smooth. Hence, we can restrict to points in the nonnegative orthant.

The Jacobian matrix of the map $(p_2,p_{\beta_1},\ldots,p_{\beta_{k-1}})$ equals 

\[
\left( \begin{array}{cccc}
    2x_1 &  2x_2 & \ldots & 2x_n \\
  \beta_1 x_1^{\beta_1-1}  & \beta_1 x_2^{\beta_1-1} & \ldots & \beta_{1} x_n^{\beta_{1}-1} \\
    \vdots & \vdots & \ddots & \vdots \\
   \beta_{k-1} x_1^{\beta_{k-1}-1}  & \beta_{k-1} x_2^{\beta_{k-1}-1} & \ldots & \beta_{k-1} x_n^{\beta_{k-1}-1}
\end{array} \right)
\]

Clearly, for $x\in\R_{\geq 0}^n$ with multiplicity length less than $k$ this matrix will not be of full rank. 
Now suppose that the rows of the Jacobian matrix are linearly dependent, so there exist $\ell_1,\ldots,\ell_{k} \in \mathbb{R}$, not all zero, such that every column satisfies $\ell_1\cdot 2x_j+\ell_2\cdot \beta_1 x_j^{\beta_1-1}+\ldots+\ell_{k}\cdot \beta_{k-1}x_j^{\beta_{k-1}-1}=0$. Therefore, every coordinate $x_j$ is a root of the same univariate polynomial \[f(t)= b_1t+b_2t^{\beta_1-1}+\ldots+b_{k}t^{\beta_{k-1}-1}.\] However, by Descartes' rule of signs $f$ can have at most $k-1$ different roots in $\R_{> 0}$. Therefore, if $x$ has multiplicity length larger than $k-1$, then the Jacobian matrix has full rank. 
\end{proof}

We also note the following Remark which is immediate from the proof.
\begin{remark}\label{rem:rank}
Let $F_{\ell}$ be an $\ell$-dimensional face of $W_{n-1}$. 
Then, for every $x$ in the relative interior of $F_{\ell}$ we have that the rank of the Jacobian of the map $(p_2,p_{\beta_1},\ldots,p_{\beta_{k-1}})$ is $\min\{\ell,k\}$.
\end{remark}

Now, for $\beta=(2,\beta_1,\ldots,\beta_{d-1})$ we study the level sets of the power sum polynomials $p_2,\ldots,p_{\beta_{d-2}}$ in order to understand the boundary of the image. 

\begin{lemma} \label{lem:critical points}
 For generic $c\in\{1\}\times\R^{d-2}_{\geq 0}$, the critical points of $p_{\beta_{d-1}}$ on $V_{d-1}^\beta(c)$ are precisely the points with multiplicity length equal to $d-1$. 
\end{lemma}
\begin{proof}
For generic $c$ the Vandermonde variety $V_{d-1}^\beta(c)$ is smooth and by the Jacobian criterion (\cite[Thm 16.19]{eisenbud2013commutative}) $(n-d+1)$ equidimensional (or empty). Since the Vandermonde map is weighted homogeneous, this also holds for generic $c\in\{1\}\times\R^{d-2}_{\geq 0}$. Therefore by Lemma \ref{lem:d-1 points} every point $x\in V_{d-1}^\beta(c)$ will have multiplicity length at least $d-1$.

Since all the exponents are supposed to be even, every of the polynomials in $\{p_{2},\ldots, p_{\beta_{d-1}}\}$ is invariant by sign change, and further more, the set of critical points of $p_{\beta_{d-1}}$ is closed under sign changes as well. Thus we have that  
$z=(z_1,\dots,z_n) \in  V_{d-1}^\beta(c)$ is a critical point of $p_{\beta_{d-1}}$ if and only if $(|z_1|,|z_2|,\dots,|z_n|)$ is a critical point and we my restrict to points with nonnegative coordinates.
Now consider a critical point $z=(z_1,\ldots,z_n) \in \R^n_{\geq 0}$ of $p_{\beta_{d-1}}$ on $V_{d-1}^\beta(c)$. Since $V_{d-1}^\beta(c)$ is smooth there exist $\lambda_0,\ldots,\lambda_{d-2} \in \R$ such that all partial derivatives of the Lagrangian function  

\begin{equation}\label{eq:largange} 
L(X) := p_{\beta_{d-1}}(X)+\lambda_0(p_2(X)-1)+ \sum_{i=1}^{d-2}\lambda_i (p_{\beta_i}(X)-c_{i+1})
\end{equation} 
vanish at $z$. This yields
\[ 0=\nabla p_{\beta_{d-1}} (z)  +  \lambda_0\nabla p_2 (z)+\lambda_1 \nabla p_{\beta_1}(z) + \ldots + \lambda_{d-2} \nabla p_{\beta_{d-2}}(z). \]
Again noting that $\frac{\partial{p_{a}}}{\partial x_j} =ax_j^{a-1}$ we observe that there exists a univariate polynomial
\begin{equation}\label{eq:f}
    f(t) = \beta_{d-1}t^{\beta_{d-1}-1} +2\lambda_0t +  \lambda_1 \beta_1 t^{\beta_1-1} + \ldots + \lambda_{d-2} {\beta_{d-2}}t^{\beta_{d-2}-1}
\end{equation}
such  that  $f(z_i) = 0$ for all $1 \leq i \leq n$. However, by Descartes' rule of signs the polynomial $f$ can have at most $d-1$ positive roots. Since $V_{d-1}^\beta(c)$ is smooth, the point $z$   must have exactly $d-1$ distinct non-zero coordinates by Lemma \ref{lem:d-1 points}. Conversely, assume $z \in V_{d-1}^\beta(c)$ has precisely $d-1$ distinct positive coordinates which we arrange in the following $(d-1)\times d$ matrix
\[ 
A(z) := \left( 
\begin{array}{cccc}
    2z_1 & \beta_1z_1^{\beta_1-1} & \ldots & \beta_{d-1}z_1^{\beta_{d-1}-1}\\
    2z_2 & \beta_1z_2^{\beta_1-1} & \ldots &\beta_{d-1} z_2^{\beta_{d-1}-1}\\
   \vdots & \vdots & & \vdots \\
    2z_{d-1} &\beta_1 z_{d-1}^{\beta_1-1} & \ldots & \beta_{d-1}z_{d-1}^{\beta_{d-1}-1}
\end{array}
\right)
\]
The column rank of $A$ is at most $d-1$ and thus there exist $\tilde{\lambda}_0,\ldots,\tilde{\lambda}_{d-1}$ which give the linear dependence of the columns of $A(z)$. Note that $\tilde{\lambda}_{d-1}\neq 0$ as again by Descartes rule of signs  the first $d-1$ columns are linearly independent. Thus $\frac{1}{\tilde{\lambda}_{d-1}}(\tilde{\lambda}_0,\ldots,\tilde{\lambda}_{d-2})$ are Lagrange multipliers for $z$. 
\end{proof}

\begin{proposition}\label{prop2} 
Let $V_{d-1}^\beta(c)$ be a smooth Vandermonde variety and $x$ be a critical point of $p_{\beta_{d-1}}$ on $V_{d-1}^\beta(c)$. Let $m_i$ be the multiplicities given in Definition \ref{def:type}, and let $r_i = m_i-1$. Then, if $d$ is odd (resp. even), the Hessian of $p_{\beta_{d-1}}$ on $V_{d-1}^\beta(c)$ at the point $x$ is the sum of a negative (resp. positive) definite quadratic form on $\R^{a}$ and a positive (resp. negative) definite quadratic form on $\R^{b}$, where $a = \sum_{i < d, i \not \in 2 \N} r_i$ and $b = m_0 +\sum_{i < d, i \in 2\N} r_i$. In particular, $p_{\beta_{d-1}}$ is a Morse function on $V_{d-1}^\beta(c)$ of index $a$ for $d$ odd, and index $b$ for $d$ even.
\end{proposition}
\begin{proof}
As before, it suffices to analyze critical points with nonnegative coordinates, since all of the exponents $\beta_i$ are even. Let $\tilde{x}\in \mathbb{R}^n_{\geq 0}$ be a critical point.  
Up to permutation of the coordinates $\tilde{x}$ can be assumed to be of the form
$$ \tilde{x}= (\underbrace{0,\ldots,0}_{m_0},\underbrace{y_1,\ldots,y_1}_{r_1},\ldots,\underbrace{y_{d-1},\ldots,y_{d-1}}_{r_{d-1}},y_1,y_2,\ldots,y_{d-2},y_{d-1}),$$ with $y_i>0$, as
by Lemma \ref{lem:critical points} we can assume that $\tilde{x}$ has precisely $d-1$ distinct non-zero coordinates $0<y_1<\ldots<y_{d-1}$. From  Lemma \ref{lem:d-1 points} we see that the corresponding $d-1$ last columns of the associated Jacobian will be of full rank.  Now we consider the Lagrange function $L$ from  \eqref{eq:largange} as a function in a new system of local coordinates, for which we can choose the first $n-d+1$ coordinates. Choosing these coordinates and observing that  the  Hessian of the Lagrange function in \eqref{eq:largange} equals \begin{align}\label{eq:hessianL}
\frac{\partial^2 L}{\partial {x}_i \partial {x}_j}(\tilde{x}) = 0 \text{ for }i \neq j,\text{ and  }\frac{\partial^2 L}{\partial {x}_i \partial {x}_i}(\tilde{x}) = f'(\tilde{x}_i),
\end{align}
where $f$ is given by (\ref{eq:f}), we see that the Hessian form equals 
\begin{align*} 
d^2L_{\tilde{x}}(h_1,\ldots,h_{n-d+1})=(h_1^2+\ldots+h_{r_1}^2)f'(y_1)+\ldots+(h_{n-d-r_{d-1}+2}^2+\ldots+h_{n-d+1}^2)f'(y_{d-1}).
\end{align*}
Since $\tilde{x}$ is a critical point we know from the proof of Lemma \ref{lem:critical points} that every coordinate of $\tilde{x}$ satisfies the same univariate polynomial equation  $f(y_i)=0$. Observe that $f$ has at most $d$ terms and so, by Descartes' rule of signs, $y_1,\dots,y_{d-1}$ are simple roots of $f$, therefore $f$ has exactly $d$ terms, and one of them has degree one, so $0$ is also a simple root. By the mean value theorem, the roots of the derivative $f'(t)$ interlace the roots of $f$. Noting that the leading coefficient of $f$ is positive, we find that the function values of the derivative $f'$ at the roots of $f$ satisfy 
\begin{align}\label{eq:signs}
f'(y_{d-1}) > 0,~ f'(y_{d-2}) < 0,~ f'(y_{d-3}) >0, \ldots, (-1)^{d-1} f'(y_1) < 0,~ (-1)^{d-1} f'(0) > 0,
\end{align}
from which we deduce the first statement. Finally, since $f$ has no multiple roots, equation \eqref{eq:hessianL} shows that $p_{\beta_{d-1}}$ is a Morse function on $V_{d-1}^{\beta}(c)$, and from \eqref{eq:signs} we find its index is $r_{d-2}+r_{d-4}+\dots$, which is $a$ for $d$ odd and $b$ for $d$ even.
\end{proof}

Thus, we immediately obtain:
\begin{corollary} \label{cor:pap3 1}
Suppose that $V_{d-1}^\beta(c)$ is smooth and let $x$ be a critical point of $p_{\beta_{d-1}}$ on $V_{d-1}^\beta(c)$. Then the multiplicity length of $x$ is $d-1$ and
\begin{enumerate}[label=(\alph*)]
   \item For $d$ odd: $x$ is a local minimum if and only if $x$ is of type (1), and $x$ is a local maximum if and only if $x$ is of type (2).
    \item  For $d$ even: $x$ is a local minimum if and only if $x$ is of type (2), and $x$ is a local maximum if and only if $x$ is of type (1).
\end{enumerate}
Additionally, all the local extrema are strict in these cases.
\end{corollary}

\begin{remark}\label{rem:gen}
We note that results in Lemma \ref{lem:d-1 points}-Corollary \ref{cor:pap3 1} also hold for positive rational exponents $\gamma=(\gamma_1,\dots,\gamma_{d-1})$ with $1<\gamma_1<\dots<\gamma_{d-1}$. The denominators in the fractions can be removed by using a transformation $x\mapsto x^k$ for some large integer $k$. We also expect that the results hold for arbitrary real exponents since the key ingredients, such as  the rank of the Jacobian matrix, which is a generalized Vandermonde matrix whose rank only depends on the multiplicity length of the vector $x$, and Descartes' rule of signs,  hold for positive real exponents. However, doubling the exponents, which we used to study a globally defined map, does not have the same effect for positive real exponents. 
\end{remark}

We now can prove the following crucial lemma.

\begin{lemma}\label{lem:interval}
For all $c\in\R^{d-1}$ we have that \(\Breve{V}_{d-1}^\beta(c):=V_{d-1}^\beta(c)\cap \mathcal{W}_n\) is either empty or connected. 
\end{lemma}

Kostov’s Theorem 1.13 \cite{kostov1989geometric} has the equivalent statement, for the image of the first \( d \) power sum polynomials, possibly including negative coordinates. It is rooted in the works of  Kostov, Givental, and Arnold. We adapt the argument to our more general setting and show that the original proof applies to general power sums when restricting to the non-negative orthant. Note, that  Rosenblum's PhD thesis \cite{rosenblum2023topology}, in Chapter 5, Section 2  offers another reference for the case involving even power sums.
Following the original proofs, our argument builds on the following essential ingredients. The first of these is due to Givental. His result \cite[Lemma 2]{givental1987moments} will be applied in the situation of smooth Vandermonde varieties where the function \( p_{\beta_d} \) serves as a Morse function.  It can be used to reconstruct the topology of a level set near critical points. 
\begin{lemma}[Givental]\label{lemma:giv}
The reconstruction of the topology of a level set of a function \( f \) on \( \mathbb{R}_{\geq 0}^a \times \mathbb{R}_{\geq 0}^b \) near the critical point \( (0, 0) \) with non-degenerate Hessian \( F = Q_+ + Q_- \) is trivial if \( a, b > 0 \) and consists of the birth (or death) of a simplex otherwise.
\end{lemma}
With applying Proposition \ref{prop2} in mind, we have that if both $a, b > 0$ then the topology of the level set does not change when passing through a critical point. If either $a$ or $b$ is $0$ then the level set either gains a connected component, which is homeomorphic to a simplex, or a connected component of the level set disappears, without joining the other connected components. 
As Givental notes \cite{givental1987moments}  the lemma almost directly gives the result in the smooth case. We will expand his argumentation below for the convenience of the reader following the exposition of \cite{meguerditchian1992theorem}.

As a second central ingredient that will be used in the case of non-smooth varieties, we will employ  Kostov’s Lemma \cite[Lemma 2.6]{kostov1989geometric}, which guarantees connectedness of level sets when all but finitely many of them are connected.

\begin{lemma}[Kostov]\label{lemma:kostov}
Let \( \Omega \subset \mathbb{R}^n \) be a connected compact set and \( f : \Omega \to \mathbb{R} \) be continuous. If all but finitely many sets \( \Omega \cap \{ f = \text{const} \} \) are connected, then all of them are connected.
\end{lemma}
With these preparations, we now give the proof of Lemma \ref{lem:interval}.
\begin{proof}[Proof of Lemma \ref{lem:interval}]
For a fixed  $c=(1,c_1,\dots,c_{d-1})\in\R^{d}$ the Vandermonde variety \(V^{\beta}_{d-1}(c)\) is given as the intersection of $d$ level sets of power sum polynomials. The proof is done by induction on $d$. 
The base case $d=1$ holds since $\mathbb{S}^{n-1}\cap \mathcal{W}_n$ is connected. To make the induction step, we define  for $c=(1,c_1,\dots,c_{d-1})\in\R^{d}$ the truncated sequence $c'=(1,c_1,\ldots,c_{d-2})\in\R^{d-1}$. As the induction hypothesis we assume that \(\Breve{V}^{\beta}_{d-2}(c')\) is connected, and we show that all level sets of $p_{\beta_{d-1}}$ on  \(\Breve{V}^{\beta}_{d-2}(c')\) are connected. This implies that \(\Breve{V}^{\beta}_{d-1}(c)\) is connected as well. We  first consider the situation in which  \(V^{\beta}_{d-2}(c')\) is smooth. By Sard's Lemma, this is the case for almost all $c\in\R^d$.  We note that Proposition \ref{prop2} gives us: 
\begin{enumerate}
\item If $V^{\beta}_{d-2}(c')$  is smooth, then it intersects the walls of $\mathcal{W}_n$ transversely.
 \item The critical points of $p_{\beta_{d-1}}$ on $V^{\beta}_{d-2}(c)$  are exactly the isolated points of intersection of $V^{\beta}_{d-2}$ with the $d-2$ dimensional faces of $\mathcal{W}_n$.
\end{enumerate}

Now we sweep through $\Breve{V}^{\beta}_{d-2}(c')$ with level sets of $p_{\beta_{d-1}}$. Since $\Breve{V}^{\beta}_{d-2}(c')$ is compact, for sufficiently small (and sufficiently large) $a$ the level sets $p_{\beta_{d-1}}=a$ are empty. By general Morse theory, the topology of the level sets does not change between critical points of $p_{\beta_{d-1}}$. If we can apply Lemma \ref{lemma:giv}, then at all saddle points of $p_{\beta_{d-1}}$ the topology of the level sets also does not change; at local maxima or minima the topology of the level sets changes by gaining or losing a single connected component. However, if the level set $p_{\beta_{d-1}}=a$ has at least two connected components for some $a$, then one of the components will eventually disappear (without joining the other connected components) and this contradicts connectedness of $\Breve{V}^{\beta}_{d-2}(c')$, which we have by the induction assumption. Therefore, the level sets $p_{\beta_{d-1}}=a$ are connected for all $a$. As the level sets of $p_{\beta_{d-1}}$ give us Vandermonde varieties $\Breve{V}^{\beta}_{d-2}(c)$ this finishes the proof in the smooth case, provided that we can apply Lemma \ref{lemma:giv}. We now show that this can be done.
 
Let $z \in \Breve{V}_{d-2}^\beta(c)$ be a point with $d-2$ distinct positive coordinates and multiplicity vector $m=(m_0,m_1,\ldots,m_{d-2})\in \N^{d-1}$. Consider the function  
$$\phi_m: \mathbb{R}^n \longrightarrow \mathbb{R}^n, \, x \mapsto (x_{1},\ldots,x_{m_0},x_{m_0+2}-x_{m_0+1},\ldots,p_{\beta_1}(x),\ldots,p_{\beta_{d-2}}(x))$$ where we keep the first $m_0$ coordinates of $x$, then take consecutive differences based on the multiplicity vector $m$, and leave the remaining $d-2$ coordinates for the power sums $p_{\beta_1},\dots,p_{\beta_{d-2}}$. Note that any point in $\Breve{V}_{d-2}^\beta(c)$ with multiplicity vector $m$ is mapped to $(0,\ldots,0,c_1,\ldots,c_{d-2})$.
 We observe that $\phi_m$ is a local diffeomorphism at  $z$ with $ \phi_m(z) = (0, \dots, 0, c_1, \dots, c_{d-2})$, and $\phi_m$ sends a neighborhood of $z$  in $\Breve{V}^{\beta}_{d-2}(c')$  onto a neighborhood of  $(0, \dots, 0, c_1, \dots, c_{d-2})$  in  $\mathbb{R}^{a+b}_{\ge0} \times \{(c_1, \dots, c_{d-2})\}$.
With the help of this transformation, we can apply Lemma \ref{lemma:giv}.
It remains to examine the case when $\Breve{V}^{\beta}_{d-1}(c)$ is not smooth and thus the fiber of the map $\nu_{n,\beta}$ over $c$ contains a non-regular point. Here the proof by Kostov \cite{kostov1989geometric} can be followed directly: Assuming that  $V^{\beta}_{d-2}(c)$ is of dimension $n-k$ for some $k$. Consider an $\ell$-dimensional face $F$  of $\mathcal{W}_n$. From  Remark \ref{rem:rank} we find that the rank of the Jacobian of $(p_{2},p_{\beta_1},\ldots,p_{\beta_{d-1}})$ on the relative interior of $F$ is $\ell$ if $\ell <k$ and $d$ otherwise. Thus in the first case, the intersection is zero-dimensional and it is transversal in the second case. When the intersection is transversal, we have that $\Breve{V}^{\beta}_{d-1}(c)$ is connected. Indeed, by the proof in the smooth case above, we know that there is a sequence of $\hat{c}$ converging to $c$ for which $\Breve{V}^{\beta}_{d-1}(\hat{c})$ is connected and transversality allows us to conclude connectedness in the limit. In the first case, the intersection is reduced to a finite number of points and we use Lemma \ref{lemma:kostov} to conclude.
\end{proof}

Our next goal is to understand the fibers of the projection of $\Pi_{n,\alpha}$ on the first $d-2$ coordinates. For an increasing sequence of positive integers $\alpha=(1,\alpha_1,\dots,\alpha_{d-1})$  define $\alpha'=(1,\alpha_1,\dots,\alpha_{d-2})$. We now show the following.  
\begin{proposition} \label{prop: projection of Vandermonde cell}
Let $n \geq d \geq 2$ and consider the projection $\pi: \Pi_{n,\alpha}\rightarrow \Pi_{n,\alpha'}$ onto the first $d-2$ coordinates. Then the pre-image of any point on the boundary of $\Pi_{n,\alpha'}$ is a single point, and the pre-image of any point in the interior of $\Pi_{n,\alpha'}$ is a non-degenerate interval.
\end{proposition}
First we establish some preliminary lemmas.

\begin{lemma}\label{lem: two multiplicity vector sequences}
Let $n \geq d$, $z \in  W_{n-1}$ with multiplicity vector $(e_0,e_1,\ldots,e_k)$ be the limit of two sequences of points in $W_{n-1}$ with multiplicity vectors $m$/$m'$ of types (1)/(2) and of multiplicity length $d-1$. Then $z$ is also the limit of a sequence of points of type (1) or (2) of multiplicity length $d-2$.
\end{lemma}
\begin{proof}
Suppose that $z$ is the limit of two sequences of points $x^{(i)}$/$y^{(i)}$ with multiplicity vectors $m$/$m'$ of types (1)/(2) which have both multiplicity length $d-1$.
We proceed by a case distinction. 
First, we assume that $e_0  \geq 1$. Then since $y_1^{(i)}>0$ (as $y^{(i)}$ is of type (2)) we have that $z$ is also the limit of the sequence 
\[
\hat{x}^{(i)}:=(\underbrace{0,\ldots,0}_{\# = m_1'},y_2^{(i)},\ldots,y_{d-1}^{(i)})
\]
with a fixed multiplicity vector of type (1) and of multiplicity length $d-2$.

Second, we suppose that $e_0=0$, i.e. $z_1 > 0$. Let $k$ be the minimal positive integer with $e_k \geq 2$ (which exists since $n \geq d$). Since $z$ is the limit of the sequences $x^{(i)}$/$y^{(i)}$ we must have $m_i=m_i'=1$ for all $1  \leq i < k$. If $k$ is even then, since $m'_k=1$, we have that $z$ is also the limit of the sequence 
\[\tilde{x}^{(i)}:= (y_1^{(i)},\ldots,y_{k-1}^{(i)},\underbrace{y_k^{(i)},\ldots,y_k^{(i)}}_{\# = m_k'+m_{k+1}'},y_{k+2}^{(i)},\ldots,y_{d-1}^{(i)})\]
with multiplicity vector $(1,\ldots,1,m_k'+m_{k+1}',m_{k+2}',\ldots,m_{d-1}')$ of length $d-2$ and of type (1).   
On the other hand, if $k$ is odd then since $m_k=1$ we have that $z$ is also the limit of the sequence 
\[\tilde{y}^{(i)}:= (x_1^{(i)},\ldots,x_{k-1}^{(i)},\underbrace{x_k^{(i)},\ldots,x_k^{(i)}}_{\# = m_k+m_{k+1}},x_{k+2}^{(i)},\ldots,x_{d-1}^{(i)})\]
with multiplicity vector $(1,\ldots,1,m_k+m_{k+1},m_{k+2},\ldots,m_{d-1})$ of length $d-2$ and of type (2).   
\end{proof}

\begin{lemma} \label{lem:homeom}
Let $K \subset \mathcal{W}_d$ be a full dimensional compact polyhedron. For positive weights $w_1,\ldots,w_d > 0$ and real exponents $0 < \gamma_1 < \ldots < \gamma_d$ the map \[ \nu_{w,\gamma} : K \to \nu_{w,\gamma}(K),\, x \mapsto \left(\sum_{i=1}^d w_ix_i^{\gamma_1},\ldots,\sum_{i=1}^d w_ix_i^{\gamma_d}\right)\] is a homeomorphism and a diffeomorphism when restricted to the relative interior of any face of $K$.
\end{lemma}

\begin{proof}
First, we show that the map $\nu_{w,\gamma}$ is a homeomorphism. From \cite[Theorem~1]{mas1979homeomorphisms} it is equivalent to showing that for any $x \in K$ and any subspace $L \subset \R^d$, with $x \in L$, spanned by a face of $K$, the linear map that is the composition of the multiplication by the Jacobian matrix of $\nu_{w,\gamma}$ and the orthogonal projection onto $L$ must have a positive determinant at $x$.

The Jacobian matrix of $\nu_{w,\gamma}$ is (up to scaling of the rows) a generalized Vandermonde matrix with real exponents of the form $(x_i^{\gamma_j-1})_{1\leq i,j \leq d}$. It is known that generalized Vandermonde matrices are totally positive when the variables $0<x_1<\ldots < x_d$ are totally ordered and positive (see e.g. \cite[Example~0.1.4]{Fallat2011} for integer exponents and \cite[Section~4.1]{pinkus2010totally} for real exponents).
By the Gantmacher-Krein theorem \cite[Proposition 5.4]{pinkus2010totally} any totally positive matrix $A \in \R^{d \times d}$ has positive simple eigenvalues which implies $v^TAv>0$ for any non-zero vector $v \in \R^d$. Thus the same property holds for the projection of $A$ on any subspace of $\R^d$. 

Second, we show that $\nu_{w,\gamma}$ restricted to the relative interior of any face of $K$ is a diffeomorphism. 
The Jacobian determinant of $\nu_{w,\gamma}$ restricted to the relative interior of any face of $K$ is positive and thus non-singular (because the boundary of $\mathcal{W}_d$ is given by $x_i=0$ or $x_i=x_j$ and the Jacobian matrix of $\nu_{w,\gamma}$ is totally positive in the interior of $\mathcal{W}_d$). Hence, by the inverse function theorem, the map $\nu_{w,\gamma}$ is a local diffeomorphism, but since it is also a homeomorphism then it is a diffeomorphism.
\end{proof}

\begin{lemma}\label{lem: unique pre-image Simplex}
Any point in $\bd \Pi_{n,\alpha}$ has a unique pre-image, up to permutation, in $\Delta_{n-1}$.
\end{lemma}

\begin{proof}
We recall that $\Pi_{n,\alpha}=(p_{\alpha_1},\ldots,p_{\alpha_{d-1}})(\Delta_{n-1})=(p_{\beta_1},\ldots,p_{\beta_{d-1}})(\mathbb{S}^{n-1})$.
For any $c \in \Pi_{n,\alpha}$ the variety $V_{d-1}^\beta(c) \cap \mathcal{W}_n$ is connected by Lemma \ref{lem:interval}. In particular, this holds for points $c \in \bd \Pi_{n,\alpha}$. 
Moreover, if $c \in \bd \Pi_{n,\alpha}$ the set $V_{d-1}^\beta(c) \cap \R_{\geq 0}^n$ consists only of points with at most $d-1$ distinct positive coordinates. This is, since if $x \in V_{d-1}^\beta(c)$ has at least $d$ distinct positive coordinates, then $x$ is a smooth point in the Vandermonde variety by Lemma \ref{lem:d-1 points} and cannot be mapped to the boundary of $\Pi_{n,\alpha}$ under $\nu_{n,\beta}$ by the inverse function theorem.
However, there are only finitely many multiplicity vectors for points in $\R^n_{\geq 0}$ with multiplicity length $\leq d-1$, i.e. finitely many sequences $\omega$ where $\omega_i$ is the number of coordinates equal to the $i$-th smallest positive coordinate of a point and $\omega_0$ is the multiplicity of $0$. Any such $\omega$ defines a vector $(\omega_1,\ldots,\omega_l) \in \N^l$ of positive weights and $l \leq d-1$. 
By Lemma \ref{lem:homeom} the map $\nu_{(\omega_1,\ldots,\omega_l),(\alpha_1,\ldots,\alpha_l)}  $ with domain $\{x \in \R_{\geq 0}^l: \sum_{i=1}^l \omega_i x_i = 1\} \cap \mathcal{W}_l$ is one-to-one. Thus, for any of the finitely many multiplicity patterns there is at most one point in the pre-image of $c$ with this multiplicity pattern. However, since the pre-image of $c$ in $W_{n-1}$ is connected, this implies that there is a unique pre-image in $\Delta_{n-1}$ up to permutation. \end{proof} 

\begin{proof}[Proof of Proposition \ref{prop: projection of Vandermonde cell}]
We note that if $\bar{c} \in \bd \Pi_{n,\alpha'}$ then the pre-image of $\bar{c}$ under $\nu_{n,\alpha'}$ in $W_{n-1}$ is unique by Lemma \ref{lem: unique pre-image Simplex}. This shows the pre-image of $\bar{c}$ under $\pi$ must be unique.

Now, let $\bar{c} \in \Pi_{n,\alpha'}$ be a point with a unique pre-image in $\Pi_{n,\alpha}$ under projection. We want to show that $\bar{c} \in \bd \Pi_{n,\alpha'}$. We suppose that $\bar{c} \in \inti \Pi_{n,\alpha'}$ and argue by contradiction. 
Since $\bar{c}$ lies in the interior there is a sequence $c_i \in \inti \Pi_{n,\alpha'}$ of generic points with $c_i \to \bar{c}$ for $i \to \infty$. For each $i$ the pre-image of $c_i$ under the projection $\pi$ is a non-degenerate interval since the power sums $p_1, p_{\alpha_1},\ldots,p_{\alpha_{d-1}}$ are algebraically independent. The end points of the interval have unique pre-images $x^{(i)}$ and $y^{(i)}$ under $\nu_{n,\alpha}$ in $W_{n-1}$ with mulitplicity length $d-1$ by Lemmas \ref{lem:critical points} and \ref{lem: unique pre-image Simplex}, where, by Corollary \ref{cor:pap3 1}, we can suppose that $x^{(i)}$ has multiplicity type (1) and $y^{(i)}$ has multiplicity type (2) and $\lim_{i \to \infty} \nu_{n,\alpha'}(x^{(i)})=\lim_{i\to\infty}\nu_{n,\alpha'}(y^{(i)})=\bar{c}$. Since there are only finitely many multiplicity vectors of type (1) and (2) we can without loss of generality assume that all $x^{(i)}/y^{(i)}$ have the same multiplicity vector. However, by Lemma \ref{lem: two multiplicity vector sequences} we have that $\bar{c}$ is then also the image under $\nu_{n,\alpha'}$ of a point in the closure of the multiplicity vectors of type (1) or (2) of length $d-2$. Using Corollary \ref{cor:pap3 1} this shows that $\bar{c} \in \bd \Pi_{n,\alpha'}$ which is a contradiction. 
\end{proof}

\begin{remark}  \label{rem:int} 
We note that the above arguments also show that the Vandermonde cell $\Pi_{n,\alpha}$ is the topological closure of its interior. We can use Lemma \ref{lem:homeom} and then note that the property of being the closure of the interior is preserved under linear projections.
\end{remark}

We are now in position to give a proof of Theorem \ref{thm:1}. 
\begin{proof}[Proof of Theorem \ref{thm:1}] 
    Suppose $\alpha \in \Z_{\geq 2}^{d-1}$ is an integer exponent vector with $\alpha_1 < \alpha_2 < \ldots < \alpha_d$. 
Any point of type (1) or (2) with multiplicity length $d-1$ is indeed mapped to the boundary of $\Pi_{n,\alpha}$ by Corollary \ref{cor:pap3 1}.
The closure of the image of such points is also mapped to the boundary by continuity.
Now, we assume that $\overline{c} := (p_{\alpha_1}(x),\ldots,p_{\alpha_{d-1}}(x))$ is contained in the boundary of $\Pi_{n,\alpha}$. Since $\Pi_{n,\alpha}$ is the closure of its interior by Remark \ref{rem:int}, we can find a sequence of sufficiently general points $c_i$ from the interior of $\Pi_{n,\alpha}$ converging to $\overline{c}$, such that the Vandermonde varieties $V_{d-1}^{\alpha}(c_i)$ are smooth. Consider the projection $\pi: \Pi_{n,\alpha}\to \Pi_{n,\alpha'}$ of $\Pi_{n,\alpha}$ onto the first $d-2$ coordinates as in Lemma \ref{lem:interval}. The pre-image of each $c_i$ is a non-degenerate interval and $p_{\alpha_{d-1}}$ is either minimized or maximized at the endpoints of the interval. We apply Corollary \ref{cor:pap3 1} to see that the pre-images of the endpoints of the interval must be points of type (1) or (2). If $\pi(\overline{c})$ lies in the interior of $\Pi_{n,\alpha'}$, then its pre-image under $\pi$ is also an interval, and $\bar{c}$ is an endpoint of this interval. Now the endpoints of intervals for $c_i$'s must converge to the endpoints of the interval for $\bar{c}$ by continuity. Therefore we realized $\overline{c}$ as the limit of points of type (1) or (2).

If $\pi(\bar{c})$ lies on the boundary of $\Pi_{n,\alpha'}$ then its pre-image under $\pi$ is a single point $\bar{s}$. It follows that the endpoints of the intervals for $c_i$ converge to $\bar{s}$ by continuity. Therefore we realized $\overline{c}$ as the limit of points of type (1) or (2).

\end{proof}

\begin{remark}
The description of the boundary of the $(n,\alpha)$-Vandermonde cell in Theorem \ref{thm:1} transfers to $\alpha$-Vandermonde maps with positive weight vector $w \in \R_{>0}^n$, i.e. to maps $$x \mapsto (w_1x_1^{\alpha_1}+\ldots+ w_nx_n^{\alpha_1},\ldots,w_1x_1^{\alpha_{d-1}}+\ldots+ w_nx_n^{\alpha_{d-1}}).$$
This holds since the columns of the Jacobian matrix of a weighted $\alpha$-Vandermonde map are positive scalar multiples of the columns of the Jacobian matrix of the non-weighted $\alpha$-Vandermonde map.
\end{remark}

\subsection{Boundary of the Vandermonde cells $\Pi_{n,3}$ and $\Pi_3$.}\label{subsection power sum and elementary}
In this subsection, we investigate parametrizations of $\bd\Pi_{n,3}$ and $\bd\Pi_3$. 

\begin{figure}[h!]%
    \centering
    \subfloat
    {{\includegraphics[width=15cm]{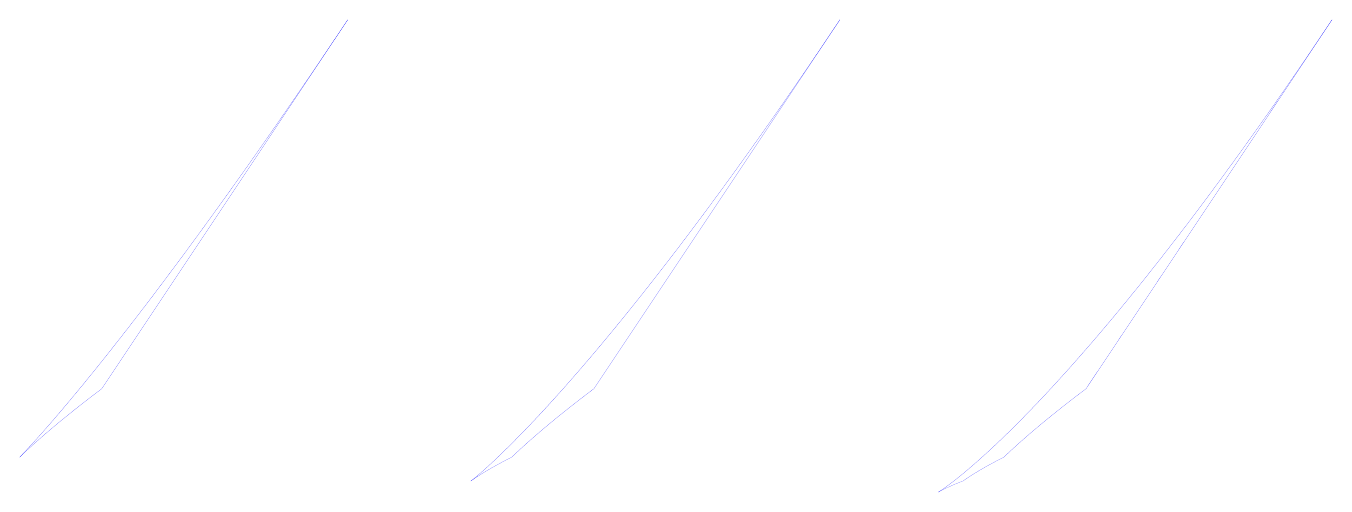} }}%
    \caption{The sets $\bd \Pi_{n,3}$ for $n=3,4,5$}%
    \label{fig: N33 and N53}%
\end{figure}

\begin{theorem} \label{thm:Parametrization Nnd}
For $n \geq 3$, a parametrization of
$\bd \Pi_{n,3}$ is given by the following $n$ arcs. The upper part of the boundary is parametrized by the arc 
\begin{align}\label{eq:upper part}
   U_{n-1}:=\left\{((n-1)t^2+(1-(n-1)t)^2,(n-1)t^3+(1-(n-1)t)^3): 0\le t\le\frac1n\right\}
\end{align}
while the lower part is parameterized by the $n-1$ arcs
\begin{align}\label{eq:lower part}
      L_k:=\left\{(kt^2+(1-kt)^2,kt^3+(1-kt)^3):\frac1{k+1}\le t\le\frac1k\right\}
\end{align}
for $k=1,\dots,n-1$. The endpoints of the arc $U_{n-1}$ are $(\frac1{n},\frac1{n^2})$ and $(1,1)$, while the endpoints of the arc $L_k$ are $(\frac1{k+1},\frac1{(k+1)^2})$ and $(\frac1k,\frac1{k^2})$ for each $k=1,\dots,n-1$. 
\end{theorem}
\begin{proof}
By Theorem \ref{thm:1} the boundary of $\Pi_{n,3}$ consists of the closure of the set of all points $(p_2,p_3)$ obtained by evaluating at the points $(0,\ldots,0,x_1,\ldots,x_1,x_2,\ldots,x_2)\in\Delta_{n-1}$, with $0<x_1<x_2$, of type (1) or (2). Moreover, by Corollary \ref{cor:pap3 1}, and since $\Pi_{n,3}\subset\R^2$, the lower boundary comes from points of type (1), while the upper boundary comes from points of type (2).
\medskip

The upper boundary then comes from points of the form $(a,\ldots,a,b)\in\Delta_{n-1}$ with $0\le a\le b$, so it can be parametrized by setting $a=t$ and $b=1-(n-1)t$ with $0\le t\le 1-(n-1)t$, and we obtain (\ref{eq:upper part}).
\medskip

The lower boundary comes from points of type (1). Observe there are $n-1$ types of points of type (1),
$$(\underbrace{0,\ldots,0}_{n-k-1},a,\underbrace{b,\ldots,b}_{k})$$
for $k=1,\dots,n-1$. Setting $a=1-kt$ and $b=t$ with $0\le 1-kt\le t$ we obtain (\ref{eq:lower part}).
\end{proof}

As a corollary, we obtain parametrizations of the boundary of the sets $E_{n,3}$ via Newton's identities. See Figure \ref{fig: E33 and E53} for a visualisation of these boundaries. In this case the map $(p_2,p_3)\to(e_2,e_3)$ is affine linear: $e_2=\frac12-\frac12p_2$ and $e_3=\frac16-\frac12p_2+\frac13p_3$ subject to $p_1=1$.

\begin{figure}[h!]%
    \centering
    \subfloat%
    {{\includegraphics[width=15cm]{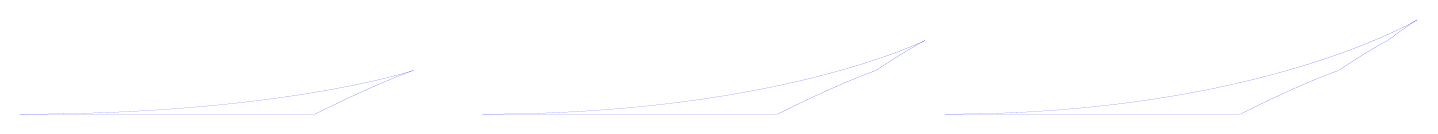} }}%
    \caption{The sets $\bd E_{n,3}$ for $n=3,4,5$}%
    \label{fig: E33 and E53}%
\end{figure}

In the transition from $\Pi_{n,3}$ to $\Pi_{n+1,3}$ in Theorem \ref{thm:Parametrization Nnd} the arc describing the upper boundary grows and converges. Its limit has the parametrization $(t,t^{3/2})$, $0\le t\le1$. Moreover, any point on the lower boundary of $\Pi_{n,3}$ remains on the lower boundary of $\Pi_{n+1,3}$, only the arc $L_n$ is added.

\begin{figure}[h!]%
    \centering
    {{\includegraphics[width=6cm]{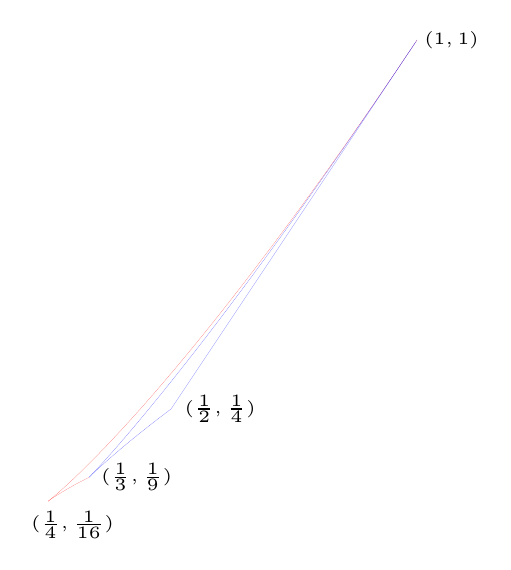}}}%
    \caption{Transition from $\Pi_{3,3}$ to $\Pi_{4,3}$}%
    \label{fig:4to5}%
\end{figure}

\begin{corollary} \label{cor:boundaryN3}
The boundary of the set $\Pi_{3} $ equals
$\left\{(t,t^{3/2}) : 0 \leq t \leq 1\right\} \cup \bigcup_{k\ge1} L_k$. 
\end{corollary}
\begin{proof}
By Theorem \ref{thm:Parametrization Nnd}, the lower boundary of $\Pi_{n,3}$ is the union of the arcs $L_1,\dots,L_{n-1}$. Hence, since $\Pi_3=\clo\bigcup_{n\ge3}\Pi_{n,3}$, the lower boundary of $\Pi_3$ is $\clo\bigcup_{k\ge1}L_k=(0,0)\cup\bigcup_{k\ge1}L_k$.
\medskip

    By Theorem \ref{thm:1} and Corollary \ref{cor:pap3 1} the upper boundary $U_{n-1}$ of $\Pi_{n,3}$ lies below the upper boundary $U_n$ of $\Pi_{n+1,3}$. Observe that, by setting $u=(n+1)t$, $$U_n=\left\{\left(\frac{n}{(n+1)^2}u^2+\left(1-\frac{n}{n+1}u\right)^2,\frac{n}{(n+1)^3}u^3+\left(1-\frac{n}{n+1}u\right)^3\right):0\le u\le1\right\}$$
    and so the sequence of sets $(U_n)_{n\ge2}$ converges to the set $$\left\{(1-u)^2,(1-u)^3):0\le u\le1\right\}$$
    as $n$ goes to infinity. This set can be parametrized as $\left\{(t,t^{3/2}):0\le t\le1\right\}$.
\end{proof}

\begin{remark}\label{rem:arcs pi3}
For each $k=1,\dots,n-1$ the functions $f_k(t):=kt^2+(1-kt)^2$ and $g_k(t):=kt^3+(1-kt)^3$ are increasing on $\left[\frac1{k+1},\frac1k\right]$. Moreover, $f_k(\frac1{k+1})=\frac1{k+1}$, $f_k(\frac1k)=\frac1k$, $g_k(\frac1{k+1})=\frac1{(k+1)^2}$, $g_k(\frac1k)=\frac1{k^2}$. Hence, restricting to the lower boundary of $\Pi_{n,3}$, only consecutive arcs in the list $L_1,L_2,\dots,L_{n-1}$ intersect. For each $k=2,\dots,n-1$ the arcs $L_{k-1}$ and $L_k$ intersect at the single point $(\frac1k,\frac1{k^2})$. Also, the upper boundary $U_{n-1}$ intersects only the arcs $L_1$, at $(1,1)$, and $L_{n-1}$, at $(\frac1n,\frac1{n^2})$.
\medskip

The points $(\frac1k,\frac1{k^2})$, for $k=1,\dots,n$, are singular points of $\bd\Pi_{n,3}$. We check this by computing the slopes at these points with respect to the arcs to which they belong. For instance,
\begin{align*}
    L_k'(t)&=(2kt-2k(1-kt),3kt^2-3k(1-kt)^2)\\
    &=(2k(t-(1-kt)),3k(t-(1-kt))(t+(1-kt)))\\
    &=k(t-(1-kt))\cdot(2,3(t+(1-kt)))
\end{align*}
and so the slope of the tangent line at $L_k(t)$ is $\frac32(1-(k-1)t)$, which is a decreasing function positive on $\left[\frac1{k+1},\frac1k\right]$. So along each arc the slope is positive and decreasing, thus each arc of the lower boundary is concave. At the left and right endpoints of $L_k(t)$ we obtain slopes of $\frac3{k+1}$ and $\frac3{2k}$ respectively. For $k=2,\dots,n-1$ we thus obtain different slopes at $(\frac1k,\frac1{k^2})$: $\frac3{2k}$ as a right endpoint of $L_k$, and $\frac3k$ as a left endpoint of $L_{k-1}$. However, for the points $(1,1)$ and $(\frac1n,\frac1{n^2})$ the slope is the same with respect to each arc that contains them, but in either case the tangent vector jumps to the opposite orientation when crossing each point while traveling continuously through the arcs. So these points are cusps, hence also singular. 
\end{remark}

\begin{corollary} \label{cor:N3 infinite many }
The set $\bd\Pi_{3}$ has countably many isolated singular points, which are precisely $$\left(\frac{1}{k},\frac{1}{k^2}\right), ~k \in \N_{>0}.$$ 
\end{corollary}
\begin{proof}
    By Corollary \ref{cor:boundaryN3} singular points of the lower boundary of $\Pi_{n,3}$ remain singular in $\bd\Pi_3$, except perhaps $(1,1)$ which is also part of the upper boundary. However, a short computation shows that $(1,1)$ is a cuspidal point of the curve $\left\{(t,t^{3/2}):0\le t\le1\right\}\cup L_1$.
\end{proof}

\begin{corollary}\label{cor:not semialgebraic}
The sets $\Pi_{d}$ and $E_d$ are not semialgebraic for all $d \geq 3$.
\end{corollary}
\begin{proof}
We show that $\Pi_3$ is not semialgebraic. Then for $d \geq 4$ the set $\Pi_d$ is not semialgebraic, since for $d\geq 3$ we have $\Pi_3 = \pi (\Pi_d)$, where $\pi : \R^d \to \R^2$ denotes the projection onto the first $2$ coordinates. Moreover, $\Pi_d$ is a polynomial image of the set $E_d$, so $E_d$ cannot be semialgebraic for $d\ge3$.
\medskip

\indent Suppose that the set $\Pi_3$ is semialgebraic, hence $\bd\Pi_3$ is semialgebraic. Let $T$ denote the union of the singular points of $\bd\Pi_3$. The set $T$ is semialgebraic since the singularity condition can be formalized as the vanishing and non-vanishing of a finite number of polynomial equalities. By (\cite[Theorem 2.4.4]{bochnak2013real}) every semialgebraic set is a disjoint union of a finite number of semialgebraic sets. So $T$ is a finite set of points, which contradicts Corollary \ref{cor:N3 infinite many }.
\end{proof}

\begin{remark}
    Using Green's theorem and the parametrizations given in Theorem \ref{thm:Parametrization Nnd} we can compute the area of $\Pi_{n,3}$ for each $n\ge3$ and thus that of $\Pi_3$. We do so by computing, for each $k=1,\dots,n-2$, the area of the region $A_k$ enclosed by the curves $U_k$, $U_{k+1}$ and $L_{k+1}$ (where $U_1:=L_1$). These regions partition $\Pi_{n,3}$, with $A_k$ and $A_{k+1}$ sharing only the boundary curve $U_{k+1}$. 
    
    A short computation then shows that, for each $k=2,\dots,n-1$ the area of $A_{k-1}$ is $\frac1{10}(\frac1{k^2}-\frac1{k^3})$, and so the area of $\Pi_{n,3}$ is $\frac1{10}(\sum_{k=2}^{n-1}\frac1{k^2}-\sum_{k=2}^{n-1}\frac1{k^3})$, therefore the area of $\Pi_3$ is $\frac1{10}(\zeta(2)-\zeta(3))$ where $\zeta$ is the Riemann zeta function.

    Observe also that, subject to $p_1=1$, the Jacobian of the map $(p_2,p_3)\mapsto(e_2,e_3)$ is $\frac16$, so the areas of $E_{n,3}$ and $E_3$ are $\frac16$ of those of $\Pi_{n,3}$ and $\Pi_3$. 
    
\medskip

\end{remark}

\subsection{The boundary of the image of sub-probability simplex} \label{sec:boundary of sub-probability simplex image}
To study the boundary of the set $\Pi_{n,d}$ in the limit as $n$ approaches infinity it will be useful to also consider the image of the sub-probability simplex $\tilde{\Delta}_n:=\{x\in \R^n_{\geq 0}\,:\, p_1(x)\leq 1\}$ and $\widetilde{W}_n := \tilde{\Delta}_n \cap \mathcal{W}_n$. 
We will see in Section \ref{sec:Vandermonde map limit image} that the limits of images of the probability simplex and the sub-probability simplex under the Vandermonde map coincide, and it will be more convenient to consider the limit of the sub-probability simplex.

The main difference to the case of the probability simplex $\Delta_{n-1}$ is that now we distinguish between points whose coordinates sum up to one and those whose coordinate sum is strictly below 1. This changes the points given by Theorem \ref{thm:1} slightly. On the one hand, we obtain two types of points with a coordinate sum strictly smaller than one which are mapped to the boundary. These points have $d-2$ distinct non-zero coordinates. On the other hand, we obtain one type of points corresponding to points with a coordinate sum equal to one which are mapped to the boundary. These points have $d-1$ distinct non-zero coordinates. More concretely, we have the following.

\begin{theorem}\label{thm:subprob}
For integers $n\geq d$ the boundary of $\nu_{n,\alpha}(\tilde{\Delta}_n)$ is given by the closure of the image of the following three types of points: type (1) and (2) points with multiplicity length $d-2$ and type (1) points with multiplicity length $d-1$ and with coordinate sum $1$.
\end{theorem}

As in Subsection \ref{subsection boundary of the vandermonde cell} we consider $\alpha= (\alpha_1, \ldots, \alpha_{d-1}) \in \Z_{\geq 1}^n$ with $\alpha_i < \alpha_{i+1}$, write $\beta := 2\cdot \alpha$, and assume $\alpha_1 \geq 2$. Then we have \[\nu_{n,\alpha}(\tilde{\Delta}_n)=\nu_{n,\beta}(\mathcal{W}_n \cap \{x\in \R^n : p_2(x) \leq 1\}).\]

By Remark \ref{rem:gen}, $V_k^\beta(c)$ is generically smooth and the rank of the Jacobian matrix of the map $(p_{\alpha_1},\ldots,p_{\alpha_k})$ at the relative interior of an $\ell$-dimensional face of $\widetilde{W}_n$ is $\min\{\ell,k\}$. 

For the sub-probability simplex $\tilde{\Delta}_n$ we will be interested in the intersection of $V_k^\beta(c)$ with the unit ball in $\R^n$ defined by $p_2(x)\leq 1$, which we denote by $\tilde{V}_k^\beta(c)$. The crucial difference is in adjusting the statements of Corollary \ref{cor:pap3 1}.

\begin{lemma} \label{lem:sub-simplex max/min}
Suppose that $V_{d-2}^\beta(c)$ and $V_{d-1}^{(2,\beta_1,\ldots,\beta_{d-2})}(1,c_1,\ldots,c_{d-2})$ are both smooth, and let $x$ be a critical point of $p_{\beta_{d-1}}$ on $\tilde{V}_{d-2}^\beta(c)$. Then $x$ belongs to the following two cases
\begin{enumerate}[label=(\alph*)]
   \item $x$ has multiplicity length $d-2$ and is either of type $(1)$ or $(2)$, or
    \item $x$ has multiplicity length $d-1$ and is of type $(1)$ with $p_2(x)=1$. 
\end{enumerate}
Additionally, all the local extrema are strict in these cases. For odd $d$, case $(a)$ type $(1)$ points are local maxima and case $(a)$ type $(2)$ points and case $(b)$ points are local minima. For even $d$, case $(a)$ type $(1)$ points are local minima, and case $(a)$ type $(2)$ points and case $(b)$ points are local maxima.
\end{lemma}

\begin{proof}
The proof follows the same line of arguments as the ones used for $\Delta_{n-1}$. 
As before we will argue via the Lagrange function defined in \eqref{eq:largange}. Now, however, the optimiality conditions that we need are the Karush-Kuhn-Tucker criterion \cite[Thm. C.15]{lasserre2010moments}, instead of Lagrange multipliers we used previously. The condition for $z=(z_1,\ldots,z_n)\in\R^n_{\geq 0}$ to be a critical point of $p_{\beta_{d-1}}$ is the following:
\[ 0=\nabla p_{\beta_{d-1}} (z)  +  \lambda_0\nabla p_2 (z)+\lambda_1 \nabla p_{\beta_1}(z) + \ldots + \lambda_{d-2} \nabla p_{\beta_{d-2}}(z), \]
with $\lambda_0\geq 0$ in case we are interested in a minimum, and $\lambda_0\leq 0$ in case we are interested in a maximum of $p_{\beta_{d-1}}$. As in the proof of Lemma \ref{lem:critical points} this leads to a univariate polynomial $f(t)$ of degree $\beta_{d-1}$ such that $f(z_i)=0$ for all $1\leq i\leq n$. However, depending on whether $p_2(z)=1$  or $p_2(z)<1$ we have  two  cases that we need to discuss separately:
\begin{enumerate}
    \item For the first case, we suppose that the constraint $p_2(z)\leq 1$ is not active, i.e., $p_2(z)<1$. In this case $\lambda_0=0$, and the arguments in the proof of 
    Corollary \ref{cor:pap3 1} apply. The only difference is that the number of power sums involved is $d-2$ instead of $d-1$. However, we proceed in exactly the same way, and we obtain that boundary points with $p_2<1$ are given by points of multiplicity length $d-2$ and type (1) and (2). We also get the information on local maxima and minima as claimed.
    \item For the second case, suppose that the constraint $p_2(z)\leq1$ is active (i.e. we have $p_2(z)=1$). Here we can also follow almost the same argument as in the proof Corollary \ref{cor:pap3 1}. However, unlike in the case when we use the constraint  $p_2=1$, here the constraint $p_2\leq 1$ imposes inequality constraints on $\lambda_0$. To analyze this difference, we observe, that for $d$ odd and $\lambda_0\geq 0$ we have by Descartes' rule of signs that the critical point $z$ can have multiplicity length at most $d-2$. Thus the multiplicity length of $z$ is exactly $d-2$, and it is of type (1) or (2). 
    Similarly, even $d$ and $\lambda_0\geq0$ will also yield a point of multiplicity length $d-2$ and type (1) or (2). Therefore for $d$  even, the only case that can lead to points of a different type is $\lambda_0<0$ which means that $z$ will be a local maximum, and for $d$ odd the only case is $\lambda_0>0$ which means that $z$ is a local minimum. Then by Corollary \ref{cor:pap3 1} these points will have multiplicity length $d-1$ and type (1).
\end{enumerate}

\end{proof}

\begin{lemma} \label{lem: subSimplex Vandermonde variety connected}
For all $1 \leq k \leq d-1$ the intersection of the Vandermonde variety $V_k^\beta(c)$ with $\{x \in \mathcal{W}_n : p_2(x) \leq 1\}$ is connected.    
\end{lemma}
\begin{proof}
Note that this is equivalent to verifying that $V_k^\alpha (c) \cap \widetilde{W}_n $ is connected.
 Suppose that $V_k^\alpha(c)\cap\widetilde{W}_n $ is disconnected. Since $V_k^\alpha(c) \cap \widetilde{W}_{n}$ is closed, it follows that there exist two non-empty closed sets $A$ and $B$ such that $V_k^\alpha(c) \cap \widetilde{W}_{n}=A\cup B$ and $A\cap B=\emptyset$. Let $\Delta_{n-1}(s)$ be the simplex in the nonnegative orthant where the sum of coordinates is $s$. We know that $V_k^\alpha(c)\cap \Delta_n(s) \cap W_{n-1}$ is connected for all $s\geq 0$ by Lemma \ref{lem:interval}. It follows that for all $s\geq 0$ we have $A\cap \Delta_n(s)  =\emptyset$ or $B\cap \Delta_n(s)=\emptyset$. Since both $A$ and $B$ are non-empty and closed, we see that there exist some $0<s_1<s_2<s_3$ such that $A\cap \Delta_n(s_1)$ is nonempty $A\cap\Delta_n(s_2)=B\cap \Delta_n(s_2)=\emptyset$ and $B\cap \Delta(s_3)$ is non-empty (switch $A$ and $B$ if needed). But then $V_k^\alpha(c) \cap \mathcal{W}_n$ is disconnected, which is a contradiction.   
\end{proof}

\begin{lemma}\label{lem: bound distinct coordinates sub-simplex}
Any pre-image in $\tilde{\Delta}_n$ of a point in $\bd \nu_{n,\alpha}(\tilde{\Delta}_n)$ has at most $d-1$ distinct positive coordinates.
\end{lemma}
\begin{proof}
We consider the pre-images in $\mathcal{W}_n \cap \{ x \in \R^n : p_2(x) \leq 1\}$ of points in $\bd \nu_{n,\alpha}(\tilde{\Delta}_n)$ under the map $\nu_{n,\beta}$. 
 We distinguish between the points with $p_2(x)=1$ and $p_2(x) < 1$. If $z \in V_{d}^{(2,\beta)}(1,c) \cap \mathcal{W}_n$, where $c=(c_1,\dots,c_{d-1})$, then $z$ is a smooth point of the Vandermonde variety $V^{(2,\beta_1,\ldots,\beta_{d-1})}_d(1,c)$ if and only if the number of distinct positive coordinates of $z$ is at least $d$ by Lemma \ref{lem:d-1 points}. 
 Since the Jacobian matrix of the map $\R^n \to \R^d, x \mapsto (p_2(x),p_{\beta_1}(x),\ldots,p_{\beta_{d-1}}(x)) $ has full rank at a smooth point $z$, $c$ cannot lie on the boundary of $\nu_{n,\alpha}(\tilde{\Delta}_n)$ by the inverse function theorem.
 
Analogously, any point in $V_{d-1}^\beta \cap \mathcal{W}_n \cap \{x \in \R^n :  p_2(x) < 1\}$ with at least $d-1$ distinct positive coordinates is a smooth point in the variety $V_{d-1}^\beta(c) $. Again by the inverse function theorem such a point cannot be mapped to a point in $\bd \nu_{n,\alpha}(\tilde{\Delta}_n)$.    
\end{proof}

Combining Lemmas \ref{lem:homeom} and Lemma \ref{lem: bound distinct coordinates sub-simplex} we obtain that the pre-image in $\tilde{\Delta}_n$ of any point in the boundary of $\nu_{n,\alpha}(\tilde{\Delta}_n)$ is unique up to permutation. 
\begin{lemma}\label{lem: unique pre-image subSimplex}
Any point in $\bd \nu_{n,\alpha}(\tilde{\Delta}_n)$ has a unique pre-image, up to permutation, in $\tilde{\Delta}_{n}$.
\end{lemma}
\begin{proof}
For any $c \in \bd\nu_{n,\alpha}(\tilde{\Delta}_n)$ the set $W_\alpha(c):=V^\alpha_{d-1}(c) \cap \widetilde{W}_{n} $ is connected by Lemma \ref{lem: subSimplex Vandermonde variety connected}. 
Moreover, the set $W_\alpha(c)$ consists only of points with at most $d-1$ distinct positive coordinates by Lemma \ref{lem: bound distinct coordinates sub-simplex}. However, there are only finitely many multiplicity patterns $\omega$ that encode the equal coordinates of points in $W_\alpha(c)$.
For multiplicity length $d-1$ the weighted Vandermonde map $\nu_{\omega,\alpha}$ mapping $\tilde{\Delta}_{d-1}$ to a subset of $\nu_{n,\alpha}(\tilde{\Delta}_n)$ is a homeomorphism by Lemma \ref{lem:homeom}. Thus, for any of the finitely many multiplicity patterns there is maximally one point in pre-image of $c$ in $\widetilde{W}_n$. However, the pre-image of $c$ in $\widetilde{W}_n$ is connected by Lemma \ref{lem: subSimplex Vandermonde variety connected} which implies that there is a unique pre-image in $\tilde{\Delta}_n$ up to permutation.  
\end{proof}

Additionally we need the following statement, which is the equivalent of Lemma \ref{lem:interval} for $\tilde{
\Delta}_n$. 

\begin{lemma}\label{lem:sub-simplex interval}
Let $n \geq d \geq 2$ and consider the projection $\pi: \nu_{n,\alpha}(\tilde{\Delta}_n)\rightarrow \nu_{n,\alpha'}(\tilde{\Delta}_n)$
onto the first $d-2$ coordinates. Then the pre-image of any point on the boundary of $\nu_{n,\alpha'}(\tilde{\Delta}_n)$ is a single point, and the pre-image of any point in the interior of $\nu_{n,\alpha'}(\tilde{\Delta}_n)$ is a non-degenerate interval. 
\end{lemma}
We follow the same roadmap as in the proof of Proposition \ref{prop: projection of Vandermonde cell}, as we already established the equivalent statements for the subprobability simplex.
\begin{proof}
Any point $\bar{c}$ in the boundary of $\nu_{n,\alpha'}(\tilde{\Delta}_n)$ has a unique pre-image in $\nu_{n,\alpha}(\tilde{\Delta}_n)$, since $\bar{c}$ has a unique pre-image in $\widetilde{W}_n$ by Lemma \ref{lem: unique pre-image subSimplex}.

We now want to show that any $\bar{c} \in \nu_{n,\alpha'}(\tilde{\Delta}_n)$ with a unique pre-image in $\nu_{n,\alpha}(\tilde{\Delta}_n)$ must lie in the boundary of $\nu_{n,\alpha'}(\tilde{\Delta}_n)$.
We suppose otherwise and assume that $\bar{c}$ lies in $\inti \nu_{n,\alpha'}(\tilde{\Delta}_n)$. Then there is a sequence of generic points $c_i \in \inti \nu_{n,\alpha'}(\tilde{\Delta}_n)$ which converge to $\bar{c}$. Thus for any $c_i$ the varieties $V_{d-2}^\beta(c_i)$ and $V_{d-1}^{(2,\beta)}(1,c_i)$ are generic. The pre-image of $c_i$ in $\nu_{n,\alpha}(\tilde{\Delta}_n)$ is an interval and the endpoints are images of points of type (1) and (2) of multiplicity length $d-2$, or of type (1) of multiplicity lengths $d-2$ and $d-1$ by Lemma \ref{lem:sub-simplex max/min}.
Moreover, for $i \to \infty$ the endpoints of the intervals converge to the unique pre-image of $\bar{c}$ in $\nu_{n,\alpha}(\tilde{\Delta}_n)$. We can suppose without loss of generality that all endpoints are either points of type (1) and (2) of length $d-2$, or of type (1) with lengths $d-2$ and $d-1$. In the first case, by Lemma \ref{lem: two multiplicity vector sequences} the unique pre-image of $\bar{c}$ is also attained as the image of a limit of points of type (1) or (2) and of multiplicity length $d-3$. In particular, this shows $\overline{c} \in \bd \nu_{n,\alpha'}(\tilde{\Delta}_n)$ by Lemma \ref{lem:sub-simplex max/min}.
In the other case, we have that $\overline{c}$ is the image of a limit of points of type (1) with coordinate sum $1$ and multiplicity length $d-2$ and therefore $\bar{c}$ lies also in the boundary of $\nu_{n,\alpha'}(\tilde{\Delta}_n)$ by Lemma \ref{lem:sub-simplex max/min}. 
\end{proof}
\begin{proof}[Proof of Theorem \ref{thm:subprob}]
    We follow the same path as in the proof of Theorem \ref{thm:1}.

    Suppose that $\alpha \in \Z_{\geq 2}^{d-1}$ is an integer exponent vector with $\alpha_1 < \alpha_2 < \ldots < \alpha_d$. 
Any point of type (1) or (2) with multiplicity length $d-2$, or of type (1) with coordinate sum $1$ and multiplicity length $d-1$ is indeed mapped to the boundary of $\Pi_{n,\alpha}$ by Corollary \ref{lem:sub-simplex max/min}.
The closure of the image of such points is also mapped to the boundary by continuity.
Now, we assume that $\overline{c} := (p_{\alpha_1}(x),\ldots,p_{\alpha_{d-1}}(x))$ is contained in the boundary of $\nu_{n,\alpha}(\tilde{\Delta})$. Since $\nu_{n,\alpha}(\tilde{\Delta})$ is the closure of its interior we can find a sequence of sufficiently general points $c_i$ from the interior of $\nu_{n,\alpha}(\tilde{\Delta})$ converging to $\overline{c}$, such that the Vandermonde varieties $V_{d-1}^{\beta}(c_i),V_d^{(2,\beta)}(1,c_i)$ are smooth. Consider the projection $\pi: \nu_{n,\alpha}(\tilde{\Delta})\to \nu_{n,\alpha'}(\tilde{\Delta})$ of $\nu_{n,\alpha}(\tilde{\Delta})$ onto the first $d-2$ coordinates as in Lemma \ref{lem:sub-simplex interval}. The pre-image of each $c_i$ is a non-degenerate interval and $p_{\beta_{d-1}}$ is either minimized or maximized at the endpoints of the interval. We apply Lemma \ref{lem:sub-simplex max/min} to see that the pre-images of the endpoints of the interval must be points of type (1) or (2) of multiplicity length $d-2$, or points of type (1) and coordinate sum $1$ and multiplicity length $d-1$.

If $\pi(\overline{c})$ lies in the interior of $\nu_{n,\alpha'}(\tilde{\Delta})$, then its pre-image under $\pi$ is also an interval, and $\bar{c}$ is an endpoint of this interval. Now the endpoints of intervals for $c_i$'s must converge to the endpoints of the interval for $\bar{c}$ by continuity. Therefore we realized $\overline{c}$ as the limit of points of type (1) or (2).

If $\pi(\bar{c})$ lies on the boundary of $\nu_{n,\alpha'}(\tilde{\Delta})$ then its pre-image under $\pi$ is a single point $\bar{s}$. It follows that the endpoints of the intervals for $c_i$ converge to $\bar{s}$ by continuity. Therefore we realized $\overline{c}$ as the limit of points of type (1) or (2).
\end{proof}

We revisit the image of the map $(p_2,p_3)$ but this time we extend its domain to the sub-probability simplex $\Tilde{\Delta}_n$. Let $\Tilde{\Pi}_{n,d}:=(p_2,\dots,p_d)(\Tilde{\Delta}_n)$.

\begin{example}[Description of $\bd\Tilde{\Pi}_{n,3}$] For $d=3$ the three families of points in Theorem \ref{thm:subprob} are 
\begin{align*}
    (0,\ldots,0,x_1)\in\Tilde{\Delta}_n& \quad\text{(type (1) and multiplicity length $d-2$)},\\
    (x_1,\ldots,x_1)\in\Tilde{\Delta}_n& \quad\text{(type (2) and multiplicity length $d-2$)},\\
    (0,\ldots,0,x_1,x_2,\ldots,x_2)\in\Delta_{n-1}& \quad\text{(type (1) and multiplicity length $d-1$)}.
\end{align*}
By Lemma \ref{lem:sub-simplex max/min}, and since $d=3$ is odd, the points $(0,\dots,0,x_1)$ give the upper boundary, which is the arc parametrized by $(t^2,t^3)$ for $0\le t\le1$.
The other two families give the lower boundary. Belonging to $\Delta_{n-1}$, the family $(0,\dots,0,x_1,x_2,\dots,x_2)$ gives precisely the $n-1$ lower arcs of $\Pi_{n,3}$ (see Theorem \ref{thm:Parametrization Nnd}). The points $(x_1,\ldots,x_1)$ give an additional lower arc, which joins $(\frac1n,\frac1{n^2})$ to the origin, and can be parametrized by $(\frac1nt^2,\frac1{n^2}t^3)$ for $0\le t\le1$.
\smallskip

A visualization of the transition from $\Pi_{4,3}$ to $\Tilde{\Pi}_{4,3}$ is shown in Figure \ref{fig:subprob43}.

\begin{figure}[h!]%  
    \centering
    {{\includegraphics[width=6cm]{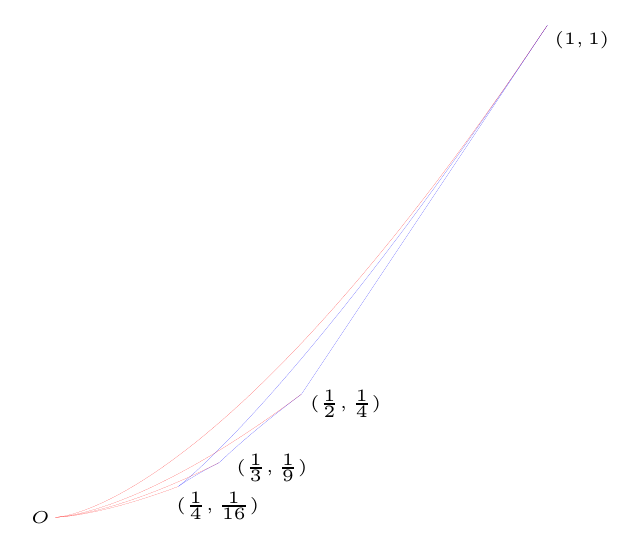}}}%
    \caption{Transition from $\Pi_{4,3}$ to $\Tilde{\Pi}_{4,3}$. Red curves represent the flow $(at^2,bt^3)$ of the vertices $(a,b)\in\Pi_{4,3}$ for $0\le t\le1$.}%
    \label{fig:subprob43}%
\end{figure}
\end{example}

We will go into a more detailed examination of the relative boundary of $\Tilde{\Pi}_{n,d}$ in Section \ref{sec:sub}.

\section{Combinatorial properties of the boundary of $\Pi_{n,d}$} \label{sec:comb prop}

Our main result in this section is the following theorem, which provides a combinatorial description of the boundary of the Vandermonde cell. 

\begin{theorem} \label{thm:combinatorially cyclic polytope}
The set $\Pi_{n,d}$ has the combinatorial structure of the cyclic polytope $C(n,d-1)$ (see Definition \ref{def:combst}). 
\end{theorem}

Cyclic polytopes are well-studied objects in polyhedral combinatorics.

\begin{definition}
For $n > d \geq 2$ the \emph{cyclic polytope} $C(n,d)$ is the convex polytope with $n$ vertices which are points on the real moment curve $(t,t^2,\ldots,t^d)$. 
\end{definition}
The combinatorial structure, i.e. the $f$-vector, of $C(n,d)$ is independent of the chosen points and its boundary is a $(d-1)$-dimensional \textit{simplicial polytope}. Thus, we can speak about \textit{the} cyclic polytope $C(n,d)$. 
Cyclic polytopes are interesting objects. For instance, by the \emph{upper bound theorem} $C(n,d)$ has the component-wise maximal $f$-vector among all $d$-dimensional convex polytopes with $n$ vertices \cite{mcmullen1970maximum,stanley1975upper}. We refer to \cite[Section 0]{ziegler2012lectures} for more background on cyclic polytopes. 

For all $n \geq d\geq 3$ we have that $\conv \left\{ \left( \frac{1}{k},\frac{1}{k^2},\ldots,\frac{1}{k^{d-1}} \right) : k \in [n] \right\}$ is the cyclic polytope $C(n,d-1)$ and we will usually use this choice of vertices for $C(n,d-1)$.

\begin{definition}\label{def:combst}
   A closed set $S \subset \R^d$ which is homeomorphic to a closed ball in $\R^d$ has the \emph{combinatorial structure} of the cyclic polytope $C(n,d)$ if there exists a homeomorphism $\Phi : \bd C(n,d) \to \bd S$ which is a diffeomorphism when restricted to the relative interior of any face of $\bd C(n,d)$. The \emph{vertices} of $S$ are the images of the vertices of $C(n,d)$. \\
 We call a set $S \subset \mathbb{R}^n$ a \emph{curved $k$-simplex} if $S$ is the image of $\Delta_k$ under a homeomorphism $f$, such that $f$ is a diffeomorphism when restricted to the relative interior of any face of $\Delta_k$. The \emph{vertices} of a curved $k$-simplex are the images of the vertices of the simplex $\Delta_k$. 
\end{definition}

Note that the boundary of a set that has the combinatorial structure of a cyclic polytope is the gluing of as many patches as the number of facets of $C(n,d)$. Each patch is a curved $d$-simplex with vertices labeled by the so-called Gale's evenness condition (see below). Moreover, patches of the boundary intersect if and only if the intersection of their sets of vertices is non-empty.
\medskip

 The facets of a cyclic polytope $C(n,d)$ are characterized by \textit{Gale's evenness condition}. For $k\ne0$ we write $\hat{k} := \left( \frac{1}{k},\frac{1}{k^2},\ldots,\frac{1}{k^{d}} \right)$.
 
\begin{theorem}[\cite{gale1963neighborly}] \label{thm:facets of cyclic polytope}
The facets of $C(n,d)$ are precisely given by all $\{ \hat{k} : k \in S\}$, where $S \subset [n]$ is any set of size $d$ satisfying 
\begin{enumerate}
    \item[i)] If $d$ is even, then $S$ is either a disjoint union of consecutive pairs $\{i,i+1\}$, or a disjoint union of consecutive pairs $\{i,i+1\}$ and $\{1,n\}$.
    \item[ii)] If $d$ is odd, then $S$ is a disjoint union of consecutive pairs $\{i,i+1\}$ and either the singleton $\{1\}$ or $\{n\}$.
\end{enumerate}
\end{theorem}
The standard formulation of Gale's evenness condition is the following.
Let $n > d$, $k_1 < \ldots < k_n \in \mathbb{R}$ and $T = \{\hat{k}_1,\ldots,\hat{k}_n\}$ be the vertices of $\conv \{ \hat{k}_i : 1 \leq i \leq n\}$. Then a set $T_d \subset T$ of size $d$ spans a facet of $C(n,d)$ if and only if any two elements in $T \setminus T_d$ are separated by an even number of elements from $T_d$ in the sequence $(\hat{k}_1,\hat{k}_2,\ldots,\hat{k}_n)$.

We briefly sketch an outline of our proof of Theorem \ref{thm:combinatorially cyclic polytope}. Let \[\nu_{n,d}^*:=(p_2,\ldots,p_d) : \Delta_{n-1} \to \R^{d-1}.\] From Theorem \ref{thm:1} we know that the boundary of $\Pi_{n,d}$ is the closure of the set of evaluations of $\nu_{n,d}^*$ at points in the probability simplex with multiplicity length $d-1$ and type (1) and (2). We associate these multiplicity vectors $m$ with $(d-2)$-dimensional simplices $\Delta_m^*$. Those simplices correspond to the facets of $C(n,d-1)$ (see Proposition \ref{prop:simplices and cyclic polytopes}). In Lemma \ref{lem:homeom 2} we observe that the restrictions of the map $\nu_{n,d}^*$ induce homeomorphisms $\nu_{n,d}^* : \Delta_m^* \to \nu_{n,d}^*(\Delta_m^*)$ which are diffeomorphisms when restricted to the relative interior of any face of $\Delta_m^*$ (as in Lemma \ref{lem:homeom}).

For a sequence $m = (m_0,m_1,\ldots,m_{d-1}) \in \N^d$ with $m_i \geq 1$ for all $1 \leq i \leq d-1$ we define the $(d-2)$-dimensional simplex $$\Delta_m^* := \{ (\underbrace{0,\ldots,0}_{m_0},\underbrace{x_1,\ldots,x_1}_{m_1},\ldots,\underbrace{x_{d-1},\ldots,x_{d-1}}_{m_{d-1}}) \in \R_{\geq 0}^n : x_i \leq x_{i+1} \forall i, \sum_{i=1}^{d-1}m_ix_i = 1\}.$$

\begin{lemma} \label{lem:homeom 2}
Let $m \in \N^{d}$ with $\sum_{i = 0}^{d-1} m_i = n$ and $m_i \geq 1$ for all $1 \leq i \leq d-1$. Then the map $ \nu_{n,d}^* : { \Delta_m^*} \to \nu_{n,d}^*(\Delta_m^*)$ is a homeomorphism and a diffeomorphism when restricted to the relative interior of any face of $\Delta_m^*$.
\end{lemma}
\begin{proof}
This is a special case of Lemma \ref{lem:homeom}.
\end{proof}

For $1 \leq k \leq n$ we write $\overline{k} := (\underbrace{0,\ldots,0}_{n-k},\underbrace{\frac{1}{k},\ldots,\frac{1}{k}}_k)$. We immediately obtain the following Corollary.

\begin{corollary} \label{cor:deformed facets}
Let $m \in \N^{d}$ with $\sum_{i = 0}^{d-1} m_i = n$ and $m_i \geq 1$ for all $1 \leq i \leq d-1$. Each simplex $\Delta_m^*$ is mapped to a curved $(d-2)$-simplex in $\R^m$ by $\nu_{n,d}^*$ and the vertices of the curved simplex $\nu_{n,d}^*(\Delta_m^*)$ are the points $(\frac{1}{k},\frac{1}{k^2},\ldots,\frac{1}{k^{d-1}})$ for all vertices $\overline{k}$ of $\Delta_m^*$.
\end{corollary}

It follows from Theorem \ref{thm:1} that the set $\bd \Pi_{n,d}$ is the union of the curved simplices $\nu_{n,d}^*( \Delta_m^*)$ for all multiplicity vectors of type (1) and (2) and multiplicity length $d-1$. We still have to show that the vertices of each simplex satisfy Gale's evenness condition. \\

Recall, a multiplicity vector $m$ of type (1) has the form $m_0 \geq 0$, $m_{2i-1}=1$ and $m_{2i} \geq 1$ and a multiplicity vector of type (2) is of the form $m_0=0, m_{2i-1} \geq 1$ and $m_{2i}=1$ for all $i$.

\begin{lemma}  \label{lem:gale}
Let $m \in \N^d$ with $\sum_{i=0}^{d-1}m_i=n$ be a multiplicity vector of multiplicity length $d-1$ and type (1) or (2). The vertices of $\Delta_m^*$ are 
\[ 
\left\{ \begin{array}{cc}
 \overline{n-m_0},\overline{n-m_0-1},\overline{n-m_0-m_2-1},\overline{n-m_0-m_2-2},\ldots     &  \text{ if $m$ is of type $(1)$}, \\
 \overline{n},\overline{n-m_1},\overline{n-m_1-1},\overline{n-m_1-m_3-1},\overline{n-m_1-m_3-2},\ldots     &   \text{ if $m$ is of type $(2)$}.
\end{array} \right.
\]
\end{lemma}
\begin{proof}
The vertices are the points where all but one of the defining inequalities of the simplex $\Delta_m^*$ are tight. Thus, the vertices are the $d-1$ points for which $x_i = 0, 1 \leq i \leq k-1$ and $x_k = \ldots = x_{d-1}$ for all $0 \leq k \leq d-2$.    
\end{proof}

In the following Proposition, we observe that multiplicity vectors of type (1) or (2) encode Gale's evenness condition with the identification $k \leftrightarrow \overline{k}$. 

\begin{proposition} \label{prop:simplices and cyclic polytopes}
Let $\Gamma$ be the map sending a multiplicity vector $m$ of a point $x \in \Delta_{n-1}$ with multiplicity length $d-1$ and type (1) or (2) to the subset $S$ of $[n]$ consisting of integers $k$ such that $\overline{k}$ is a vertex of $\Delta_m^*$. Then $\Gamma$ is a bijection between multiplicity vectors of points in $\Delta_{n-1}$ with multiplicity length $d-1$ and type (1) or (2) and subsets of $[n]$ satisfying Gale's evenness condition.
\end{proposition}
\begin{proof}
We distinguish between $d-1$ odd and even. 
First, we consider the case $d-1$ is odd. In this situation we have  by Lemma \ref{lem:gale} that the vertices of $\Delta_m^*$ are 
        \[ 
\left\{ \begin{array}{cc}
\overline{1} \text{ and } \overline{1+m_{d-1}},\overline{2+m_{d-1}},\ldots,\overline{n-m_0-1},\overline{n-m_0}  &   \text{ if $m$ is of type $(1)$}, \\
\overline{n}  \text{ and } \overline{m_{d-1}},\overline{m_{d-1}+1},\ldots,\overline{n-m_1-1},\overline{n-m_1}     &  \text{ if $m$ is of type $(2)$}. 
\end{array} \right.
\]   
    The corresponding set of integers satisfies Gale's evenness condition in both cases. 
Conversely, we suppose $S \subset [n]$ with $|S| = d-1$ and $S$ satisfies Gale's evenness condition. We construct the associated multiplicity vector $m$.
\begin{enumerate}[leftmargin=*]
    \item First, we suppose $S = \biguplus_{j=1}^{\frac{d}{2}-1}\{i_j,i_j+1\} \uplus\{1\}$ and $1 < i_1 < \ldots < i_{\frac{d}{2}-1}<n$. Then $d-2$ is even and we define $m_{d-2}:=i_1-1 \geq 1$ and $m_{d-2j} := i_j-i_{j-1}-1 \geq 1$ for all $1 < j \leq \frac{d}{2}-1$. Then, $$m_2+m_4+\ldots+m_{d-2} = i_{\frac{d}{2}-1}-\frac{d}{2}+1 \leq n-\frac{d}{2}.$$ We set $m_1 = m_3 = \ldots = m_{d-1} :=1$ and $ 0 \leq m_0 := n - m_1+m_2+\ldots+m_{d-1}.$ We note $m_1+\ldots+m_{d-1}  \leq n -\frac{d}{2} + \lceil \frac{d-1}{2}\rceil \leq n$. Thus, the multiplicity vector $m$ is indeed of type (1) and the simplex $\Delta_m^*$ has vertex set $\{\overline{k} : k \in S\}$.
    \item Second, we suppose $S = \biguplus_{j=1}^{\frac{d}{2}-1}\{i_j,i_j+1\} \uplus\{n\}$ and $1 \leq i_1 < \ldots < i_{\frac{d}{2}-1}<n$. We define $m_{d-1} := i_1$, $m_{d-2j+1}:=i_j-i_{j-1}-1 \geq 1$ for all $2 \leq j \leq \frac{d}{2}-1$. Then, $$m_{d-1}+m_{d-3}+\ldots+m_3 = i_{\frac{d}{2}-1}-\frac{d}{2}+2 \leq n-\frac{d}{2}.$$ We define $m_2 = m_4 =\ldots = m_{d-2} := 1$ and $m_1 := n-(m_{d-1}+m_{d-1}+\ldots+m_2) \geq 1$. We have that $m$ is of type (2) and the simplex $\Delta_m^*$ has vertex set $\{\overline{k} : k \in S\}$. 
\end{enumerate}

We now turn to the case $d-1$ is even. 
Again by Lemma \ref{lem:gale} the vertices of $\Delta_m^*$ are         \[ 
\left\{ \begin{array}{cc}
\overline{m_{d-1}},\overline{m_{d-1}+1},\ldots,\overline{n-m_0-1},\overline{n-m_0}    &  \text{ if $m$ is of type $(1)$}, \\
\overline{1},\overline{n} \text{ and } \overline{m_{d-2}+1},\overline{m_{d-2}+2},\ldots,\overline{n-m_1-1},\overline{n-m_1} &   \text{ if $m$ is of type $(2)$}.
\end{array} \right.
\]   
    Also in this case  the corresponding set of integers satisfies Gale's evenness condition. Conversely, let $S \subset [n]$ be of size $d-1$ satisfying Gale's evenness condition. We construct the associated multiplicity vector $m$. 
\begin{enumerate}[leftmargin=*]
    \item First, we suppose \[S = \{1,n\} \uplus \biguplus_{j=1}^{\frac{d-3}{2}} \{ i_j,i_j+1\} \text{ with } < i_1 < \ldots < i_{\frac{d-3}{2}} < n-1.\] We define $m_{2} = m_4 = \ldots = m_{d-1} := 1$, $m_{d-2}:=i_1-1 \geq 1$, $m_{d-{2k}} := i_k-i_{k-1} -1 \geq 1$ for all $2 \leq k \leq \frac{d-3}{2}$, and $m_1 :=n-(m_2+\ldots+m_{d-1})\geq 1$. The vector $m$ is of type (2) and the simplex $\Delta_m^*$ has the vertex set $\{\overline{k} : k \in S\}$. 
    \item Second, we suppose \[S = \biguplus_{j=1}^{\frac{d-1}{2}}\{i_j,i_j+1\}\text{ and }1 \leq i_1 < \ldots < i_{\frac{d-1}{2}} < n.\] We set $m_1 = \ldots = m_{d-2}:=1$, $m_{d-1} := i_1 \geq 1$, $m_{d-2k+1} := i_k-i_{k-1}-1 \geq 1$ for all $2 \leq k \leq \frac{d-1}{2}$ and $m_0 := n-(m_1+\ldots+m_{d-1}) \geq 0$. Then, the vector $m$ is of type (1) and the vertex set of the simplex $\Delta_m^*$ is $\{\overline{k} : k \in S$\}. 
\end{enumerate}
\end{proof}

\begin{corollary}
    \label{cor:3.9}
    The map $$ \abb{\kappa_{n,d}}{ \bd \left( \conv \{ \left(\frac{1}{i},\ldots,\frac{1}{i^{d-1}}\right) : 1 \leq i\leq n \}\right)}{\bigcup_{m \in \N^d \text{ has type }(1),(2)}\Delta_m^*}{\sum_{j=1}^{d}\lambda_{i_j} \left( \frac{1}{i_j},\ldots,\frac{1}{i^{d-1}_j}\right) }{\sum_{j=1}^{d}\lambda_{i_j} \left(0,\ldots,0,\frac{1}{i_j},\ldots,\frac{1}{i_j}\right)} $$ is a homeomorphism and a diffeomorphism when restricted to the relative interior of any face of $\bd \left( \conv \{ \left(\frac{1}{i},\ldots,\frac{1}{i^{d-1}}\right) : 1 \leq i\leq n \}\right)$.
\end{corollary} 
The map $\bd C(n,d-1) \to \bd \Pi_{n,d}$ in Theorem \ref{thm:combinatorially cyclic polytope} will be the composition $\nu_{n,d}^* \circ \kappa_{n,d} $.
\begin{proof}
Since any facet of the cyclic polytope $\conv \{ \left(\frac{1}{i},\ldots,\frac{1}{i^{d-1}}\right) : 1 \leq i\leq n \}$ is the convex hull of $d-1$ points on the moment curve, these points are convexly independent. Moreover, the facet defining sets of vertices correspond to the multiplicity vectors $m$ of multiplicity length $d-1$ and type (1) or (2) by Proposition \ref{prop:simplices and cyclic polytopes}. Thus the map $\kappa_{n,d}$ is well-defined. This map is a homeomorphism and a diffeomorphism when restricted to the relative interior of any face of $\bd \left( \conv \{ \left(\frac{1}{i},\ldots,\frac{1}{i^{d-1}}\right) : 1 \leq i\leq n \}\right)$ since it is an affine linear map on any facet of $C(n,d-1)$.
\end{proof}

We are ready to give a proof of Theorem 
\ref{thm:combinatorially cyclic polytope}. 

\begin{proof}[Proof of Theorem \ref{thm:combinatorially cyclic polytope}]
Let $C(n,d-1) = \conv \{ \left(\frac{1}{i},\ldots,\frac{1}{i^{d-1}}\right) : 1 \leq i\leq n \}$. By Corollary \ref{cor:3.9} the map $\kappa_{n,d} : \bd C(n,d-1) \to \bigcup_{m \in \N^d \text{ type }(1) \text{ or }(2)}\Delta_m^*$ is a homeomorphism and a diffeomorphism when restricted to the relative interior of any face of $\bd C(n,d-1)$. \\
We consider the map $\nu_{n,d}^* \circ \kappa_{n,d} : \bd C(n,d-1) \to \bd \Pi_{n,d}$. The map $\nu_{n,d}^*$ is surjective by Theorem \ref{thm:1}. 
The restriction of the map $\nu_{n,d}^*$ on $\Delta_{n-1}$ to the pre-image of $\bd \Pi_{n,d}$ is injective by Lemma \ref{lem: unique pre-image Simplex}. 

On each simplex $\Delta_m^*$ the map $\nu_{n,d}^*$ is a homeomorphism and a diffeomorphism when restricted to the relative interior of any face by Lemma \ref{lem:homeom}. Thus, $\nu_{n,d}^* \circ \kappa_{n,d}$ is a diffeomorphism on the restriction of any face of $\bd C(n,d-1)$. 
\end{proof}

\begin{remark}
    The natural extension of $\kappa_{n,d}$ to the interior of $C(n,d-1)$
    $$ \abb{\kappa}{\conv \{ \left(\frac{1}{i},\ldots,\frac{1}{i^{d-1}}\right) : 1 \leq i\leq n \}}{\Pi_{n,d}}{\sum_{j=1}^{d}\lambda_{i_j} \left( \frac{1}{i_j},\ldots,\frac{1}{i^{d-1}_j}\right) }{\nu_{n,d}^*\left(\sum_{j=1}^{d}\lambda_{i_j} \left(0,\ldots,0,\frac{1}{i_j},\ldots,\frac{1}{i_j}\right) \right)} $$
    is not well defined. For instance, $\frac{1}{13}(\frac{1}{2},\frac{1}{4})+\frac{12}{13}(\frac{1}{4},\frac{1}{16})= \frac{27}{52}(\frac{1}{3},\frac{1}{9})+\frac{25}{52}(\frac{1}{5},\frac{1}{25})$ but 
    \begin{align*}
    \begin{array}{llll}
        &\nu_{5,3}\left(1/13 \left(0,0,0,1/2,1/2\right)+12/13\left(0,1/4,1/4,1/4,1/4\right)\right) & =& \left(1, 85/338\right) \\
        &\nu_{5,3}\left(27/52\left(0,0,1/3,1/3,1/3\right)+25/52\left(1/5,1/5,1/5,1/5,1/5\right) \right) & =& \left(1,319/1352\right).
    \end{array}
    \end{align*}
    \end{remark}

\subsection{The sub-probability simplex} \label{sec:sub}
In analogy to the discussion about $\Pi_{n,d}$ above we briefly discuss the combinatorial properties of the image of the sub-probability simplex under the Vandermonde map. 
We write $\overline{\infty}$ for the point $(0,\ldots,0) \in \Tilde{\Delta}_n$. The following discussion of vertices parallels Lemma \ref{lem:gale} and Proposition \ref{prop:simplices and cyclic polytopes}.

We discuss the three cases of Theorem \ref{thm:subprob} separately. Suppose that $x$ has type (1) and multiplicity length $d-1$. Then its multiplicity vector $m$ has the form $(m_0,m_1,\dots,m_{d-1})$. 
To this multiplicity vector $m$ we associate the simplex $\tilde{\Delta}_m^*$ with vertices $\overline{m_{d-1}},\overline{m_{d-2}+m_{d-1}},\dots,$ $\overline{m_1+m_2+\dots+m_{d-1}}$. It can be seen that the closure of the image under $\nu_{n,d}^*$ of all points with multiplicity type $m$ is the same as the image of $\tilde{\Delta}_m^*$.

Suppose that $x$ has type (1) and multiplicity length $d-2$. Then its multiplicity vector $m$ has the form $(m_0,m_1,\dots,m_{d-2})$. To this multiplicity vector $m$ we associate the simplex $\tilde{\Delta}_m^*$ with vertices $\overline{\infty},\overline{m_{d-2}},\overline{m_{d-3}+m_{d-2}},\dots,\overline{m_1+m_2+\dots+m_{d-2}}$. It can be seen that the closure of the image under $\nu_{n,d}^*$ of all points with multiplicity type $m$ is the same as the image of $\tilde{\Delta}_m^*$.

Suppose that $x$ has type (2) and multiplicity length $d-2$. Then its multiplicity vector $m$ has the form $(m_1,\dots,m_{d-2})$. To this multiplicity vector $m$ we associate the simplex $\tilde{\Delta}_m^*$ with vertices $\overline{\infty}, \overline{m_{d-2}},\overline{m_{d-3}+m_{d-2}}\dots,\overline{m_1+m_2+\dots+m_{d-2}}$. It can be seen that the closure of the image under $\nu_{n,d}^*$ of all points with multiplicity type $m$ is the same as the image of $\tilde{\Delta}_m^*$.

\begin{remark}
The image of the first type of simplex above is a patch on the boundary of $\Pi_{n,d}$ and was studied previously in Section \ref{sec:comb prop}. The other two types consist of taking a $d-2$ dimensional simplex $S$ (which has vertices given by sums of $m_i$) and taking the convex hull with the origin, which is denoted by $\overline{\infty}$. Let $S'=\conv \{\{0\}\cup S\}$. The image $\nu_{n,d}^*(S')$ can be obtained from $\nu_{n,d}^*(S)$ as follows: take a point $u=(u_1,\dots,u_{d-1}) \in \nu_{n,d}^*(S)$ and define a parametrized segment $(u_1t^2,u_2t^3,\dots,u_{d-1}t^d)$ for $0\leq t\leq 1$. We can think of this segment as the point $u$ ``flowing'' to the origin via a segement of a monomial curve. The image $\nu_{n,d}^*(S')$ is the union of all such segments for $u\in \nu_{n,d}^*(S)$.
\end{remark}

Similarly to Lemma \ref{lem:homeom} and Corollary \ref{cor:deformed facets} we can also show the following: 

\begin{proposition}\label{prop:sub diff}
The map $\nu_{n,d}^*$ restricted to a simplex $\tilde{\Delta}_m^*$ is a homeomorphism and a diffeomorphism when restricted to the relative interior of any face of the simplex.
\end{proposition}

    Additionally, we can see that the $d-1$ points specified as vertices for the simplices above satisfy Gale's evenness condition, e.g. in the first example the points are $\frac{1}{m_{d-1}},\frac{1}{m_{d-2}+m_{d-1}}, \dots, \frac{1}{m_1+\dots+m_{d-1}},$ and in the second example $0,\frac{1}{m_{d-2}},\frac{1}{m_{d-3}+m_{d-2}}, \dots, \frac{1}{m_1+\dots+m_{d-2}}$. It can be seen in the same way that the image $\tilde{\Pi}_{n,d}=\nu_{n,d}^*(\tilde{\Delta}_{n})$ also has the combinatorial type of a cyclic polytope. Here the vertices of the image are the images of $0,\frac{1}{n},\frac{1}{n-1},\dots,1$ on the moment curve in $\mathbb{R}^{d-1}$. We formalize our discussion in the following proposition.

\begin{proposition}\label{prop:bound sub}
The boundary of $\tilde{\Pi}_{n,d}$ consists of patches which are images of the simplices $\tilde{\Delta}_m^*$. The vertices of the patches correspond to subsets of $\{0\}\cup \{\frac{1}{k} \,\, :\,\, k \in \mathbb{Z}, \,\, 1\leq k \leq  n\}$ (or equivalently $[n]\cup \{ \infty \}$) satisfying Gale's evenness condition.
\end{proposition}

\section{The boundary at infinity} \label{sec:Vandermonde map limit image}

We show that the set $\bd \Pi_d$ is a gluing of countably many patches and each patch is a curved $(d-2)$-simplex. The vertices of all patches lie in a countable set of points on the moment curve.
The boundary patches correspond to subsets of size $d-1$ of this set satisfying Gale's evenness condition, in direct analogy to cyclic polytopes (see Theorem \ref{thm:limit bound}). We begin with investigating properties of $\Pi_d$.

\begin{lemma} \label{lem:bound}
For $x =(x_1,\ldots,x_{d-1})\in \Pi_d$ we have $(t^2x_1,\ldots,t^d x_{d-1}) \in \Pi_d$ for all $0 \leq t \leq 1$. 
\end{lemma}
\begin{proof}
Suppose $x$ is in the interior of $\Pi_d$, thus there exists $n$ such that $x\in\Pi_{n,d}$. Then $x=(p_2,\ldots,p_d)(z)$ for a point $z \in \Delta_{n-1}$. Consider the point $z'$: $$z'=\left(tz,\frac{1-t}{n},\ldots,\frac{1-t}{n}\right) \in \Delta_{2n-1}.$$ We see that $$ \nu_{2n,d}^*(z')= (t^{2}x_1,\ldots,t^{d}x_{d-1})+\left(\frac{(1-t)^{2}}{n},\ldots,\frac{(1-t)^{d}}{n^{{d-1}}}\right) \in \Pi_{2n,d},$$ so letting $n\to\infty$ we obtain $(t^2x_1,\ldots,t^dx_{d-1}) \in \Pi_d$ because $\Pi_d$ is the closure of $\bigcup_{n\ge d}\Pi_{n,d}$. By the continuity of the weighted scaling and the closedness of $\Pi_d$ this property also holds for all points in the boundary of $\Pi_d$.
\end{proof}

We now show that one arrives at the limit set $\Pi_d$ regardless if one considers the probability or the sub-probability simplex as the domain of the Vandermonde map.

\begin{lemma} \label{image_in_subprob_goes_to_prob}
The set $\Pi_d$ is the closure of the limit of the sets $\tilde{\Pi}_{n,d}=\nu_{n,d}^*(\tilde{\Delta}_{n})$ as $n$ goes to infinity, i.e. $\Pi_d = \clo \left(\bigcup_{n\geq d}\tilde{\Pi}_{n,d}\right)$.
\end{lemma}
\begin{proof}
Since $\Delta_{n-1}\subset\tilde{\Delta}_{n}$ for all $n\ge d$ then $\Pi_d \subset \clo (\bigcup_{n\geq d}\tilde{\Pi}_{n,d})$. If $x=(x_1,\ldots,x_n) \in  \tilde{\Delta}_{n}$ then 
$\theta := \sum_{j =1}^n x_j \leq 1$  
and thus $x^{[m]} := (x,(1-\theta)/m,\ldots,(1-\theta)/m) \in \Delta_{n+m-1}$. Therefore $$\Pi_d \ni \lim_{m \to \infty} \nu_{m+n,d}^*(x^{[m]}) = \nu_{n,d}^*(x).$$ So $\Pi_d\supset\clo(\bigcup_{n\geq d}\tilde{\Pi}_{n,d})$ since $\Pi_d$ is closed.
\end{proof}

\begin{lemma}\label{lem:boundary as limit of sequences}
Let $ d \geq 3$. If $u \in \bd \Pi_d$ then there exists a sequence $(u_n)$ such that $u_n \in \bd \tilde{\Pi}_{n,d}$ and $u_n \to u$ as $n \to \infty$.

\end{lemma}
For a set $A \subset \R^n$ and a point $x \in \R^n$ we denote the \textit{distance} from $x$ to $A$ by $\operatorname{d}(x,A)$, i.e. $\operatorname{d}(x,A)  = \inf ~ \{ ||x-a|| : a \in A\}$. 
\begin{proof}
Suppose that $u \in \bd \Pi_d$. Then, since the sets $\tilde{\Pi}_{n,d}$ are nested increasingly, we have that $\operatorname{d}(u,\tilde{\Pi}_{n,d})$ is a decreasing sequence in $n$. However, $u \in \bd \Pi_d$ implies that for all $n\ge d$ $u \not \in \inti \nu_{n,d}^*(\tilde{\Delta}_n)$, and thus $\operatorname{d}(u,\tilde{\Pi}_{n,d}) = \operatorname{d}(u,\bd \tilde{\Pi}_{n,d})$. Hence $\operatorname{d}(u,\bd \tilde{\Pi}_{n,d}) \to 0$ which implies that there exists a sequence $(u_n)_n$ with $u_n \in  \bd \tilde{\Pi}_{n,d}$ and $u_n \to u$. 
\end{proof}
 Let $T=\{0\}\cup \{\frac{1}{k} \,\, : \,\,  k\in \mathbb{N}_{>0}\}$. A finite subset $R$ of $T$ is said to satisfy Gale's
evenness condition if for any pair of points $a_1,a_2\in T\setminus R$ the number of points of $R$ lying between $a_1$ and $a_2$ is even. This is completely analogous to the standard definition for a finite set.

\begin{theorem}\label{thm:limit bound}
 Let $T=\{0\}\cup \{\frac{1}{k} \,\, : \,\,  k\in \mathbb{N}_{>0}\}$. The boundary of $\Pi_d$ consists of patches whose vertices are (images under the moment map of) subsets of $T$ of size $d-1$ satisfying Gale's evenness condition (see Section \ref{sec:sub} for the definition of the corresponding simplices). Each such patch occurs as a boundary patch of $\tilde{\Pi}_{n,d}$ for all $n$ sufficiently large.
\end{theorem}

\begin{proof}
In the proof we identify elements of $T$ with their images under the moment map. By Lemmas \ref{image_in_subprob_goes_to_prob} and \ref{lem:boundary as limit of sequences} a point on the boundary of $\Pi_d$ is a limit of boundary points of $\tilde{\Pi}_{n,d}$. Proposition \ref{prop:bound sub} tells us that the boundary of $\tilde{\Pi}_{n,d}$ is given by patches whose vertices are subsets of $\{0\}\cup\{\frac{1}{k} \, \, : \, \, k\in[n]\}$. We take the limits of patches by considering the limits of corresponding sets of $d-1$ vertices. Note that the set $T=\{0\}\cup \{\frac{1}{k} \,\, : \,\,  k\in \mathbb{N}_{>0}\}$ is compact. 
Given a sequence $v_i$ of vertex sets, where $v_i=(v^{(1)}_{i},\dots,v_i^{(d-1)})$ with $v^{(1)}_i<v^{(2)}_i<\dots<v_i^{(d-1)}$, $v_i^{(k)}\in T$, we can consider limits of $k$-th coordinates $v_i^{(k)}$ for all $1\leq k \leq d-1$. By compactness of $T$, we can restrict to a subsequence of $v_i$ (which by abuse of notation we will still call $v_i$), such that $v_i^{(k)}$ converges for all $1\leq k \leq d-1$. Let $\overline{v}^{(k)}=\lim_{n\rightarrow \infty}v^{k}_n$, and $\overline{v}=(\overline{v}^{(1)},\dots, \overline{v}^{(d-1)})$. Note that if $\overline{v}^{(k)} \neq 0$, then we have $v_i^{(k)}=\overline{v}^{(k)}$ for all $i$ sufficiently large. Therefore, if $\overline{v}^{(1)} \neq 0$ then $\overline{v}^{(k)} \neq 0$ for all $1\leq k\leq d-1$, and therefore the sequence of patches is eventually constant. 

Now suppose that $\overline{v}^{(1)}=0$. Let $m$ be the smallest index $k$ (resp. $M$ be the largest index $k$) such that $\overline{v}^{(k)}=0$, and the sequence $v_i^{(k)}$ is not equal to $0$ for all $i$ sufficiently large. If no such integers $m,M$ exist, then it follows that the sequence of patches is again eventually constant.

We suppose that $m,M$ exist. Observe that since each $v_i$ satisfies Gale's evenness condition, we must have that $M-m$ is odd, i.e. an even number of vertices distinct from $0$ in $v_i$ converge to $0$, in particular, we must have at least two. 

To see that $M-m$ is odd let $R_i \subset T$ denote the set of size $d-1$ containing the coordinates of the vector $v_i$. We distinguish between the cases $m=1$ and $m>1$ and apply Gale's evenness condition. If $m=1$ we set $a_1 :=0,a_2 := \overline{v}^{(M+1)}/2 \in T$ and for sufficiently large $i$ we have $a_1 < v_i^{(1)} < \ldots < v_i^{(M)} < a_2 < v_i^{(M+1)}$. Since the set $R_i$ satisfies Gale's evenness condition $a_1,a_2 \in T\setminus R_i$ are separated by an even number of elements in $R_i$. Thus $M$ is even and $M-m=M-1$ is odd. In the second case we must have $m=2$. This is, since $v_i^{(1)}=0$ for sufficiently large $i$, but $v_i^{(2)} > v_i^{(1)}$ for all $i$. First, we suppose $M=d-1$. Then, since $v_i^{(1)}=0 < a_1:=v_i^{(2)}/2 < v_i^{(2)} < \ldots < v_i^{(M)} < a_2 :=1$ for all sufficiently large $i$, $a_1,a_2 \in T\setminus R_i$ we have that $M-1$ is even and therefore $M-m$ is odd. Second, if $M< d-1$ then for sufficiently large $i$ we have $v_i^{(1)}=0 < a_1:=v_i^{(2)}/2 < v_i^{(2)} < \ldots < v_i^{(M)} < a_2 :=\frac{1}{1/\overline{v}^{(M+1)}+1} < v_i^{(M+1)}$ and $a_1,a_2 \in T\setminus R_i$. Thus $M-2=M-m$ must be odd.

Note that the distinct points of $\overline{v}$ will still satisfy Gale's evenness condition. Therefore by Proposition \ref{prop:sub diff} the points of the patches corresponding to vertices $v_i$ converge to boundary points of patches of $\tilde{\Pi}_{n,d}$ specified in the theorem statement.
The above analysis shows that either a sequence of patches becomes eventually constant, or the sequence of patches approaches the boundary of a patch of $\tilde{\Pi}_{n,d}$ specified in the theorem. Therefore, the boundary of $\Pi_d$ is contained in the union of patches such that their vertices are subsets of $T$ of size $d-1$ satisfying Gale's evenness condition, and furthermore the interiors of all such patches lie on the boundary of $\Pi_d$. The theorem now follows by Lemma \ref{image_in_subprob_goes_to_prob}.
\end{proof}

\begin{remark}
   We can view patches of the boundary of  $\tilde{\Pi}_{n,d}$ as limits of patches of the boundary of $\Pi_{n,d}$. This can be seen by looking at the limits of vertices for a patch as in the proof of Theorem \ref{thm:limit bound}. A boundary patch of $\Pi_{n,d}$ has $d-1$ vertices, and this vertex set satisfies Gale's evenness condition in $[n]$ (or equivalently in $\{\frac{1}{k} \,\, : \,\, k\in[n]\}$). As we take limits of these sets of vertices we obtain all subsets of $T=\{0\}\cup \{\frac{1}{k} \,\, : \,\,  k\in \mathbb{N}_{>0}\}$ of size $d-1$ which satisfy Gale's evenness condition in $T$.
\end{remark}

\section{Convex hull for elementary symmetric polynomials and test sets for copositivity} \label{sec:Convex hull}
In this section we study the convex hulls of the sets $E_{n,d}$, $\Pi_{n,d}$, $E_d$ and $\Pi_d$. Recall that, for $n\ge d$, $E_{n,d}$ is the image of the probability simplex under the so called \emph{Vieta map}
\begin{align*}
    \Delta_{n-1}&\to\R_{\ge0}^{d-1}\\
    x&\mapsto(e_2,\dots,e_d)(x),
\end{align*}
where $e_k$ is the $k$-th elementary symmetric polynomial.
\medskip

Although $ E_{n,d} \cong \Pi_{n,d}$ and $E_d \cong \Pi_d$ are diffeomorphic, we show that $\conv E_{n,d}$ has some properties which are not shared by $\conv \Pi_{n,d}$. We relate the study of the convex hulls to copositivity of certain symmetric forms. Based on our finding that $\conv E_{n,d}$ is a convex polytope with vertices $\{(e_2,\ldots,e_d)(\bar{k})\, :\, k\in[n]\}$ we geometrically explain and slightly generalize a result by Choi, Lam and Reznick on a test set for nonnegativity of even symmetric sextics \cite{choi1987even}. 
\medskip

Our main finding of this section is that $\mathcal{E}_{n,d}:=\conv E_{n,d}$ is a polytope that has the combinatorial type of the cyclic polytope $C(n,d-1)$. 
\medskip
 
Recall $\bar{k}:=(0,\dots,0,\frac1k,\dots,\frac1k)\in\Delta_{n-1}$, for $k\in[n]$, so $(e_2,\dots,e_d)(\bar{k})=\left({k\choose2}\frac1{k^2},\dots,{k\choose d}\frac1{k^d} \right)$. 
 
\begin{theorem} \label{thm:Convex hull Image of Elementary}
For $d\ge3$ the set $\mathcal{E}_{n,d}$ is the image of the cyclic polytope $C(n,d-1)$ under an invertible affine linear map, and it is the convex hull of the points $(e_2,\dots,e_d)(\bar{k})$ for $k\in[n]$.
\end{theorem}

The above observation appeared for the first time in the context of extremal combinatorics. In the planar setting it was proven by Bollobás to give a description of the convex hull of the range of edge versus triangle densities of graphs \cite{bollobas1976relations}. The result was extended to larger dimensions shortly afterwards and new proofs appeared for instance also in \cite{foregger1987relative,kovavcec2012note,riener2012degree,linear}. 
We present a proof using the following two Lemmas. The following proof is a formalization of Bollobás's original argument, which we borrow from \cite{zhao2023graph} and provide for completeness. 

\begin{lemma}[\cite{zhao2023graph},~Lemma~5.4.3] \label{lem:zhao}
A real finite linear combination of $n$-variate elementary symmetric polynomials attains its minimum value on the probability simplex $\Delta_{n-1}$ in at least one point of multiplicity length equal to one.
\end{lemma}
\begin{proof}
    Let $\phi(x):=c_1+c_2e_2(x)+\ldots+c_de_d(x)$ with $c_i\in\R$ and, among its minimizers on $\Delta_{n-1}$ consider $x^*$ with the maximum number of zero coordinates. We show that $x^*=\bar{k}$ up to permutation for some $k\in[n]$. 
    
    Since $\phi$ is symmetric then $\phi(x)=A+Bx_1+Bx_2+Cx_1x_2$, where $A,B,C$ are functions of $x_3,\ldots,x_n$. If $x^*=(1,0,\dots,0)$ up to permutation we are done, so suppose without loss of generality that $x_1^*>0$ and $x_2^*>0$. 
    
    By fixing $x_1+x_2=x_1^*+x_2^*$ we obtain $\phi(x)=Cx_1x_2+D$ where $D$ is a function of $x_3,\dots,x_n$. If $C(x^*)\geq 0$, holding $x_1+x_2=x_1^*+x_2^*$ fixed, we set either $x_1=0$ or $x_2=0$ and obtain a minimizer with more zero coordinates than $x^*$, a contradiction. So $C(x^*)<0$ and we obtain that $\phi(x)$ is minimized when $x_1=x_2$ (when the sum of nonnegative $x_1,x_2$ is fixed, its product is maximized when $x_1=x_2$). Since the choice of two positive coordinates was arbitrary, we must have $x^*=\bar{k}$ (up to permutation) for some $k\in[n]$.
\end{proof}

\begin{lemma} \label{prop:affine isomorphism}
For $n\ge d \geq 3$ the points $(e_2,\dots,e_d)(\bar{k})$ for $k\in[n]$ are the vertices of a polytope of combinatorial type $C(n,d-1)$.
\end{lemma}
\begin{proof}

Observe that $\binom{k}i\frac1{k^i}=\frac{k(k-1)\cdots(k-i+1)}{i!k^i}=\frac1{i!}(1-\frac1k)(1-\frac2k)\cdots(1-\frac{i-1}k)$ and consider the curve $f(t):=(f_2(t),\dots,f_d(t))$ where $f_i(t):=\frac1{i!}(1-t)(1-2t)\cdots(1-(i-1)t)$ for $t\in\R$. The points $(e_2,\dots,e_d)(\bar{k})$ for $1\leq k\leq n$ then belong to the curve $f(t)$. Now, since $f_i(t)$ is a polynomial of degree $i-1$ in $t$, there exists an affine linear isomorphism $\varphi:\R^{d-1}\to\R^{d-1}$ sending the moment curve $(t,t^2,\dots,t^{d-1})$ to $(f_2(t),\dots,f_d(t))$. Further observe that, by construction, $\varphi$ sends $(\frac1k,\frac1{k^2},\dots,\frac1{k^{d-1}})$ to $(e_2,\dots,e_d)(\bar{k})$ for all $k\in\N_{>0}$. Since affine linear isomorphisms preserve combinatorial types of polytopes we get the result.
\end{proof}

\begin{proof}[Proof of Theorem \ref{thm:Convex hull Image of Elementary}]
By the Krein-Milman theorem a compact convex set in $\R^k$ is the convex hull of its extreme points. Extreme points of a compact convex set $K$ in $\R^k$ are minima of affine linear maps on $K$. An affine linear map on the compact convex set $\mathcal{E}_{n,d}$ is an affine linear combination of $e_2(x),\dots,e_d(x)$, $x\in\Delta_{n-1}$, so by Lemma \ref{lem:zhao} its minima is attained at a point $\bar{k}$ for some $k\in[n]$. Therefore the set of extreme points of $\mathcal{E}_{n,d}$ is contained in the set $\{(e_2,\dots,e_d)(\bar{k}):k\in[n]\}$. But by Lemma \ref{prop:affine isomorphism} the points $\{(e_2,\dots,e_d)(\bar{k}):k\in[n]\}$ are the vertices of a polytope of combinatorial type $C(n,d-1)$, so in particular they are in convex position, and, since they belong to $\mathcal{E}_{n,d}$, they must be precisely the extreme points of $\mathcal{E}_{n,d}$. It follows from the proof of Lemma \ref{prop:affine isomorphism} that the set $\mathcal{E}_{n,d}$ is the image of the cyclic polytope $\conv \{\nu_{n,d}^*(\overline{k})\, :\, k \in [n]\}$ under an invertible affine linear map.
\end{proof}

Observe that for fixed $d$ the sets $E_{n,d}$, for $n\ge d$, are nested non-decreasingly, i.e. $E_{n,d}\subseteq E_{n+1,d}$ for all $n\ge d$. This follows from the fact that $e_k(x_1,\dots,x_n,x_{n+1})|_{x_{n+1}=0}=e_k(x_1,\dots,x_n)$ for all $n\ge d$ and $k=1,\dots,d$. Hence also $\mathcal{E}_{n,d}\subseteq\mathcal{E}_{n+1,d}$ for all $n\ge d$. As $n$ increases, the vertices of $\mathcal{E}_{n,d}$ accumulate towards the point $$\lim_{k\to\infty}\left( {k \choose 2}\frac{1}{k^2},\ldots,{k \choose d}\frac{1}{k^d}\right)=\left(\frac{1}{2!},\ldots,\frac{1}{d!}\right),$$ which suggests the following description for the limit set $\mathcal{E}_{d}:=\clo( \bigcup_{n\ge d}\mathcal{E}_{n,d})$.

\begin{proposition} \label{cor:image infinite probability simplex}
$\mathcal{E}_{d} = \conv \left\{  \left\{ (e_2,\dots,e_d)(\bar{k}) : k \in \N_{>0} \right\}  \cup \{ \left( \frac{1}{2!},\frac{1}{3!},\ldots,\frac{1}{d!} \right) \} \right\}. $
\end{proposition}
\begin{proof}
The set $\bigcup_{n \geq d} \mathcal{E}_{n,d}$ is convex because the sets $\mathcal{E}_{n,d}$ are nested, and so its closure $\mathcal{E}_d$ is convex too. Denote $\underline{d}:=(\frac1{2!},\dots,\frac1{d!})$ and observe that by definition $(e_2,\dots,e_d)(\bar{k})\in\mathcal{E}_d$, and $\underline{d}\in\mathcal{E}_{d}$ since $\mathcal{E}_d$ is closed. Hence $S\subset\mathcal{E}_d$ where $S:=\conv \left\{\left\{  (e_2,\dots,e_d)(\bar{k}) : k \in \N_{>0} \right\}  \cup \{ \underline{d} \} \right\}$. Furthermore, $\mathcal{E}_{n,d} \subset S$ for all $n \geq d $ and thus $\clo \left( \bigcup_{n \geq d} \mathcal{E}_{n,d} \right) \subset S$, since $S$ is closed.
\end{proof}

Figure \ref{fig:1 E23 and E26} depicts how the new vertices in $\mathcal{E}_{n,3}$ accumulate towards the point $\left(\frac{1}{2!},\frac{1}{3!}\right)$ as $n$ increases. We mark the vertices in the figures below.

\begin{figure}[h!]%
    \centering
    \begin{tikzpicture}
    \begin{scope}[scale=5]
        \draw[thin, blue, fill=blue!20, opacity=0.5] 
            (0.0, 0.0) -- (0.5, 0.0) -- (0.6667, 0.2222) -- cycle;
        \foreach \x/\y in {0.0/0.0, 0.5/0.0, 0.6667/0.2222} {
            \fill (\x, \y) circle (0.005);
        }
        \node at (0.3, -0.15) {};
    \end{scope}
    \end{tikzpicture}
    \qquad
    \begin{tikzpicture}
    \begin{scope}[xshift=4cm, scale=5]
        \draw[thin, blue, fill=blue!20, opacity=0.5] 
            (0.0, 0.0) -- (0.5, 0.0) -- (0.6667, 0.2222) -- (0.75, 0.375) -- (0.8, 0.48) -- (0.8333, 0.5556) -- cycle;
        \foreach \x/\y in {0.0/0.0, 0.5/0.0, 0.6667/0.2222, 0.75/0.375, 0.8/0.48, 0.8333/0.5556} {
            \fill (\x, \y) circle (0.005);
        }
        \node at (0.45, -0.15) {};
    \end{scope}
    \end{tikzpicture}
    \qquad
    \begin{tikzpicture}
  \begin{scope}[xshift=8cm, scale=5]
        \draw[thin, blue, fill=blue!20, opacity=0.5] 
            (0.0, 0.0) -- (0.5, 0.0) -- (0.6667, 0.2222) -- (0.75, 0.375) -- (0.8, 0.48) -- (0.8333, 0.5556) --
            (0.8571, 0.6122) -- (0.875, 0.65625) -- (0.8889, 0.6914) -- (0.9, 0.72) -- (0.9091, 0.7438) --
            (0.9167, 0.7639) -- (0.9231, 0.7811) -- (0.9286, 0.7959) -- (0.9333, 0.8089) -- (0.9375, 0.8203) -- 
            (0.9412, 0.8304) -- (0.9444, 0.8395) -- (0.9474, 0.8476) -- (0.95, 0.855) -- (0.9524, 0.8617) -- 
            (0.9545, 0.8678) -- (0.9565, 0.8733) -- (0.9583, 0.8785) -- (0.96, 0.8832) -- (0.9615, 0.8876) -- 
            (0.96296, 0.8916) -- cycle;
        \foreach \x/\y in {
            0.0/0.0, 0.5/0.0, 0.6667/0.2222, 0.75/0.375, 0.8/0.48, 0.8333/0.5556, 
            0.8571/0.6122, 0.875/0.65625, 0.8889/0.6914, 0.9/0.72, 0.9091/0.7438, 
            0.9167/0.7639, 0.9231/0.7811, 0.9286/0.7959, 0.9333/0.8089, 0.9375/0.8203,
            0.9412/0.8304, 0.9444/0.8395, 0.9474/0.8476, 0.95/0.855, 0.9524/0.8617, 
            0.9545/0.8678, 0.9565/0.8733, 0.9583/0.8785, 0.96/0.8832, 0.9615/0.8876, 0.96296/0.8916} {
            \fill (\x, \y) circle (0.005);
        }
        \node at (0.65, -0.15) {};
    \end{scope}   
    \end{tikzpicture}
    \caption{The sets $\mathcal{E}_{3,3}$, $\mathcal{E}_{6,3}$ and $\mathcal{E}_{20,3}$}%
    \label{fig:1 E23 and E26}%
\end{figure}
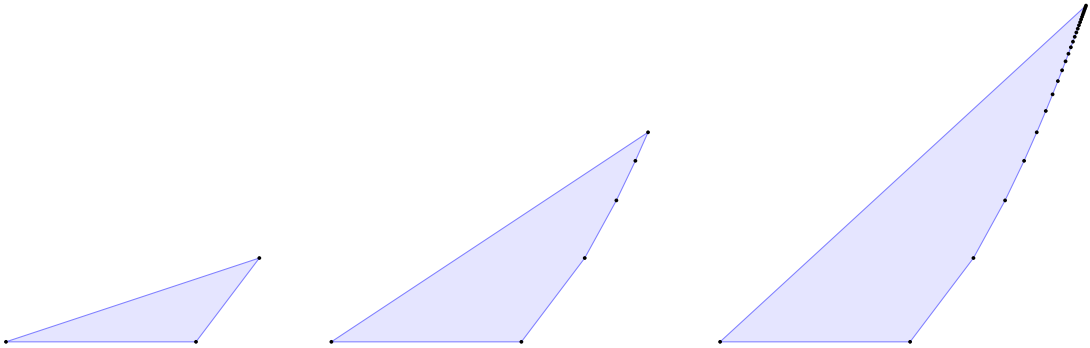

We have seen that $E_{n,d}$ is contained in $\conv\{(e_2,\dots,e_d)(x):x\in C_n\}$ where $C_n:=\{\bar{k}:k\in[n]\}$. Observe that $C_n$ is precisely the set of points in $\Gamma_{n-1}:=\Delta_{n-1}\cap\mathcal{W}_n$ where the Jacobian matrix of the Vieta map $(e_2,\ldots,e_d) : \R^n \to \R^{d-1}$ has rank one. This follows since the rank of the Jacobian matrix is preserved under diffeomorphism (Newton's identities (\ref{eq:Newton's identities}) provide polynomial diffeomorphisms between $E_{n,d}$ and $\Pi_{n,d}$), and $C_n$ is precisely the set of points on $\Gamma_{n-1}$ where the Jacobian matrix of the Vandermonde map $(p_2,\ldots,p_d) : \R^n \to \R^{d-1}$ has rank one.
\medskip

Observe that $(p_2,\dots,p_d)(\bar{k})=\left(\frac{1}{k},\frac{1}{k^2},\ldots,\frac{1}{k^{d-1}}\right)$ for each $k\in[n]$, and as seen in Remark \ref{rem:arcs pi3}, the arcs $L_k$, $k\in[n-1]$, are concave (and the upper arc $U_{n-1}$ is below the line joining $(1,1)$ and $(\frac1n,\frac1{n^2})$), so $\Pi_{n,3}\subset\conv\{ (\frac{1}{k},\frac{1}{k^2}):k\in [n]\}$. So it seems natural to ask whether an analogue to Proposition \ref{cor:image infinite probability simplex} holds for $\Pi_{n,d}$. 
Namely, is $\Pi_{n,d}\subset\conv\{(p_2,\dots,p_d)(x):x\in C_n\}$? One of the reasons this holds for $d=3$ but might fail for $d\ge4$ (and it fails) is that the diffeomorphism provided by Newton's identities is affine linear for $d=3$ while it is not for $d\ge4$ (e.g. $p_4$ is quadratic in $e_2$).

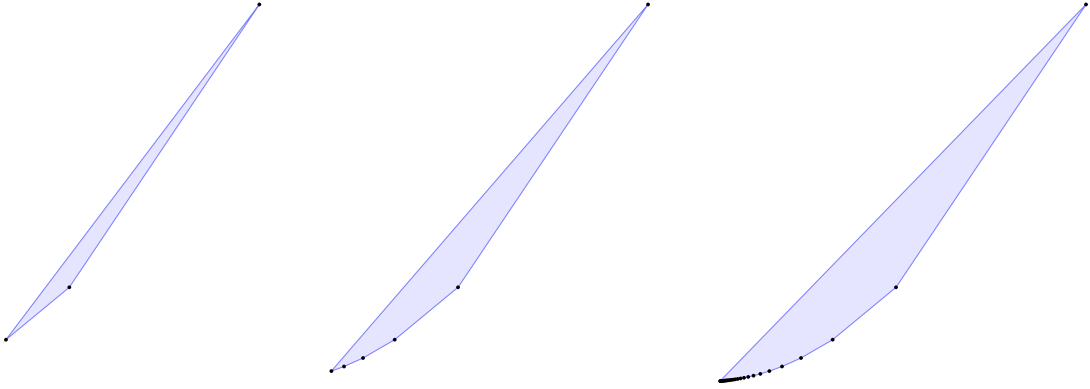
\begin{figure}[h!]%
    \centering
    \begin{tikzpicture}
    \begin{scope}[scale=5]
        \draw[thin, blue, fill=blue!20, opacity=0.5] 
            (1.0, 1.0) -- (0.5, 0.25) -- (0.3333, 0.1111) -- cycle;
        \foreach \x/\y in {1.0/1.0, 0.5/0.25, 0.333333/0.111111} {
            \fill (\x, \y) circle (0.005);
        }
        \node at (0.5, -0.15) {};
    \end{scope}
    \end{tikzpicture}
    \qquad
    \begin{tikzpicture}
        \begin{scope}[xshift=4cm, scale=5]
        \draw[thin, blue, fill=blue!20, opacity=0.5] 
            (1.0, 1.0) -- (0.5, 0.25) -- (0.3333, 0.1111) -- (0.25, 0.0625) -- (0.2, 0.04) -- (0.1667, 0.0278) -- cycle;
        \foreach \x/\y in {1.0/1.0, 0.5/0.25, 0.3333/0.1111, 0.25/0.0625, 0.2/0.04, 0.1667/0.0278} {
            \fill (\x, \y) circle (0.005);
        }
        \node at (0.35, -0.15) {};
    \end{scope}
    \end{tikzpicture}
    \qquad
    \begin{tikzpicture}
        \begin{scope}[xshift=8cm, scale=5]
        \draw[thin, blue, fill=blue!20, opacity=0.5] 
            (1.0, 1.0) -- (0.5, 0.25) -- (0.3333, 0.1111) -- (0.25, 0.0625) -- (0.2, 0.04) -- (0.1667, 0.0278) --
            (0.1429, 0.0204) -- (0.125, 0.015625) -- (0.1111, 0.01235) -- (0.1, 0.01) -- (0.0909, 0.00826) --
            (0.0833, 0.00694) -- (0.07692, 0.00592) -- (0.07143, 0.0051) -- (0.06667, 0.00444) -- (0.0625, 0.00391) -- 
            (0.05882, 0.00346) -- (0.05556, 0.00309) -- (0.05263, 0.00277) -- (0.05, 0.0025) -- (0.04762, 0.00227) -- 
            (0.04545, 0.00207) -- (0.04348, 0.00189) -- (0.04167, 0.00174) -- (0.04, 0.0016) -- (0.03846, 0.00148) -- 
            (0.03704, 0.00137) -- cycle;
        \foreach \x/\y in {
            1.0/1.0, 0.5/0.25, 0.3333/0.1111, 0.25/0.0625, 0.2/0.04, 0.1667/0.0278, 
            0.1429/0.0204, 0.125/0.015625, 0.1111/0.01235, 0.1/0.01, 0.0909/0.00826, 
            0.0833/0.00694, 0.07692/0.00592, 0.07143/0.0051, 0.06667/0.00444, 0.0625/0.00391,
            0.05882/0.00346, 0.05556/0.00309, 0.05263/0.00277, 0.05/0.0025, 0.04762/0.00227, 
            0.04545/0.00207, 0.04348/0.00189, 0.04167/0.00174, 0.04/0.0016, 0.03846/0.00148, 
            0.03704/0.00137} {
            \fill (\x, \y) circle (0.005);
        }
        \node at (0.55, -0.15) {};
    \end{scope}    
    \end{tikzpicture}
    \caption{The convex polytopes $\conv \Pi_{n,3}$ for $n=3$, $n=6$ and $n=20$}%
    \label{fig:3 P23 and P26}%
\end{figure}

\begin{proposition} \label{prop:not power sums}
For all $n\ge d\ge4$ the set $\Pi_{n,d}$ is not contained in $\conv \left\{ \left( \frac{1}{k},\frac{1}{k^2},\cdots,\frac{1}{k^{d-1}} \right) : k \in [n] \right\}$. Moreover, $\Pi_d \not \subseteq \conv \left\{(0,\ldots,0), \left( \frac{1}{k},\frac{1}{k^2},\cdots,\frac{1}{k^{d-1}} \right) : k \in \N_{>0} \right\}$.
\end{proposition}
\begin{proof}
Consider the symmetric polynomial $f(p_1,p_2,p_3,p_4) = 2p_4-3p_3p_1+p_2p_1^2$ which is expressed in power sums. For $m:=n-1$ we define \[  g_m(t) := f(p_1,\ldots,p_4)(t,\underbrace{1,\ldots,1}_{m \text{ times}}) = 
-mt^3 + (m^2+m)t^2 + (2m^2-3m)t + m^3-3m^2+2m.\] 
The leading coefficient of the univariate polynomial $g_m(t)$ is negative, so $g_m(a)<0$ for sufficiently large $a>0$. Hence $f$ cannot be nonnegative on $\R_{\geq 0}^n$ and, since $f$ is homogeneous, it cannot be nonnegative on $\Delta_{n-1}$. However, the symmetric form $f(1,p_2,p_3,p_4)$ is nonnegative on the points on the moment curve of the form $(1/k,1/k^2,1/k^3)$, since \[ f\left(1,\frac1k,\frac1{k^2},\frac1{k^3}\right) = \frac{(k-1)(k-2)}{k^3} \geq 0\] for all integers $k \geq 1$. 
Since $f(1,p_2,p_3,p_4)$ is affine in the $p_i$'s and $f(1,1/k,1/k^2,1/k^3) \geq 0$, the polynomial $f(p_1,p_2,p_3,p_4)$ is nonnegative on $\conv \left\{ \left( 1,\frac{1}{k},\frac{1}{k^2},\frac{1}{k^3} \right) : k \in [n] \right\}$. Thus, we have $\Pi_{n,4} \not \subset \conv \left\{ \left( 1,\frac{1}{k},\frac{1}{k^2},\frac{1}{k^3} \right) : k \in [n] \right\}$. 
For $d > 4$ the claim follows by projection to the first four coordinates and the argument above.
\end{proof}

\subsection{Test sets for copositivity}
Choi, Lam and Reznick studied even symmetric sextics and found strikingly simple test sets for nonnegativity in any number of variables \cite{choi1987even}. Namely, given an $n$-variate even symmetric sextic $f(p_2,p_4,p_6)=ap_2^3+bp_4p_2+cp_6$, $a,b,c\in\R$, $f$ is nonnegative if and only if $f(t,t,t)\ge0$ for $t=1,\dots,n$. From our setup above we can naturally restate it as follows.  

\begin{theorem}[\cite{choi1987even},~Theorem~3.7] \label{thm:choi-lam-reznick}
Let $f(p_2,p_4,p_6)$ be an even symmetric sextic in $n \geq 3$ variables. Then $f$ is nonnegative if and only if $f\left(1,\frac{1}{k},\frac{1}{k^2}\right)$ is nonnegative for all $k \in [n]$.
\end{theorem}

\begin{remark} \label{rmk:choi-lam-reznick}
In contrast to Theorem \ref{thm:choi-lam-reznick}, a consequence of Proposition \ref{prop:not power sums} is that the finite set $\{(\frac1k,\frac1{k^2},\dots,\frac1{k^{d-1}}):k\in[n]\}$ is not enough for testing copositivity of $n$-variate hook-shaped polynomials (changing elementary symmetric polynomials to power sums) of degree $d$, for any $d\ge4$. 
\end{remark}

The convexity properties of the Vieta map discussed above allow us to geometrically explain and generalize Theorem \ref{thm:choi-lam-reznick} to a slightly larger family of symmetric polynomials.

A hook-shaped polynomial is homogeneous and thus copositive if and only if it is nonnegative on the probability simplex $\Delta_{n-1}$. Restricting to the probability simplex we have $e_1=1$ and thus hook-shaped polynomials are affine linear combinations of $e_2,\ldots,e_d$. Now, by Theorem \ref{thm:Convex hull Image of Elementary} $\E_{n,d}:=\conv E_{n,d}$ is a convex polytope, hence nonnegativity of a hook-shaped polynomial $f$ on $E_{n,d}$ is equivalent to its nonnegativity on the vertices of $\mathcal{E}_{n,d}$.

\begin{theorem}\label{thm:discrete test set in elementarys}
Let $f=g(e_1,\dots,e_d)$ be a hook-shaped polynomial of degree $d$ in $n\ge d$ variables. Then $f$ is copositive if and only if $g(1,(e_2,\ldots,e_d)(\bar{k}))\ge0$ for all $k \in[n]$.
\end{theorem}
\begin{proof}
Since $f$ is homogeneous we can restrict to the domain $\Delta_{n-1}$, where $e_1$ is the constant function $1$. As $g(1,e_2,\ldots,e_d)$ is affine linear in $e_2,\dots,e_d$, we observe that $f$ is copositive if and only if $g(1,x)\ge0$ for all $x\in\mathcal{E}_{n,d}$. In particular, $g(1,x)\ge0$ on $\mathcal{E}_{n,d}$ if and only if $g(1,x)$ is nonnegative on the vertices of $\mathcal{E}_{n,d}$, which are precisely the claimed points by Theorem \ref{thm:Convex hull Image of Elementary}.
\end{proof}

The convex geometry of the limit set $\mathcal{E}_d$ (see Proposition \ref{cor:image infinite probability simplex}) allows us to test nonnegativity of hook-shaped polynomials for any number of variables.  

\begin{corollary} \label{cor:limit test set}
Let $f=g(e_1,e_2,\ldots,e_d)$
be a hook-shaped polynomial. Then $f$ is nonnegative in any number of variables $n\geq d$ if and only if $g$ is nonnegative on the discrete set $\left\{ (e_1,e_2,\ldots,e_d)(\bar{k}) \, : \, k \in \N_{>0} \right\}$.
\end{corollary} 
\begin{proof}
    The proof is analogous to the one for Theorem \ref{thm:discrete test set in elementarys} changing $\mathcal{E}_{n,d}$ by $\mathcal{E}_d$, and using Proposition \ref{cor:image infinite probability simplex}. By continuity, the set $\left\{ (e_1,e_2,\ldots,e_d)(\bar{k}) \, : \, k \in \N_{>0} \right\}$ is enough to test nonnegativity because $\left( \frac{1}{2!},\frac{1}{3!},\ldots,\frac{1}{d!} \right)$ is its accumulation point.
\end{proof}

Observe that the set $\mathcal{E}_d$ is not semialgebraic for any $d\ge3$ since it has countably many vertices by Proposition \ref{cor:image infinite probability simplex}.

\begin{remark} 
Let $T=\{\left(\frac{1}{2!},\ldots,\frac{1}{d!}\right)\} \cup \{\left( {k \choose 2}\frac{1}{k^2},\ldots,{k \choose d}\frac{1}{k^d} \right): k \in \N_{>0}\}$. The boundary of $\mathcal{E}_d$ consists of the convex combinations of $d-1$ points in $T$ satisfying Gale's evenness condition. 
\end{remark}

\section{Undecidability of nonnegativity of trace polynomials} \label{sec:undecidability}

In this Section, we show that the problem of deciding nonnegativity of trace polynomials in real symmetric matrices of all sizes is undecidable (see Theorem \ref{thm:undecidable traces}). This result stays sharp in contrast to the case of finitely many variables. Surprisingly, we then prove that the analogous problem defined with normalized traces is decidable (see Theorem \ref{thm:normalized decidable}). The key for the undecidability lies in the geometry of $\Pi_3$. To prove Theorem \ref{thm:undecidable traces} we show that deciding copositivity of multihomogeneous \emph{product symmetric} polynomials in several groups of variables and in any number of variables is an undecidable problem (see Theorem \ref{thm:undecidable}). The proof follows from \cite{hatami2011undecidability,blekherman2022undecidability} on undecidability in graph homomorphism densities.

\begin{definition}
For a variable $X$ we denote by $\Tr (X^n)$ the \emph{formal trace symbol} on $X^n$. A \emph{trace polynomial} in the variables $X_1,\ldots,X_k$ is a polynomial expression in formal trace symbols of powers of the variables $X_1,\ldots,X_k$. A trace polynomial is \emph{univariate} if $k=1$.
\end{definition}

For instance, $2\Tr (X_1^4)\Tr (X_2)-\Tr(X_1^2)[\Tr(X_2^3)]^5$ is a trace polynomial in the variables $X_1,X_2$.  
A trace polynomial can naturally be evaluated on square matrices of all sizes. All eigenvalues of a real symmetric matrix $A$ are real. Thus we have $\Tr (A^k) \in \R$ for all $k \in \N$. 
We call a trace polynomial $f(X_1,\ldots,X_k)$ \textit{nonnegative} if $f(A_1,\ldots,A_k) \geq 0$ for all real symmetric matrices $A_1,\ldots,A_k$ of all sizes. We show that establishing nonnegativity of a trace polynomial is an undecidable problem. 
\begin{theorem} \label{thm:undecidable traces}
The following decision problem is undecidable.
\begin{itemize}
    \item[{\footnotesize Instance:}] A positive integer $k$ and a trace polynomial $f(X_1,\ldots,X_k)$.
    \item[{\footnotesize Question:}] Is $f(M_1,\ldots,M_k)$ nonnegative for all real symmetric matrices $M_1,\ldots,M_k$  of all sizes? 
\end{itemize}
\end{theorem}
The undecidability occurs already if we only allow $\Tr(X_i^2),\Tr(X_i^4),\Tr(X_i^6)$ in any of the $k$ variables.
We now give an intuitive explanation of the hardness of the decision problem, and relate trace nonnegativity to the geometry of the limit Vandermonde cell.
For an $(n\times n)$-matrix $A$ we can replace $\Tr(A^m)$ by the power sum $p_m(\lambda)$ in the eigenvalues $\lambda_1,\ldots,\lambda_n$ of $A$. A polynomial in $k$ pairwise disjoint groups of variables $x^{(i)}=(x_1^{(i)},\ldots,x_{n_i}^{(i)})$ for $1 \leq i \leq k$ is \emph{multihomogeneous} if it is homogeneous in all of the groups of variables. For instance, $p_2(x^{(1)})p_4^2(x^{(2)})p_3(x^{(3)})+p_1^2(x^{(1)})p_8(x^{(2)})p_2(x^{(3)})p_1(x^{(3)})$ is multihomogeneous in $x^{(1)},x^{(2)},x^{(3)}$. A trace polynomial is \emph{multihomogeneous} if the associated polynomial in power sums in several groups of variables is multihomogeneous. A multihomogeneous trace polynomial in $\Tr(X_i^2),\Tr(X_i^4),\Tr(X_i^6)$ for $1 \leq i \leq k$ is nonnegative if the associated polynomial \[f(p_2(x^{(1)}),p_4(x^{(1)}),p_6(x^{(1)}),\ldots,p_2(x^{(k)}),p_4(x^{(k)}),p_6(x^{(k)}))\] is nonnegative on $(\{1\}\times \Pi_3)^k$.
The key to the hardness of the problem is the geometry of $\Pi_3$ which we investigated in Section \ref{sec:Vandermonde map}. We reduce establishing nonnegativity of multihomogeneous integer polynomials on $(\Pi_3)^k$ to deciding nonnegativity of $k$-variate polynomials on $\N^k$ which is known to be undecidable \cite{hatami2011undecidability}.

\begin{remark}
It follows from Theorem \ref{thm:undecidable traces} that there cannot exist a unified algorithm or an effective certificate to determine the validity of polynomial inequalities in traces of powers of real symmetric matrices of all sizes. Note that for a finite number of variables it follows by Artin's solution to Hilbert's 17th problem \emph{\cite{artin1927zerlegung}} that validity of polynomial inequalities on semialgebraic sets is decidable. 
\end{remark}

Nonnegativity of trace polynomials is investigated in the context of non-commutative real algebraic geometry. There one usually considers normalized trace polynomials. In \cite{klep2021positive} the authors prove a Positivstellensatz for univariate normalized trace polynomials. 

\begin{definition}
For a variable $X$ we denote by $\widetilde{\operatorname{tr}}(X)$ the \emph{formal normalized trace symbol} on $X$. A \emph{normalized trace polynomial} in the variables $X_1,\ldots,X_k$ is a polynomial expression in formal normalized trace symbols of powers of the variables $X_1,\ldots,X_k$. A normalized trace polynomial is \emph{univariate} if $k=1$.
\end{definition}
As the name formal normalized trace operator already indicates, for a matrix $A \in \R^{n \times n}$ we define the evaluation $\widetilde{\operatorname{tr}}(A) := \frac{1}{n}\Tr (A)$. A normalized trace polynomial is \textit{nonnegative} if its evaluation on all real symmetric matrices of all sizes is nonnegative. 

\begin{theorem}\label{thm:normalized decidable}
 The following decision problem is decidable.
\begin{itemize}
    \item[{\footnotesize Instance:}] A positive integer $k$ and a normalized trace polynomial $f(X_1,\ldots,X_k)$.
    \item[{\footnotesize Question:}] Is $f(M_1,\ldots,M_k)$ nonnegative for all real symmetric matrices $M_1,\ldots,M_k $ of all sizes?
\end{itemize}   
\end{theorem}

For matrices of fixed size deciding nonnegativity of normalized trace polynomials and trace polynomials is equivalent. The sharp contrast appears when we ask about nonnegativity for matrices of all sizes. Geometrically, the limit image of the normalized Vandermonde map of $\R_{\geq 0}^n$ corresponds to the set of the first $d$ moments of a probability measure supported on $\mathbb{R}_{\geq 0}$ \cite{acevedo2024power}. It is well-known that this set can be described by linear matrix inequalities \cite{MR3729411}. In particular, the normalized limit is semialgebraic for all $d$. The phenomenon of decidability for normalized trace can also be explained with the half-degree principle (\cite{timofte2003positivity}, Corollary 2.1), and we follow this direction in our proof.

\subsection{Proof of Theorem \ref{thm:undecidable traces}} 
Let $\Delta := \{x \in \R_{\geq 0}^\N : \sum_{i=1}^\infty x_i =1 \}$ denote the infinite probability simplex.
We write $\mathfrak{p}_m := \sum_{i \in \N} x_i^m$ for the \emph{power sum function} and $\mathfrak{e}_m:=\sum_{I \subset \N, |I|=m}\prod_{i \in I}x_i$ for the \emph{elementary symmetric function} in countably many variables. If $x \in \R^\N$ contains only finitely many non-zero coordinates we could also write $p_k(x)$ (resp. $e_k(x)$) instead of $\mathfrak{p}_k(x)$ (resp. $\mathfrak{e}_k(x)$). Like for finitely many variables, Newton's identities provide polynomial relations between $\mathfrak{p}_m$'s and $\mathfrak{e}_m$'s. A \emph{symmetric function} is a polynomial expression in the $\mathfrak{p}_m$'s, or equivalently in the $\mathfrak{e}_m$'s. If we have $k$ pairwise disjoint groups of countably many variables we can consider \emph{product symmetric functions}. These are polynomial expressions in $\mathfrak{p}_{m,(1)},\ldots,\mathfrak{p}_{m,(k)}$ for $1 \leq m \in \Z$, where $\mathfrak{p}_{m,(i)}$ denotes the $m$-th power symmetric function in the $i$-th group of variables. We write $\mathfrak{e}_{m,(i)}$ for the $m$-th elementary symmetric functions in the $i$-th group of variables. \\
To prove Theorem \ref{thm:undecidable} we require access to polynomials with domain $E_d^k$. For this purpose we introduced product symmetric functions.

\begin{theorem}\label{thm:undecidable}
 The following problem is undecidable.
\begin{itemize}
    \item[{\footnotesize Instance:}] A positive integer $k$ and a multihomogeneous product symmetric function $\mathfrak{f}$ in $k$ pairwise disjoint groups of variables.
    \item[{\footnotesize Question:}] Is $\mathfrak{f}$ nonnegative on the set $ \Delta^k$?
\end{itemize}
\end{theorem}

We follow (\cite[Section 5]{hatami2011undecidability}) and use their notation. Hatami and Norin's work concerns undecidability of determining the validity of linear inequalities in graph homomorphism densities for graphons. 
By adapting only very few parts of Hatami and Norin's proof we show that an undecidable problem can be embedded into the problem of deciding copositivity of product symmetric homogeneous functions in $\mathfrak{e}_1,\mathfrak{e}_2,\mathfrak{e}_3$. 

The outline of the proof is as follows. Hatami and Norin use that Matiyasevich’s solution to Hilbert’s
tenth problem implies that deciding nonnegativity of rational polynomials on $\N^k$ is undecidable which is equivalent to deciding nonnegativity of rational polynomials on $\{\frac{n-1}{n} : n \in \N\}$ (\cite[Lemma~5.1]{hatami2011undecidability}). 
They define the compact convex set $R \subset [0,1]^2$ with the countably many vertices $(1,1),(0,1)$ and $(\frac{n-1}{n},\frac{(n-1)(n-2)}{n^2})$ for $n \in \Z_{\geq 1}$. Hatami and Norin then verify that deciding nonnegativity of rational polynomials on $\N^k$ is equivalent to deciding nonnegativity of rational polynomials on $R^k$ which is again equivalent to nonnegativity of rational polynomials on $\{(\frac{n-1}{n},\frac{(n-1)(n-2)}{n^2}) : 1 \leq n \in \Z\}^k$. 

In the proof of Theorem \ref{thm:undecidable} we define $C:=\operatorname{conv}(2\mathfrak{e}_2,6\mathfrak{e}_3)(\Delta) \subset [0,1]^2 \subset R$. Then the lower parts of the boundaries of $C$ and $R$ are equal by Proposition \ref{cor:image infinite probability simplex}. As in \cite{hatami2011undecidability} we define a picewise linear function $L : [0,1] \to [0,1]$ which graph parameterizes the lower part of the boundary of $C$ and $R$, and a strictly convex function $g$ on $[0,1]$ which graph lies below the graph of $L$. Moreover, $g(x)=L(x)$ if and only if $x \in \{1,\frac{n-1}{n} : n \in \Z_{>0}\}$.
We then associate to rational polynomials $p(x_1,\ldots,x_k)$ product symmetric functions $\tau(p)$ in $\mathfrak{e}_{1,(i)},\mathfrak{e}_{2,(i)},\mathfrak{e}_{3,(i)}$ for $1 \leq i \leq k$. We observe that $p(x_1,\ldots,x_k)$ is nonnegative on $\{\frac{n-1}{n} : n \in \N\}$ if and only if $\tau(p)$ restricted to $\Delta^k$ is nonnegative on $C^k$. We refer to Figure \ref{figure:6} which depicts the subsets of $C$ and $R$ with vertices for $1 \leq k \leq 100$ and the function $g$.

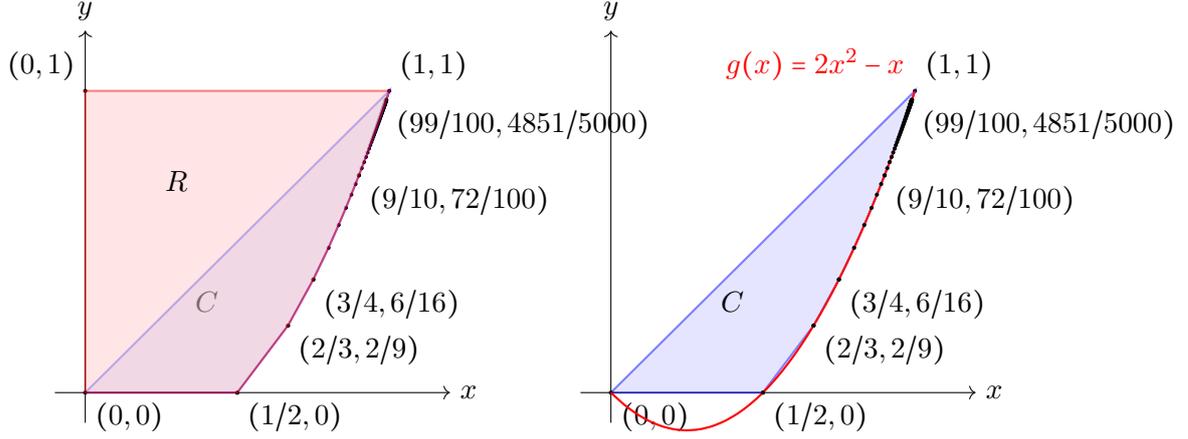
\begin{figure}[h!]
    \centering
    \begin{minipage}{0.45\textwidth}
        \centering
        \begin{tikzpicture}[scale=4]
            \draw[->] (-0.1,0) -- (1.2,0) node[right] {$x$};
            \draw[->] (0,-0.1) -- (0,1.2) node[above] {$y$};
            \foreach \k in {2,3,...,100} {
                \pgfmathsetmacro{\x}{(\k-1)/\k}
                \pgfmathsetmacro{\y}{(\k-1)*(\k-2)/\k^2}
                \filldraw (\x,\y) circle (0.005) node[above left] {};
            }
            \filldraw (1,1) circle (0.005) node[above right] {$(1,1)$};
            \draw[thick, blue, fill=blue!20, opacity=0.5] plot coordinates {(0.0, 0.0) (0.5, 0.0) (0.6667, 0.2222) (0.75, 0.375) (0.8, 0.48) (0.8333, 0.5556) (0.8571, 0.6122) (0.875, 0.6563) (0.8889, 0.6914)  (0.9, 0.72) (0.9091, 0.7438) (0.9167, 0.7639) (0.9231, 0.7811) (0.9286, 0.7959) (0.9333, 0.8089) (0.9375, 0.8203) (0.9412, 0.8304) (0.9444, 0.8395) (0.9474, 0.8476) (0.95, 0.855) (0.9524, 0.8617) (0.9545, 0.8678) (0.9565, 0.8733) (0.9583, 0.8785) (0.96, 0.8832) (0.9615, 0.8876) (0.963, 0.8916) (0.9643, 0.8954) (0.9655, 0.8989) (0.9667, 0.9022) (0.9677, 0.9053) (0.9688, 0.9082) (0.9697, 0.9109) (0.9706, 0.9135) (0.9714, 0.9159) (0.9722, 0.9182) (0.973, 0.9204) (0.9737, 0.9224) (0.9744, 0.9244) (0.975, 0.9263) (0.9756, 0.928) (0.9762, 0.9297) (0.9767, 0.9313) (0.9773, 0.9329) (0.9778, 0.9343) (0.9783, 0.9357) (0.9787, 0.9371) (0.9792, 0.9384) (0.9796, 0.9396) (0.98, 0.9408) (0.9804, 0.9419) (0.9808, 0.943) (0.9811, 0.9441) (0.9815, 0.9451) (0.9818, 0.9461) (0.9821, 0.9471) (0.9825, 0.948) (0.9828, 0.9489) (0.9831, 0.9497) (0.9833, 0.9506) (0.9836, 0.9514) (0.9839, 0.9521) (0.9841, 0.9529) (0.9844, 0.9536) (0.9846, 0.9543) (0.9848, 0.955) (0.9851, 0.9557) (0.9853, 0.9563) (1,1) (0.9855072463768116, 0.9569418189455997) (0.9857142857142858, 0.9575510204081633) (0.9859154929577465, 0.958143225550486) (0.9861111111111112, 0.9587191358024691) (0.9863013698630136, 0.959279414524301) (0.9864864864864865, 0.9598246895544192) (0.9866666666666667, 0.9603555555555555) (0.9868421052631579, 0.9608725761772853) (0.987012987012987, 0.9613762860516107) (0.9871794871794872, 0.9618671926364234) (0.9873417721518988, 0.9623457779202051) (0.9875, 0.9628125) (0.9876543209876543, 0.9632677945435147) (0.9878048780487805, 0.9637120761451516) (0.9879518072289156, 0.9641457395848454) (0.9880952380952381, 0.9645691609977324) (0.9882352941176471, 0.9649826989619377) (0.9883720930232558, 0.965386695511087) (0.9885057471264368, 0.9657814770775531) (0.9886363636363636, 0.9661673553719008) (0.9887640449438202, 0.9665446282035096) (0.9888888888888889, 0.9669135802469135) (0.989010989010989, 0.9672744837580003) (0.9891304347826086, 0.9676275992438563) (0.989247311827957, 0.9679731760897213) (0.9893617021276596, 0.96831145314622) (0.9894736842105263, 0.9686426592797784) (0.9895833333333334, 0.9689670138888888) (0.9896907216494846, 0.9692847273886704) (1,1) } -- cycle;    
            \node at (0.4, 0.3) {$C$};   
            \filldraw (0,1) circle (0.005) node[above left] {$(0,1)$};
            \filldraw (0,0) circle (0.005) node[below right] {$(0,0)$};
            \filldraw (1/2,0) circle (0.005) node[below right] {$(1/2,0)$};
            \filldraw (2/3,2/9)  circle (0.005) node[below right] {$(2/3,2/9) $};
            \filldraw (3/4,6/16)  circle (0.005) node[below right] {$(3/4,6/16)  $};
            \filldraw (9/10,72/100)   circle (0.005) node[below right] {$(9/10,72/100)   $};
            \filldraw (99/100,4851/5000)  circle (0.005) node[below right] {$(99/100,4851/5000)   $};  
            \draw[thick, red, fill=red!20, opacity=0.5] plot coordinates {(0.0, 0.0) (0.5, 0.0) (0.6667, 0.2222) (0.75, 0.375) (0.8, 0.48) (0.8333, 0.5556) (0.8571, 0.6122) (0.875, 0.6563) (0.8889, 0.6914)  (0.9, 0.72) (0.9091, 0.7438) (0.9167, 0.7639) (0.9231, 0.7811) (0.9286, 0.7959) (0.9333, 0.8089) (0.9375, 0.8203) (0.9412, 0.8304) (0.9444, 0.8395) (0.9474, 0.8476) (0.95, 0.855) (0.9524, 0.8617) (0.9545, 0.8678) (0.9565, 0.8733) (0.9583, 0.8785) (0.96, 0.8832) (0.9615, 0.8876) (0.963, 0.8916) (0.9643, 0.8954) (0.9655, 0.8989) (0.9667, 0.9022) (0.9677, 0.9053) (0.9688, 0.9082) (0.9697, 0.9109) (0.9706, 0.9135) (0.9714, 0.9159) (0.9722, 0.9182) (0.973, 0.9204) (0.9737, 0.9224) (0.9744, 0.9244) (0.975, 0.9263) (0.9756, 0.928) (0.9762, 0.9297) (0.9767, 0.9313) (0.9773, 0.9329) (0.9778, 0.9343) (0.9783, 0.9357) (0.9787, 0.9371) (0.9792, 0.9384) (0.9796, 0.9396) (0.98, 0.9408) (0.9804, 0.9419) (0.9808, 0.943) (0.9811, 0.9441) (0.9815, 0.9451) (0.9818, 0.9461) (0.9821, 0.9471) (0.9825, 0.948) (0.9828, 0.9489) (0.9831, 0.9497) (0.9833, 0.9506) (0.9836, 0.9514) (0.9839, 0.9521) (0.9841, 0.9529) (0.9844, 0.9536) (0.9846, 0.9543) (0.9848, 0.955) (0.9851, 0.9557) (0.9853, 0.9563) (0.9855072463768116, 0.9569418189455997) (0.9857142857142858, 0.9575510204081633) (0.9859154929577465, 0.958143225550486) (0.9861111111111112, 0.9587191358024691) (0.9863013698630136, 0.959279414524301) (0.9864864864864865, 0.9598246895544192) (0.9866666666666667, 0.9603555555555555) (0.9868421052631579, 0.9608725761772853) (0.987012987012987, 0.9613762860516107) (0.9871794871794872, 0.9618671926364234) (0.9873417721518988, 0.9623457779202051) (0.9875, 0.9628125) (0.9876543209876543, 0.9632677945435147) (0.9878048780487805, 0.9637120761451516) (0.9879518072289156, 0.9641457395848454) (0.9880952380952381, 0.9645691609977324) (0.9882352941176471, 0.9649826989619377) (0.9883720930232558, 0.965386695511087) (0.9885057471264368, 0.9657814770775531) (0.9886363636363636, 0.9661673553719008) (0.9887640449438202, 0.9665446282035096) (0.9888888888888889, 0.9669135802469135) (0.989010989010989, 0.9672744837580003) (0.9891304347826086, 0.9676275992438563) (0.989247311827957, 0.9679731760897213) (0.9893617021276596, 0.96831145314622) (0.9894736842105263, 0.9686426592797784) (0.9895833333333334, 0.9689670138888888) (0.9896907216494846, 0.9692847273886704) (1,1) (0,1)} -- cycle;    
            \node at (0.3, 0.7) {$R$};
        \end{tikzpicture}
    \end{minipage}%
    \hspace{0.5cm}%
    \begin{minipage}{0.45\textwidth}
        \centering
        \begin{tikzpicture}[scale=4]
            \draw[->] (-0.1,0) -- (1.2,0) node[right] {$x$};
            \draw[->] (0,-0.1) -- (0,1.2) node[above] {$y$};
            \filldraw (1,1) circle (0.005) node[above right] {$(1,1)$};
            \draw[thick, blue, fill=blue!20, opacity=0.5] plot coordinates {(0,0) (0.5,0) (2/3,2/9) (3/4,6/16) (4/5,12/25) (5/6,20/36) (6/7,30/49) (7/8,42/64) (8/9,56/81) (9/10,72/100) (10/11,90/121) (11/12,110/144) (12/13,132/169) (13/14,156/196) (14/15,182/225) (15/16,210/256) (16/17,240/289) (17/18,272/324) (18/19,306/361) (19/20,342/400) (99/100,4851/5000) (1,1) } -- cycle;   
            \node at (0.4, 0.3) {$C$};    
            \filldraw (0,0) circle (0.005) node[below right] {$(0,0)$};
            \filldraw (1/2,0) circle (0.005) node[below right] {$(1/2,0)$};
            \filldraw (2/3,2/9)  circle (0.005) node[below right] {$(2/3,2/9) $};
            \filldraw (3/4,6/16)  circle (0.005) node[below right] {$(3/4,6/16)  $};
            \filldraw (9/10,72/100)   circle (0.005) node[below right] {$(9/10,72/100)   $};
            \filldraw (99/100,4851/5000)  circle (0.005) node[below right] {$(99/100,4851/5000)   $};
            \draw[thick, red, domain=0:1, smooth] plot (\x, {2*(\x)*(\x) - \x}) node[above left] {$g(x) = 2x^2 - x$};
            \foreach \k in {2,3,...,100} {
                \pgfmathsetmacro{\x}{(\k-1)/\k}
                \pgfmathsetmacro{\y}{(\k-1)*(\k-2)/\k^2}
                \filldraw (\x,\y) circle (0.005) node[above left] {};
            }
        \end{tikzpicture}
    \end{minipage}
    \caption{Visualizations for $1 \leq k \leq 100$}
    \label{figure:6}
\end{figure}

\begin{proof}[Proof of Theorem \ref{thm:undecidable}]
By \cite[Lemma~5.1]{hatami2011undecidability} it follows from Matiyasevich’s solution to Hilbert’s tenth problem that the following validity problem is undecidable.
\begin{itemize}
    \item[{\footnotesize Instance:}] A positive integer $k$ and a polynomial $p \in \Z[Y_1,\ldots,Y_k]$.
    \item[{\footnotesize Question:}] Do there exist $x_1,\ldots,x_k \in \left\{ \frac{n-1}{n} : n \in \N\right\}$ with $p(x_1,\ldots,x_k) <0$?
\end{itemize}
When we replace $(\mathfrak{e}_2,\mathfrak{e}_3)$ by $(2\mathfrak{e}_2,6\mathfrak{e}_3)$ in Proposition \ref{cor:image infinite probability simplex} we have 
\[C :=  \conv (2\mathfrak{e}_2,6\mathfrak{e}_3)(\Delta) = \conv \left\{(1,1),\left( \frac{n-1}{n}, \frac{(n-1)(n-2)}{n^2}\right) : n \in \Z, n \geq 1\right\} \] as $(2\mathfrak{e}_2,6\mathfrak{e}_6)(\frac{1}{n},\ldots,\frac{1}{n},0,\ldots,0)=\left( \frac{n-1}{n}, \frac{(n-1)(n-2)}{n^2}\right)$ if $\frac{1}{n}$ appears $n$ times in $(\frac{1}{n},\ldots,\frac{1}{n},0,\ldots,0)$.

We define a piecewise linear function $L$ on $[0,1]$ by $$
L(x):= \frac{3t_x^2-t_x-2}{t_x(t_x+1)}x-\frac{2(t_x-1)}{t_x+1}$$ where $t_x$ is a positive integer for which $x \in [1-1/{t_x},1-1/({t_x}+1)]$ and $L(1)=1$. 
For a positive integer $t$ we have $L(\frac{t-1}{t})=\frac{(t-1)(t-2)}{t^2}=6e_3(0,1/t,\ldots,1/t)$ and $L(\frac{t}{t+1})=\frac{t(t-1)}{(t+1)^2}=6e_3(1/{(t+1)},\ldots,1/{(t+1)})$ where the vectors are supposed to lie in $\Delta_{t}$. Moreover, we have $\lim_{t \to \infty}(\frac{3t^2-t-2}{t(t+1)}x-\frac{2(t-1)}{t+1})=3x-2$. Thus, the map $L$ is well defined and continuous.   
Since $L$ is picewise linear and $L(\frac{n-1}{n})=\frac{(n-1)(n-2)}{n^2}$ the lower part of the boundary of $C $ equals $\{(x,y) :  x \in [0,1], L(x)=y\} $. 
We define $R := \{ (x,y) \in [0,1]^2 : y \geq L(x)\} $ and have nested sets $(2\mathfrak{e}_2,6\mathfrak{e}_3)(\Delta) \subset C \subset R$.
The function $g(x):=2x^2-x$ is strictly convex and for any positive integer $t$ we have $g(\frac{t-1}{t})=\frac{(t-1)(t-2)}{t^2}=L((t-1)/t)$. This shows $L(x)-g(x) \geq 0$ on the interval $[0,1)$. The symmetric function $6\mathfrak{e}_3-g(2\mathfrak{e}_2)$ is zero on $\Delta$ if and only if $2\mathfrak{e}_2 \in \{1, (n-1)/n : n \in \Z, n \geq 1\}$. This follows, since the boundaries of $\mathcal{E}_3$ and $E_3$ only intersect on the vertices of $E_3$. 

Let $p \in \R[Y_1,\ldots,Y_k]$ be a polynomial and let $M$ be the sum of the absolute values of its coefficients multiplied by $100 \deg (p)$. We consider the auxiliary polynomial 
$$ q(Y_1,\ldots,Y_k,Z_1,\ldots,Z_k):= p(Y_1,\ldots,Y_k)\cdot \prod_{i=1}^k (1-Y_i)^6+M \left( \sum_{i=1}^kZ_i - g(Y_i) \right).$$ Then, by (\cite[Lemma 5.4]{hatami2011undecidability}) the following are equivalent: 
\begin{enumerate}
    \item[{(i)}] $q(x_1,\ldots,x_k,y_1,\ldots,y_k) < 0$ for some $x_1,\ldots,x_k,y_1,\ldots,y_k$ with $(x_i,y_i) \in R$ for all $1 \leq i \leq k$;
    \item[{(ii)}] $p(x_1,\ldots,x_k) < 0$ for some $x_1,\ldots,x_k \in \{  \frac{n-1}{n} : n \in \Z, n \geq 1\}$.
\end{enumerate}
Let $\mathcal{R}_k$ denote the set of product symmetric functions in $k$ pairwise disjoint groups of variables. We consider the map 
$$ \abb{\tau}{\R[Y_1,\ldots,Y_k,Z_1,\ldots,Z_k]}{\mathcal{R}_k}{f(Y_1,\ldots,Y_k,Z_1,\ldots,Z_k)}{\prod_{i=1}^k \mathfrak{e}_{1,(i)}^{3\deg f} \cdot f\left( \frac{\mathfrak{e}_{2,(1)}}{\mathfrak{e}_{1,(1)}^2},\ldots,\frac{\mathfrak{e}_{2,(k)}}{\mathfrak{e}_{1,(k)}^2},\frac{\mathfrak{e}_{3,(1)}}{\mathfrak{e}_{1,(1)}^3},\ldots,\frac{\mathfrak{e}_{3,(k)}}{\mathfrak{e}_{1,(k)}^3}\right)}.$$
Since ${\mathfrak{e}_{2,(i)}}$ and ${\mathfrak{e}_{1,(i)}^2}$ (resp. $\mathfrak{e}_{3,(i)}$ and ${\mathfrak{e}_{1,(i)}^3}$) have degree $2$ (resp. $3$) multiplying by $\mathfrak{e}_{1,(i)}^{3\deg f}$ ensures that $\tau (f)$ has always nonnegative exponent in  $\mathfrak{e}_{1,(i)}$ for all $1 \leq i \leq k$. Thus $\tau (f)$ is a product symmetric function in $\mathfrak{e}_{1,(i)},\mathfrak{e}_{2,(i)},\mathfrak{e}_{3,(i)}$ for $1 \leq i \leq k$. By construction, the product symmetric function $\tau(f)$ is multihomogeneous. \smallskip

Analogously to (\cite{hatami2011undecidability}, Claims 5.7 \& 5.8) we claim that the following assertions are equivalent 
\begin{itemize}
    \item[(a)] $q(x_1,\ldots,x_k,y_1,\ldots,y_k) < 0$ for some $x_1,\ldots,x_k,y_1,\ldots,y_k$ with $(x_i,y_i) \in R$ for all $1 \leq i \leq k$;
    \item[(b)] $\tau (q)$ attains a negative value on $\Delta^k$.
\end{itemize}
First, we suppose (a). Hatami and Norin show in the proof of (\cite[Lemma~5.4]{hatami2011undecidability}) that if \\
$q(x_1,\ldots,x_k,y_1,\ldots,y_k) < 0$ for some $x_1,\ldots,x_k,y_1,\ldots,y_k$ with $(x_i,y_i) \in R$ for all $1 \leq i \leq k$ then the $x_i$'s can be chosen as $x_1,\ldots,x_k \in \{  \frac{n-1}{n} : n \in \N\}$, and $y_i = L(x_i)$. Thus, $\tau(q)$ is negative on $\Delta^k$ by Proposition \ref{cor:image infinite probability simplex}. More precisely, $\mathfrak{e}_{1,(i)}=1, 2\mathfrak{e}_{2,(i)}=x_i$ and $6\mathfrak{e}_{3,(i)}=y_i$ for all $1 \leq i \leq k$ is feasible and thus $\tau(q)$ attains a negative value on $\Delta$. \\
Second, we suppose $q(x_1,\ldots,x_k,y_1,\ldots,y_k) \geq 0$ for all $x_i,y_i$ with $(x_i,y_i) \in R$ for all $1 \leq i \leq k$, then $\tau(q)$ is nonnegative on $\Delta^k$, since $((2\mathfrak{e}_2,6\mathfrak{e}_3)(\Delta))^k \subset C^k \subset R^k$.  \smallskip

So the assertions (ii) and (b) are equivalent. Given an instance of the undecidable problem (\cite[Lemma 5.1]{hatami2011undecidability}), i.e. a polynomial $p \in \Z[Y_1,\ldots,Y_k]$, consider the associated rational auxiliary polynomial $q(Y_1,\ldots,Y_k,Z_1,\ldots,Z_k)$ as above. Then by (\cite[Lemma 5.4]{hatami2011undecidability}) the polynomial $p$ is nonnegative on $\{\frac{n-1}{n} : n\in \Z, n \geq 1\}^k$ if and only if $q$ is nonnegative on $R^k$. To establish the latter is equivalent to testing nonnegativity of the multihomogeneous product symmetric polynomial $\tau (q)$ on $\Delta^k$. This proves the Theorem. 
\end{proof}

We are ready to prove the main theorem on deciding nonnegativity of trace polynomials.

\begin{proof}[Proof of Theorem \ref{thm:undecidable traces}]
    We consider the subproblem of determining validity of nonnegativity of multihomogeneous trace polynomials $f(X_1,\ldots,X_k)$ in which any formal trace symbol is in an even square of a variable up to degree $6$, i.e. $f$ is a polynomial expression in $\Tr (X_i^{2}),\Tr (X_i^4), \Tr (X_i^6)$ for $1 \leq i \leq k$. Then deciding nonnegativity of $f$ for all symmetric matrices $M_1,\ldots,M_k$ of all sizes is equivalent to deciding nonnegativity of a multihomogeneous product symmetric function $g(\mathfrak{p}_{2,(1)},\mathfrak{p}_{4,(1)},\mathfrak{p}_{6,(1)},\ldots,\mathfrak{p}_{2,(k)},\mathfrak{p}_{4,(k)},\mathfrak{p}_{6,(k)}).$ Its nonnegativity is equivalent to nonnegativity of $g (\mathfrak{p}_{1,(1)},\mathfrak{p}_{2,(1)},\mathfrak{p}_{3,(1)},\ldots,\mathfrak{p}_{1,(k)},\mathfrak{p}_{2,(k)},\mathfrak{p}_{3,(k)})$ on $\Delta^k$. However, this problem is undecidable by Theorem \ref{thm:undecidable} since Newton's identities provide a linear relation between the power sums and elementary symmetric polynomials up to degree $3$ on the probability simplex. 
\end{proof}

\begin{remark}
It was shown by Jones \cite{jones1982universal} that deciding whether Diophantine equations in $9$ variables have solutions in $\N^9$ is already undecidable. Thus, Hatami-Norin's work shows that it is enough to consider trace polynomials in $9$ variables for undecidability. Moreover, in the proof of Theorem \ref{thm:undecidable} we could have directly worked with power sum functions, $\Pi_3$ and $\conv \Pi_3$, and we refer to \cite{blekherman2022undecidability} for details.     
\end{remark}

\subsection{Proof of Theorem \ref{thm:normalized decidable}}
The small adjustment of using normalized traces makes the problem of establishing nonnegativity decidable. An important role is played by Timofte's half degree principle. The decidability was implicitly observed by Blekherman and Riener in \cite{blekherman2021symmetric}. 

\begin{theorem}[\cite{timofte2003positivity}] \label{thm:half degree princ}
A symmetric $n$-variate polynomial $f \in \R[x]$ is nonnegative if and only if $f(a) \geq 0$ for any $a \in \R^n$ with $\# \{a_1,\ldots,a_n\} \leq \max \{ \lfloor \frac{\deg f}{2} \rfloor ,2\}$.    
\end{theorem}

We briefly illustrate the subtle difference between normalized power sums and power sums based on Timofte's half degree principle. 
Suppose we are given a power sum $p_d=x_1^d+\ldots+x_n^d$ in $n$ variables of degree $d \geq 4$ and let $\kappa := \lfloor \frac{d}{2} \rfloor$. Testing nonnegativity of $p_d$ is equivalent to testing nonnegativity of the $\kappa$-variate polynomials $p_{d,\alpha} = \alpha_1x_1^d+\ldots+\alpha_{\kappa}x_{\kappa}^d$ for all integer coefficient sequences $\alpha \in \N^{\kappa}$ with $\sum_{i=1}^{\kappa}\alpha_i = n$ by Theorem \ref{thm:half degree princ}. If we divide $p_{d,\alpha}$ by the number of variables $n$ of $p_d$ we observe that the coefficient vector $\frac{\alpha}{n}=(\frac{\alpha_1}{n},\ldots,\frac{\alpha_\kappa}{n}) $ is contained in the $(\kappa -1)$-dimensional probability simplex $\Delta_{\kappa-1}$. Thus, testing nonnegativity of the power mean $\frac{p_d}{n}$ in any number of variables $n \geq d$ is equivalent to testing nonnegativity of $\frac{p_{d,\alpha}}{n}= \frac{\alpha_1}{n}x_1^d+\ldots+\frac{\alpha_{\kappa}}{n}x_{\kappa}^d$ for all $n \geq d$ and all $\alpha \in \N^\kappa$ such that $\sum_i\alpha_i = n$. Then nonnegativity of the power mean $\frac{p_d}{n}$ for all $n$ is equivalent to nonnegativity of $\beta_1x_1^d+\ldots+\beta_{\kappa}x_{\kappa}^d$ for all $(\beta_1,\ldots,\beta_\frac{d}{2}) \in \Delta_{\kappa-1} \times \R^\kappa$ due to the density of $\mathbb{Q}$ in $\R$. However, the set $\Delta_{\kappa-1} \times \R^\kappa$ is semialgebraic.

\begin{definition}
    Let $\mathfrak{f}=\sum_{\lambda \in \N^{2d}} c_\lambda \mathfrak{p}_{\lambda_1}\cdots \mathfrak{p}_{\lambda_{2d}}$ be a symmetric function of degree $2d$. We denote by $\mathfrak{f}_n = \sum_{\lambda \in \N^{2d}}c_\lambda \frac{p_{\lambda_1}p_{\lambda_2}\cdots p_{\lambda_{2d}}}{n^{\mid\{i ~:~ \lambda_i > 0\}\mid }} $ the symmetric polynomial in $n$ variables that we obtain from $\mathfrak{f}$ by replacing any power sum function $\mathfrak{p}_{\lambda_i}$ in $\mathfrak{f}$ by the scaled power sum polynomial $\frac{1}{n}p_{\lambda_i}$ in $n$ variables.
    We call each $\mathfrak{f}_n$ a \emph{power mean} polynomial.
    We define an associated $2d$-variate function $\Phi_\mathfrak{f}$ as 
    $$ \Phi_\mathfrak{f}(s,t) =  \sum_\lambda c_\lambda \prod_{i=1}^{l} (s_1t_1^{\lambda_i}+\ldots + s_d t_d^{\lambda_i}). $$
\end{definition}

The following Lemma generalizes the application of Timofte's half degree principle from the discussion above to arbitrary normalized symmetric polynomials. 
\begin{lemma}[\cite{blekherman2021symmetric} Theorem~3.4] \label{lem:gregcordian}
 Let $\mathfrak{f}=\sum_{\lambda \in \N^{2d}} c_\lambda \mathfrak{p}_{\lambda_1}\cdots \mathfrak{p}_{\lambda_{2d}}$ be a symmetric function of degree $2d$. Then the $n$-variate symmetric polynomials $\mathfrak{f}_n$ are nonnegative for all $n \in \N$ if and only if $\Phi_{\mathfrak{f}}$ is nonnegative on $\Delta_{d-1} \times \R^d$.
\end{lemma}

We are ready to prove Theorem \ref{thm:normalized decidable}.
\begin{proof}[Proof of Theorem \ref{thm:normalized decidable}]
For a real symmetric matrix $M \in \R^{n \times n}$ with eigenvalues $\lambda_1,\ldots,\lambda_n$ we have $$\widetilde{\operatorname{tr} }\left(M^k\right) = \frac{1}{n} \Tr(M^k) = \frac{1}{n} \sum_{i=1}^n \lambda_i^k = \frac{1}{n}p_k(\lambda). $$ 
Thus verifying nonnegativity of a univariate normalized trace polynomial is equivalent to verifying nonnegativity of the associated sequence of power mean polynomials in any number of variables. By Lemma \ref{lem:gregcordian} this is equivalent to nonnegativity of a polynomial on the semialgebraic set $\Delta_{d-1}\times \R^d$ and thus decidable. \\
For a multivariate normalized trace polynomial we proceed analogously. We have that nonnegativity of a normalized trace polynomial in $k$ variables is equivalent to nonnegativity of an associated polynomial on the semialgebraic set $(\Delta_{d_1-1} \times \R^{d_1})\times \ldots \times (\Delta_{d_k-1} \times \R^{d_k})$ where $d_i$ is the maximal occurring exponent of the $i$-th variable.
\end{proof}

\section{Conclusion and open questions}
In this article, we have studied the wonderful geometry of the Vandermonde map in the finite and infinite setup. In particular, we have shown how a connection to trace polynomials allows to show that the problem of determining if a given multivariate trace polynomial is  nonnegative is undecidable. Our proof inspired by Hatami-Norin's proof \cite{hatami2011undecidability} relied on Matiyasevich work on  Hilbert's tenth problem \cite{matiyasevich1970diophantineness} which showed that it is not possible to computationally decide if a Diophantine equation in several variables has an integer solution. In this context it is worth noticing that asserting that a given univariate polynomial has a root in the integers is a decidable task. Our construction used to prove Theorem \ref{thm:undecidable traces} does not apply if we restrict to univariate trace polynomials and therefore it remains a natural question whether verification of nonnegativity of univariate trace polynomials is decidable. 

\printbibliography
\end{document}